\documentclass[10pt]{amsart}
\usepackage{amsmath}
\usepackage[shortlabels]{enumitem}
\usepackage{graphicx}
\usepackage{amssymb,amsthm}
\usepackage{hyperref}
\usepackage{esint}
\usepackage{epstopdf}
\usepackage{color}
\usepackage{url}
\usepackage[top=.9in, left=.9in, right=.9in, bottom=.8in]{geometry}

\numberwithin{equation}{section}

\newtheorem{theorem}{Theorem}[section]
\newtheorem{lemma}[theorem]{Lemma}
\newtheorem{proposition}[theorem]{Proposition}
\newtheorem{corollary}[theorem]{Corollary}

\newtheorem{remark}[theorem]{Remark}
\newtheorem{definition}[theorem]{Definition}

\newtheorem{example}[theorem]{Example}

\def\R{{\mathbb R}}

\def\cH{{\mathcal H}}

\def\a{\alpha}
\def\b{\beta}
\def\e{\varepsilon}
\def\d{\delta}

\def\k{\kappa}
\def\l{\lambda}

\def\n{\nabla}
\def\p{\partial}
\def\r{\rho}

\def\w{\omega}
\def\W{\Omega}
\def\g{\gamma}
\def\z{\zeta}

\def\1{\left(}
\def\2{\right)}
\def\3{\left\{}
\def\4{\right\}}
\def\8{\infty}
\def\sm{\setminus}
\def\ss{\subseteq}
\def\cc{\subset\subset}

\newcommand{\mres}{\mathbin{\vrule height 1.6ex depth 0pt width
0.13ex\vrule height 0.13ex depth 0pt width 1.3ex}}

\DeclareMathOperator*{\dvg}{div}

\DeclareMathOperator*{\diam}{diam}
\DeclareMathOperator*{\supp}{supp}
\DeclareMathOperator*{\osc}{osc}

\DeclareMathOperator*{\sign}{sign}

\usepackage[scr=boondox]{mathalfa} 

\title[Linear Stability Implies Nonlinear Stability for Faber-Krahn]{Linear Stability Implies Nonlinear Stability\\ for Faber-Krahn Type Inequalities}

\author{Mark Allen}
\address[Mark Allen]{Department of Mathematics, Brigham Young University, Provo,  UT}
\email{allen@mathematics.byu.edu}

\author{Dennis Kriventsov}
\address[Dennis Kriventsov]{Department of Mathematics, Rutgers University,  Piscataway, NJ}
\email{dnk34@math.rutgers.edu}

\author{Robin Neumayer}
\address[Robin Neumayer]{Department of Mathematical Sciences, Carnegie Mellon University, Pittsburgh, PA}
\email{neumayer@cmu.edu}

\begin{document}

\def\Lap{\Delta}
\def\grad{\nabla}

\def\En{\mathcal{E}}

\def\tor{\mathrm{tor}}
\def\ei{\l_1}
\def\vpar{\eta}
\def\tpar{\mathfrak{T}}
\def\fv{f_{v, \vpar}}
\def\vmax{v_{\text{max}}}

\def\err{\tau}
\def\Ep{\mathcal{F}_{\err}}

\begin{abstract}
	For a domain $\W \ss \R^n$ and a small number $\tpar > 0$, let
	\[
		\En_0(\W) = \ei(\W) + \tpar \tor(\W) = \inf_{u, w \in H^1_0(\W)\sm \{0\}} \frac{\int |\grad u|^2}{\int u^2} + \tpar \int \frac{1}{2} |\grad w|^2 - w
	\]
	be a modification of the first Dirichlet eigenvalue of $\W$. It is well-known that over all $\W$ with a given volume, the only sets attaining the infimum of $\En_0$ are balls $B_R$; this is the Faber-Krahn inequality. The main result of this paper is that, if for all $\W$ with the same volume and barycenter as $B_R$ and whose boundaries are parametrized as small $C^2$ normal graphs over $\p B_R$ with bounded $C^2$ norm,
	\[
		\int |u_\W - u_{B_R}|^2 + |\W \triangle B_R|^2 \leq C [\En_0(\W) - \En_0(B_R)]
	\]
	(i.e. the Faber-Krahn inequality is \emph{linearly stable}), then the same is true for \emph{any} $\W$ with the same volume and barycenter as $B_R$ without any smoothness assumptions (i.e. it is \emph{nonlinearly stable}). Here $u_\W$ stands for an $L^2$-normalized first Dirichlet eigenfunction of $\W$. Related results are shown for Riemannian manifolds. The proof is based on a detailed analysis of some critical perturbations of Bernoulli-type free boundary problems. The topic of when linear stability is valid, as well as some applications, are considered in a companion paper.
\end{abstract}

\maketitle

\section{Introduction}

For a domain $\W \ss \R^n$, let
\[
	\ei(\W) = \inf_{u \in H^1_0(\W)} \frac{\int |\grad u|^2}{\int u^2}
\]
be the first Dirichlet eigenvalue, and let $u_\W$ be a nonnegative function with $\int u_\W^2 = 1$ attaining the infimum. The Faber-Krahn inequality asserts that
\begin{equation}\label{e:introfk}
	\ei(\W) \geq \ei(B_R),
\end{equation}
where $B_R$ is a ball with $|B_R| = |\W|$, and moreover this inequality is strict unless $\W$ is
equal to a translation of $B_R$, up to a set of zero capacity.
In recent years, the topic of \emph{stability} for this inequality has received considerable attention: if $\ei(\W) - \ei(B_R)$ is small, how closely must $\W$ resemble $B_R$? For a survey of the recent literature on this question, see \cite{BD17}. In particular, \emph{quantitative} stability refers to inequalities of the form
\begin{equation} \label{e:introqs}
	d^\a(B_R(x_\W), \W) \leq C [\ei(\W) - \ei(B_R)],
\end{equation}
where $d$ is some notion of distance between sets, $\a > 0$ is a positive number, and $x_\W = \fint_\W x$ is the barycenter. Here it is desirable to find the strongest $d$, as well as the smallest number $\a$, for which such an inequality may hold.

The best result of this type in the literature is due to Brasco, De Philippis, and Velichkov \cite{BDV15}, which establishes \eqref{e:introqs} with 
\begin{equation}\label{e: d0}
	d = d_0(E, E') = |E \triangle E'|
\end{equation}
and $\a = 2$. This theorem is sharp, in the sense that \eqref{e:introqs} is false with $d = d_0$ and $\a < 2$. However, it does not fully resolve the issue of quantitative stability, as it seems unlikely that $d_0$ is the ``strongest'' distance for which \eqref{e:introqs} holds. To understand why, it is useful to separate the stability question into two ``regimes'' where it may be studied: we say that \eqref{e:introfk} is \emph{linearly stable} with respect to $d, \a$ if \eqref{e:introqs} holds for any $\W$ whose boundary may be expressed as a $C^2$ normal graph over $\p B_R$, and we say \eqref{e:introfk} is \emph{nonlinearly stable} with respect to $d, \a$ if \eqref{e:introqs} holds for any $\W$, unconditionally. For example, the result of \cite{BDV15} is a nonlinear stability result.

It is not difficult to see that $d_0$ is not the strongest norm for which \emph{linear} stability is valid. Indeed, if $\p \W$ is given by a normal graph $g$, then $d_0$ is comparable to the $L^2$ norm of $\xi$, 
while, as shown in \cite{BDV15}, the right-hand side of \eqref{e:introqs} controls a term of the form $\int_{\p B_R} \xi L \xi$, where $L$ is a first-order differential operator.
 In particular, Sobolev norms of $\xi$ up to $\|\xi\|_{H^{1/2}(\p \W)}$ appear to be controllable. Unfortunately, this kind of observation does not easily lead to improved nonlinear stability statements, as it is unclear how to define a stronger Sobolev-type distance for arbitrary sets $\W$.

In the nonlinear context, $d_0$ has a different sort of drawback, which is discussed in \cite[Open Problem 3]{BD17}: it is possible to modify $\W$ by removing a set of measure $0$ in a way which strictly increases its eigenvalue (consider, for example, $B_1$ vs. $B_1 \sm \{t e_1 : t > 0 \}$, a slit domain in $\R^2$). Any change to $\W$ which modifies the local capacity of its compliment should be reflected in $\ei(\W)$. On the other hand, $d_0$ is clearly insensitive to any such changes.

We propose a distance
\begin{equation*}
	d_1(E, E') = \sqrt{\int |u_E - u_{E'}|^2}
\end{equation*}
where $u_E, u_{E'}$ are the normalized first eigenfunctions as above extended by zero to be defined on all of $\R^n$, partially in an attempt to formulate nonlinear stability statements for \eqref{e:introfk} which address these issues. This has advantages and disadvantages in comparison to $d_0$. In the nonlinear context, it is sensitive to capacity-zero perturbations of $\W$. 
It also turns out to be specifically relevant to some applications to monotonicity formulas in free boundary theory, which we discuss in \cite{QACF}. On the other hand, $d_1$ is a kind of indirect measure of set difference and may be difficult to use in practice; we also do not know whether $d_1(E, B_R)$ controls $d_0(E, B_R)$ for all sets.

The results of this paper, in conjunction with our companion paper \cite{QACF}, establish sharp quantitative stability for the Faber-Krahn inequality with respect to the distance $d_1$. The result is proven on the round sphere and hyperbolic space as well as in Euclidean space: we prove in \cite[Theorem 1.1]{QACF} that \eqref{e:introqs} holds with $d =d_1 +d_2$ and $\alpha =2$ on simply connected space forms. The proof is based on a {\it selection principle}. This scheme, introduced by Cicalese and Leonardi in \cite{CL12} in the context of the isoperimetric inequality,  has been implemented in various settings in recent years (including \cite{BDV15})  to establish  quantitive forms of geometric inequalities. A selection principle has two main steps:
\begin{enumerate}
	\item The reduction step to show that linear stability implies nonlinear stability, based on regularity theory for a penalized functional.
	\item The local step to prove that linear stability holds. 
\end{enumerate}
This paper is dedicated to carrying out  Step (1): if \eqref{e:introfk} is linearly stable with $d = d_1$ and $\a = 2$, then it is also nonlinearly stable. Step (2) on simply connected space forms is established in \cite{QACF}. Let us describe how the classical selection principle argument for Step (1) might go in this context, and where the significant challenges arise in this setting. Suppose by way of contradiction that  \eqref{e:introqs} fails.  We may thus find a sequence of sets $\W_k$ with $|\W_k| = |B_R|$ and $\ei(\W_k) - \ei(B_R)  \to 0$, while 
\begin{equation}\label{e: contra hp intro}
	c_k := d^2_1(B_R(x_{\W_k}), \W_k) \geq k [\ei(\W_k) - \ei(B_R)]
\end{equation}
The goal, then, is to replace each $\W_k$ with a set $U_k$ that also satisfies \eqref{e: contra hp intro}, but such that $\p U_k$  
is
a smooth normal graph over $\p B_R$, in this way contradicting the linear stability assumption. Such a set is obtained by choosing $U_k$ to be a minimizer of a functional that roughly takes the form
\[
	\Ep^0(U) = \ei(U) + \err \sqrt{c_k^2 + (c_k- d_1(U,B_R(x_U))^2)^2 }
\]
for a small parameter $\err > 0$ and establishing regularity estimates for such minimizers. The second term in this functional forces $d_1(U_k, B_R(x_{U_k}))$ to be close to $d_1(\W_k, B_R(x_{\W_k}))$ so that \eqref{e: contra hp intro} will be satisfied by $U_k$, while the first term is hopefully regularizing the boundary of $U_k$. 

The Euler-Lagrange equation for minimizers of $\Ep^0$ leads to a free boundary condition along $\p U_k$ of Bernoulli type, similar to the one studied in \cite{AC81}.
 When working with $d_0$ or other weaker distances, the second term in $\Ep^0$ is of lower order, and can easily be seen to be a minor perturbation to the functional. In particular, the regularity theory of Alt and Caffarelli \cite{AC81} applies more or less directly (at least in conjunction with more recent literature, such as \cite{DeSilva11, DT15}). 
 In the case of $d_1$, however, this is no longer a sufficient heuristic: the second term of $\Ep^0$ is now a \emph{critical}, or same-order, perturbation of $\ei(U)$. The key point to establish regularity becomes controlling the distance $d_1(U', U)$ between a minimizing set $U$ and various competitors $U'$. This amounts to needing estimates like
\begin{equation} \label{e:intronoest}
	\int |u_U - u_{U'}| \leq C |\ei(U) - \ei(U')|,
\end{equation}
at least for the competitors $U'$ necessary to obtain estimates on $U$. These competitors were essentially laid out in \cite{AC81}, and closely follow minimal surface theory; they consist of $U \cup B_r(x)$, $U \sm B_r(x)$, and $\phi_t(U)$ for smooth diffeomorphisms $\phi_t$.


We do not know if \eqref{e:intronoest} is valid for these competitors, and believe this is an interesting problem. This means that we are unable establish the necessary regularity theory for $\Ep^0$ to  carry out Step (1) of the selection principle as described above. In order to overcome this challenge, we \emph{modify the functional}: let
\begin{equation}\label{eqn: intro  functuonal}
	\Ep^*(U) = \ei(\W) + \tpar \tor(\W) + \err \sqrt{c_k^2 +  (c_k - d_1(U,B_R(x_U))^2)^2 },
\end{equation} 
where $\tpar>0$ is a small parameter and
\[
	\tor(\W) = \int \frac{1}{2} |\grad w|^2 - w
\]
is the torsional rigidity of $\W$. The torsional rigidity has a long history of appearing when considering spectral optimization problems (see \cite{Bucur12, BDV15}), and shares with $\ei$ the ball as the unique volume-constrained minimizer. The functional \eqref{eqn: intro  functuonal} corresponds to a {\it vectorial } free boundary problem. The crucial benefit of introducing this torsion term is that we {\it can} establish the analogue of the key estimate \eqref{e:intronoest} corresponding to this functional:
\begin{equation}\label{e:introkey}
	\int |u_U - u_{U'}| \leq C \1|\ei(U) - \ei(U')| + |\tor(U) - \tor(U')|\2
\end{equation}
for all the relevant competitors $U'$. This is shown in Proposition~\ref{lem:key} and is the starting point toward establishing an existence and regularity theory for $\Ep^*$ and its generalizations, which constitutes the  core analysis of this paper.
\begin{theorem}\label{t:intro}
	There exists  $\tpar_0 > 0$ such that for each $\tpar < \tpar_0$, there exists $\err_0(\tpar) > 0$ such that if $\err < \err_0$:
	\begin{enumerate}
		\item There exists a minimizer $U$ of $\Ep^*$ over the class of all open sets $U$ with $|U| = |B_R|$.
		\item This $U$ has $C^{2, \a}$ boundary for any $\a < 1$, and $\p U$ may be expressed as a normal graph $\xi$ over $\p B_R$ with $\|\xi\|_{C^2} = o_\err(1)$ (i.e. for every $\e > 0$, there is a $\err(\e)$ such that if $\err < \err(\e)$, then $\|\xi\|_{C^{2, \a}} \leq \e$).
	\end{enumerate} 
\end{theorem}
Beyond establishing the key estimate \eqref{e:introkey}, further difficulties arise in proving Theorem~\ref{t:intro}, which we discuss below. Theorem~\ref{t:intro} holds for a more general class of critical perturbations of the Alt-Caffarelli functional that is introduced in Section~\ref{ss:mainprob}. One generalization is that Theorem~\ref{t:intro} holds for the functional \eqref{eqn: intro  functuonal} with $d_1 $ replaced by $\tilde{d}_0 +d_1$ where $\tilde{d}_0$ is an essentially equivalent regularization of $d_0$ in \eqref{e: d0} above.

Returning to the discussion of the selection principle, we use the functional $\Ep^*$ to establish Step (1) of the selection principle corresponding to the inequality $\lambda_1(\W) + \tpar \tor(\W ) \geq \lambda_1(B_R) + \tpar \tor(B_R)$ instead of \eqref{e:introfk}. This linear stability-implies-nonlinear stability result takes the following form. 
\begin{corollary}\label{c:intro} Assume that, for  some $ 0 < \tpar < \tpar_0$ and $\e > 0$, the inequality
	\[
		d_1^2(B_R(x_\W), \W) +d_0^2(B_R(x_\W), \W) \leq C [\ei(\W) - \ei(B_R) + \tpar(\tor(\W) - \tor(B_R))]
	\]
	holds for all $\W$ with $|\W| = |B_R|$ whose boundaries may be expressed as a $C^{2, \a}$ normal graph over $\p B_R$ with norm bounded by $\e$. Then it also holds for all open sets $\W$ with $|\W| = |B_R|$ (with a possibly larger constant).
\end{corollary}
Although it may appear that we have departed from the central aim of establishing Step (1) of the selection principle for the Faber-Krahn inequality, we show in the companion paper \cite{QACF} that the statement of Corollary~\ref{c:intro} actually implies that the same is valid with $\tpar = 0$ after combining with the Kohler-Jobin inequality (\cite{KJ2,KJ1}, see also \cite{KJ3, BrascoInvitation}), an inequality that relates $\tor(\W) - \tor(B_R)$ to $\ei(\W) - \ei(B_R)$. It is also shown in \cite{QACF} that \eqref{e:introfk} is linearly stable with $d = d_1 + d_2$ and $\a = 2$, which when combined with Corollary \ref{c:intro} gives nonlinear stability as well. We remark that the constant $C$ in Corollary~\ref{c:intro} is obtained through a compactness argument and is not explicit. It remains an open problem to establish a sharp quantitative form of the Faber-Krahn inequality, with any distance (see \cite[Open Problem 2]{BD17}); a  non-sharp quantitative estimate with an explicit constant was shown dimension 2 in \cite{HansenNadir94} and in general dimension in \cite[Theorem 2.10]{BD17}. 
 In a related direction, let us mention the paper \cite{BhatWeits}, in which the authors show that the deficit in the Faber-Krahn inequality controls the $L^\infty$ distance between {\it symmetrized} first eigenfunctions.

The idea of using the torsional rigidity to prove a quantitative form of the Faber-Krahn inequality was already present in \cite{BDV15}, though it plays a different role there. In \cite{BDV15}, the authors prove a quantitative version of the Saint-Venant inequality (the minimality of balls for the torsional rigidity), and then using the Kohler-Jobin inequality, they immediately obtain quantitative versions of the Faber-Krahn inequality and a whole family of shape optimization problems for the Poincar\'{e}-Sobolev constant (for which balls are always optimal). The introduction of the torsional rigidity in \cite{BDV15} simplifies some aspects of the proof and allows the authors to kill many birds with one stone, but does not appear essential to the argument for the stability of the Faber-Krahn inequality in particular. The role of the torsional rigidity in our context is plays a more pivotal role, and we do not know how to proceed without it.

In view of the results in \cite{BDV15}, one is naturally led to ask whether the results of the present paper can be extended to a more general family of inequalities for the optimization of the sharp Poincare-Sobolev constant. There is, however, a major difference between the stability result of Corollary \ref{c:intro} and the result in \cite{BDV15}: in our case the eigenfunction of the domain appears on left-hand side of the inequality. When considering stability statements like these for generalized inequalities, one must then carefully select which functions are reasonable to use in the distance, and see to what extent they can be controlled by the inequality deficit. 

As such, we leave the question of which exact generalizations of distance or Poincar\'{e}-Sobolev inequalities admit stability statements like Corollary \ref{c:intro} as an interesting and possibly challenging open problem, and believe that the methods developed here are applicable to treating it systematically. Perhaps of particular interest is the distance
\[
d_{\gamma, p}(U, V) = \|w_U - w_V\|_{L^p}
\]
where $w_U$ is the torsion function associated with $U$ (see below, Section \ref{s:basics}), both for the Faber-Krahn and Saint-Venant inequalities and $p \in [1, \frac{2n}{n-2})$. These distances metrize the notion of $\g$-convergence of sets frequently used in shape optimization (see \cite[Chapter 4]{BB} for a discussion of known relations between various such distances).

While the preceding discussion was carried out in $\R^n$ to more clearly illustrate the concepts in play, the results proved here also apply to Riemannian manifolds, and are stated accordingly (see Theorem~\ref{thm: summary for minimizers}  and Theorem~\ref{thm: quantitative stability}  for the general forms of Theorem~\ref{t:intro} and Corollary~\ref{c:intro} respectively). In particular, Corollary \ref{c:intro} is also valid for any manifold for which balls are the isoperimetric sets (i.e. they are the unique sets to attain the minimum in the isoperimetric inequality), or with suitable modifications to the statement, on any manifold on which minimizers to \eqref{e:introfk} are unique up to isometry. The former condition is restrictive and only known for a handful of manifolds (see \cite{V19} for a survey of the known results), but includes the round sphere and hyperbolic space with standard metrics, which will be studied in  \cite{QACF} in greater detail: 
\begin{remark}
	{\rm
	Corollary~\ref{c:intro} holds on  hyperbolic space and the round sphere, with the barycenter $x_\W$ replaced by the set centers defined in Examples~\ref{ex: hyp space} and \ref{ex: round sphere} respectively.
	}
\end{remark}


The outline of our approach and the structure of the rest of the paper is as follows. In Section \ref{s:setup} we give careful definitions of the quantities discussed here on manifolds and set up notation for everything to follow.  In Section \ref{s:selection}, we discuss how to pass from Theorem \ref{t:intro} to Corollary \ref{c:intro} (and state it more carefully and in greater generality).   In Section \ref{s:basics} we establish basic properties of the first eigenvalue and the torsional rigidity and discuss Faber-Krahn inequalities on manifolds to the extent needed. The introduction of the torsional rigidity is the first key idea in the paper, and in Section \ref{s:basics} we also prove \eqref{e:introkey}. 
As previously discussed, this inequality will be necessary to prove even basic estimates.  

Typical approaches to Bernoulli free boundary problems involve first proving existence of minimizers, usually via the direct method, and then proving nondegeneracy and growth estimates at the free boundary. 
Due to the presence of the eigenfunction penalization term, which is of the same order as the other terms in the functional, a direct approach to the existence of minimizers fails as the functional is not lower semicontinuous with respect to topologies for which compactness is available. For this reason, we must introduce a new approach to proving existence. In Section \ref{s:lb} we prove a priori nondegeneracy estimates for the free boundary for 
inward minimizers. In Section \ref{s:ub} we prove a priori Lipschitz growth estimates at the free boundary for outward minimizers. In Section \ref{s:exist} we 
  construct a minimizing sequence of outward minimizers and utilize the Lipschitz estimate to prove the limit $\Omega$ is a minimizer. This minimizer will then have both nondegeneracy and Lipschitz estimates since it is also both an inward and outward minimizer. One benefit of this approach is that the entire argument takes place in the class of open sets; the a priori estimates allow us to avoid relaxing the problem to the class of quasi-open sets as is common in the literature \cite{BuDM93}.

From here, we would like to follow the recent approaches of \cite{CSY18, MTV17} to establish regularity for vectorial free boundary problems. The first step in doing so is to apply the boundary Harnack principle in order to establish a scalar form of the free boundary condition. Substantial difficulties arise here due to the nonlinear term, because
the contribution from $d_1$ to the Euler-Lagrange equation along $\p \W$ is of the same order as that from $\ei(\W)$ or $\tor(\W)$: it may be expressed as $|\grad v|^2$ for some function $v$ satisfying an elliptic PDE on $\W$. 
 In particular, there is no favorable sign appearing in this term (it is not \emph{elliptic}, in some sense), and so it must be controlled by the other contributions. The key tool toward controlling this term is a careful analysis 
of the Green's function $G_{\Omega}$ of $\Omega$, which we carry out in Section~\ref{s:measure}. This involves utilizing an inhomogeneous boundary Harnack principle recently shown in \cite{AKS20}.

We begin Section \ref{s:el} by considering more carefully how $d_1$ and $u_\W$ change under smooth deformations of domains. Applying the estimates from 
the previous section we obtain several forms of the free boundary condition satisfied along $\p \Omega$ and arrive at a pointwise version. The final fundamental point in the paper is to 
again apply the inhomogeneous boundary 
Harnack principle to rewrite the free boundary condition in a form suitable to known methods. 

In Section \ref{s:reg} we apply recent free boundary theory to 
conclude that $\p \Omega$ is $C^{1,\alpha}$. In Section \ref{s:highreg} we apply a higher order inhomogeneous boundary Harnack principle \cite{DSS15} to prove Theorem \ref{t:intro} (that $\p \Omega$ may be parametrized as  a
$ C^{2,\alpha}$ normal graph over the sphere).  Appendix~\ref{app A} gives further examples of functionals to which our results apply.  Finally, Appendix \ref{s:appendixnta} shows that domains supporting solutions to elliptic PDE with linear growth away from $\p U$ are nontangentially accessible, a result which is known in the literature but, to the best of our knowledge, not in the generality required here. This is a technical point relevant in Section~\ref{s:measure}.\\

{\it Acknowledgements.} MA was partially supported by Simons Collaboration Grant ID 637757. During the course of this work, RN was partially supported by the National Science Foundation under Grant No. DMS-1901427, as well as Grant No. DMS-1502632 ``RTG: Analysis on manifolds'' at Northwestern University and Grant No. DMS-1638352 at the Institute for Advanced Study. 
This project originated from discussions at the 2018 PCMI Summer Session on Harmonic Analysis. We are grateful to the helpful anonymous referee who provided detailed and valuable feedback on this paper.


\section{Setup and definitions}\label{s:setup}

\def\Isom{G}
\def\Qb{Q}

Throughout this paper, we let $(M, g)$ be a  complete,  $n$-dimensional smooth Riemannian manifold without boundary that is connected and oriented. We further assume that $(M,g)$ has bounded geometry, that is, $\text{inj}_M := \inf_{x\in M} \text{inj}_x >0$ and $\sup_{x \in M} |\text{Rm}(x)|<\infty.$ 
Here $\text{inj}_x$ is the injectivity radius of $(M,g)$ at $x$, i.e. the largest radius $r$ such that the exponential map $\text{exp}_x: T_xM \to M$ is a diffeomorphism from $B(0,r)$ to the geodesic ball $B_g(x,r)$, and $|\text{Rm}(x)|$ denotes the norm of the Riemann curvature tensor at $x$ with respect to $g$.
All constants will depend on the metric $g$; it is likely that these constants are uniform for Riemannian manifolds with uniformly bounded geometry but we do not track this dependence.

Unless otherwise specified, all integrals are taken with respect to the volume measure $m$ associated to $g$. We use the notation $|\W|$ for the volume of a measurable set $\W \ss M$ and $\cH^{n-1}$ for $(n-1)$-dimensional Hausdorff measure.  We let $d(x,y)$ denote the geodesic distance between $x,y\in M$ and let  
\[
	B_r(x) = \{y \in M \,  :  \, d(y, x) < r \}
\]
denote a geodesic ball centered at $x$ of radius $r$. Given a set $\W$ and a point $x \in M$, we let $d(x, \W) = \inf \{d(x,y): y \in \W\}.$ 
We let $\grad $ and $\dvg$ denote the gradient and divergence with respect to $g$ and  let $\Lap u = \dvg \grad u$ be the Laplace-Beltrami operator with respect to $g$.

Let $\Isom$ denote the (possibly trivial) group of isometries of $(M,g)$. Several global existence results in the sequel will hold for Riemannian manifolds such that either $M$ is compact or $M/G$ is compact; examples of the latter to keep in mind are Euclidean space, hyperbolic space, cylinders $S^{n-k}\times \R^k$ equipped with the standard product metric, and more generally products of space forms.

\subsection{Base energy}

\def\minE{\En_{min}}
\def\Min{\mathcal{M}}
\def\compset{\mathcal{H}}

Given a bounded open set $\W \ss M$, we may define the \emph{torsional  rigidity} of $\W$:
\begin{equation}\label{e:torsion}
	\tor(\W) = \inf_{u \in H^1_0(\W)} \int_{\W}\frac{1}{2}|\grad u|^2 - u.
\end{equation}
We  note that the definition of   the torsional rigidity given  here differs by a sign from  the typical definition. 
This is a negative quantity, and it is straightforward to check (using the Sobolev inequality and the Lax-Milgram theorem) that the infimum is finite and uniquely attained for any open bounded $\W \neq M$ by a function $w_\W \geq 0$. We will refer to $w_\W$ as the \emph{torsion function}. Note also that $\tor(\W)$ is decreasing with respect to set inclusion.

The \emph{first (Dirichlet) eigenvalue} of $\W$ is defined as
\begin{equation}\label{e:evalue}
	\ei(\W) = \inf_{u \in H^1_0(\W)\sm \{0\}} \frac{\int_{\W} |\grad u|^2}{\int_{\W} u^2}.
\end{equation}
This infimum is also attained by a nonnegative function $u_\W$, the \emph{first eigenfunction}, and $\ei(\W)$ is also a decreasing function with respect to set inclusion. The first eigenfunction is unique up to scalar multiples so long as $\l_1(\W) < \l_2(\W)$, which in particular is true if $\W$ is connected. The notation $u_\W$ will be used for the unique first eigenfunction with $u_\W \geq 0$ and $\int u_\W^2 = 1$, whenever it exists and is unique. We frequently will  extend $u_\W$ and $w_\W$ by zero to  be  defined on all of $M$, using the same notation  to indicate the functions $u_\W , w_\W : M\to  \R$.

In principle, we are interested in minimizing  energy functionals of the type \eqref{eqn: intro  functuonal} among sets $\Omega\subset M$ of a  fixed volume $|\W| =v$. In practice, it is more convenient to  replace the volume constraint with a volume penalization term. Following \cite{AAC, BDV15}, we define the volume penalization  
\begin{equation}\label{eqn: volume penalization def}
	\fv(t) = \begin{cases}
		\vpar (t - v) & t \leq v\\
		\frac{1}{\vpar}(t - v) & t > v\,.
	\end{cases}
\end{equation}
The idea of the volume penalization \eqref{eqn: volume penalization def} is that by choosing the parameter $\vpar$ to be sufficiently small, a set $\W$ whose volume exceeds the prescribed volume $v$ will have a large energy contribution coming from the term $\fv(|\W|)$.
We define the \emph{base energy} of $\W$ as
\begin{equation}\label{e:baseenergy}
	\En(\W) = \ei(\W) + \tpar \tor(\W) + \fv(|\W|),
\end{equation}
where $\tpar \in [0, 1]$. The principal goal of this paper is to study certain critical perturbations of the base energy introduced in the next subsection. In order to do so, we will need to first consider the minimization problem of the base energy itself. We  discuss the problem of minimizing $\En(\W)$ over sets briefly here and in more detail in Section \ref{s:basics}. Many of the existence and regularity results in later sections will apply to this base energy problem as well.
Let us begin with two observations  regarding the existence theory for this problem:
\begin{remark}[A hard volume constraint is needed]\label{rmk: soft constraint}
	{\rm 
	If we attempt to minimize the base energy $\En$ among all bounded open subsets of $M$,  then the energy may fail to be bounded from below and minimizers may fail to exist or $\W = M$ may be a minimizer. For instance, on Euclidean space, a simple scaling computation shows that  $\En(B_R) \to -\infty$ as $R \to \infty$ for any fixed triple of parameters $v, \eta, \tpar>0$. Similarly, if $M$ is compact without boundary, then  $H^1(M) = H^1_0(M)$, and so the minimizing sequence $w_k =- k$ for $\tor(M)$ tells us that $\tor(M) = -\infty$. Since the other quantities in the energy are finite, this implies that $\En(M) = -\infty$.

	 To remedy this issue, we will introduce a ``hard volume constraint'' that is much larger than $v$. Given the desired volume constraint $v \in (0,|M|)$, we fix a larger number $\vmax \in (v, |M|)$ and  minimize with respect to competitors $\W$ with $|\W| \leq \vmax$. Ultimately this hard constraint will never be saturated: in Proposition~\ref{l:volumeisright} we  show that for small values of $\vpar$ (depending on $v$ and $\vmax$) and $\tpar$ (depending on $v, \vmax$ and $\vpar$), any minimizer over this class actually satisfies the original desired volume constraint $|\W| = v$. In particular, this means the volume-constrained problem (minimizing only over $|\W| = v$) is equivalent to the volume-penalized problem where $|\W| \leq \vmax$.
	}
\end{remark}
\smallskip

\begin{remark}[Compactness issues on non-compact manifolds]\label{rmk: issues at infinity}
	{\rm On noncompact manifolds, a minimizing sequence $\W_k$  for the base energy $\En$ can ``drift to infinity'' in the sense that $d(x_0, \W_k) \rightarrow \infty$ for any reference point $x_0$. In view of this issue, we first minimize over the class of $\W$ contained in (smoothed out) balls of large radius $R$.  Depending on the geometry of the manifold at infinity,  global minimizers may simply fail to exist when $R\to \infty$. However, when $M/G$ is compact, the compactness issues are solely caused by the symmetries of the metric and we may send the radius $R$ to infinity  using concentration compactness techniques (see \cite{Lions84}). This will be addressed in Section \ref{ss:globalbase} for the base energy and then more completely in Theorem \ref{t:globalexist}. 
		
	Since geodesic balls need not have smooth boundary on an arbitrary Riemannian manifold, it is more convenient from a technical standpoint to work with an exhaustion $\Qb_R$ of smooth bounded open sets that play the role of smoothed-out balls. More specifically, we let $\Qb_R$ be a fixed 1-parameter family of open sets which have the following properties:
\begin{enumerate}
	\item $\bar{\Qb}_R \ss \Qb_S$ for any $S > R$
	\item $\diam(\Qb_R) \leq 2 R$
	\item $\partial \Qb_R$ is either smooth (a finite union of disjoint $C^2$ codimension-$1$ submanifolds) or empty
	\item $M = \cup_R \Qb_R$.
\end{enumerate}
	}
\end{remark}
\smallskip

In view of Remarks~\ref{rmk: soft constraint} and \ref{rmk: issues at infinity}, we will fix parameters $R>0$ and $\vmax \in (0, |M|)$ and minimize the base energy (as well as the main energy) over open sets in $\Qb_R$ of volume at most $\vmax.$ We use the notation 
\begin{equation}\label{e: compset}
	\compset = \compset_{R, \vmax} = \{\W \ss M \text{ open}: |\W|\leq \vmax, \W \ss \Qb_R  \}
\end{equation}
to denote this collection of sets. We will assume throughout the paper, without further note, that $R$ is chosen to be large enough so that $|\Qb_R| \geq \vmax$. In this way, the container $\Qb_R$ does not obstruct the hard volume constraint and the collection $\compset $ is nonempty.  Moreover, we will always assume that $\vmax <|M|$ when the volume of $M$ is finite.

 Given parameters $0<v< \vmax < |M|, \eta>0$, $\tpar\in (0,1] $ and $R>0$, we consider the following minimization problem for the base energy:
\begin{equation}\label{e: emin def}
	\minE = \minE(v, \vmax, \eta,  \tpar , R) = 	\inf\{ \En(\W): \W \ss \compset \};
\end{equation}
we will drop the parameters unless ambiguous.  The collection of minimizers for this problem will be denoted by
\begin{equation}\label{eqn:min}
	\Min = \Min(v, \vmax, \eta,  \tpar , R) = \{ \W \ss \compset : \En(\W) = \minE\},
\end{equation}
which we will show in Lemma \ref{l:emin} is always nonempty. While several of our intermediate results will hold for larger parameter ranges, the reader should typically think of first fixing the desired and hard volume constraints $v, \vmax$ and the radius $R$, then fixing the volume penalization parameter $\vpar$, and finally fixing the coefficient $\tpar $ in front of the torsional rigidity.

In the following example, we consider the minimization problem \eqref{e: emin def} on  simply connected space forms; generally speaking these are  important motivating  examples to keep in mind throughout the paper. Since these spaces are compact after modding out by the isometry group, we may send the parameter $R$ to infinity and so the following example has no $R$ parameter.
\begin{example}\label{ex: balls}{\rm
Let $(M,g)$ be  Euclidean space, hyperbolic space, or the round sphere.	Fix parameters  $v, \vmax $ with $0< v <\vmax$, choosing  $\vmax < |M|$ in the case of the round sphere. For $\eta\leq \eta_0(v, \vmax)$ and $\tpar \leq \tpar_0(v, \vmax, \eta)$,   geodesic balls of volume $v$ are the unique minimizers of the base energy $\En$ among all open bounded subsets of $M$ with volume at most $\vmax$.
}
\end{example}

For general Riemannian manifolds, we cannot hope to explicitly characterize minimizers as in the example above. However, the following proposition summarizes the basic facts that will be proven about minimizers of the base energy in this paper. The proof follows from combining Lemma \ref{l:emin} and  Theorem~\ref{t:globalexist} (local and global existence respectively),  Lemma \ref{l:simpleevalaux} (connectedness), and Lemma~\ref{l:volumeisright} (volume constraint satisfied).
\begin{proposition}[Minimizers of the base energy] \label{prop: base summary}
	Fix $R>0$ and $v \in (0,|M|)$. Fix any $\vmax \in (v, |M|)$. There exists $\vpar_0 = (R, v, \vmax)$ such that for $\vpar \leq \vpar_0$, there exists  $\tpar_0 = \tpar_0 ( R, v, \vmax , \vpar)$ such that the following holds. If $\tpar \leq \tpar_0$, then $\minE >-\infty$ and  a minimizer of the base energy $\En$ exists among sets in $\compset,$ i.e. the set $\Min$ is nonempty. Moreover, any minimizer in $\Min$ is connected and has volume equal to $v$.
	
	If $M$ is compact or $M/\Isom$ is compact, then constants above may be taken independent of $R$ and the  minimization of $\En$ may be taken among all open bounded subsets of $M$ with volume at most $\vmax.$
\end{proposition}

\subsection{Main energy functional}\label{ss:mainprob}
\def\nl{\mathfrak{h}}
Let us now introduce the class of energy functionals whose existence and regularity theory constitute the heart of this paper and using which  we can carry out a selection principle as described in the introduction (in Section~\ref{s:selection} below). Following the discussions in the introduction and in the previous subsection, we are interested in functionals of the type \eqref{eqn: intro  functuonal}, though as in the previous subsection we relax the volume constraint. 
So, the general form of the functionals  we consider is
\begin{equation}\label{eqn: main functional}
	\Ep(\W) = \En(\W) + \err \nl(\W).
\end{equation}
Here $\nl(\W)$ is a functional mapping open bounded sets to $\R$ and $\err < \err_0$ is a parameter that will be chosen sufficiently small depending on the parameters of the base energy $\En$. The existence and regularity results of this paper will hold whenever $\nl$ is an {\it admissible nonlinearity} as defined in Definition~\ref{def: nl} below. Before defining this general class of admissible nonlinearities, we motivate the definition with some  important concrete examples that will be used for our selection principle applications. In these examples, the  relevant nonlinearities take the form
\begin{equation}\label{eqn: SP nl}
\nl(\W) = \sqrt{c^2 +  (c - d_*(\W)^2)^2 } -c
\end{equation}
for a constant $c\in (0,1]$, where $d_*(\W)$ is a suitably defined distance of $\W$ to the class of minimizers $\Min$ of the base energy. Note that \eqref{eqn: intro  functuonal} takes this form with $d_*(\W) = d_1(\W, B_R(x_\W))$ on Euclidean space; the subtraction of the constant $c$ in \eqref{eqn: SP nl} is immaterial for the minimization problem but conveniently normalizes the functional. We are interested in generalizing this in two directions: (1) we want a distance $d_*$ that measures the size of the symmetric difference as well as the eigenfunction difference, and (2) we wish to suitably generalize this to certain Riemannian manifolds.

In order to establish Step (1) of the selection principle (and ultimately quantitative stability) in the form of Corollary~\ref{c:intro} in which both the eigenfunction distance $d_1$ and the asymmetry distance $d_0$ are controlled (say on Euclidean space), a first approach might be to let $d_*(\W) = d_1(\W, B_R(x_\W)) +d_0(\W, B_R(x_\W))$ in \eqref{eqn: SP nl}, where $d_0$ is the symmetric difference defined in \eqref{e: d0}. The  issue  here, already understood in \cite{BDV15},  is that the distance $d_0$ is not sufficiently smooth (as the integral of a characteristic function) and consequently minimizers of such a functional will not be smooth. 

To deal with this, we follow the approach of \cite{BDV15} and consider a smoothed out version of the asymmetry term $d_0$. More specifically, consider a bounded, open set $U$ with smooth boundary (say $\p U$ is a disjoint union of $C^2$ embedded submanifolds); in the relevant applications $U$ will be a minimizer of the base energy $\En$, so in Euclidean space, hyperbolic space, or the round sphere, take $U$ to be a geodesic ball. Take $f$ to be a smooth nondecreasing function with $f(t) = t$ for $|t| \leq c_0$ and $f(t)$ is locally constant for $|t| > 2 c_0$, where  $c_0$ is selected so that $d(x, \p U)$ is a $C^2$ function on $\{d(x, \p U) < 4c_0\}$. One can always find such a $c_0$; see \cite[Lemma 14.16]{GT}). Now, let $\psi_U$ be given by
\[
\psi_U(x) = \begin{cases}
f(d(x, \p U)) & x \in U \\
-f(d(x, \p U) )& x \in M \sm U.
\end{cases}
\]
The essential point of the function $\psi_U$ is that $\int_\W \psi_{U} - \int_{U} \psi_U $ is comparable to $|U \triangle \W|^2$
for $\W$ whose boundary is a (small) normal graph over $\p U$, and controls $|U \triangle \W|^2$ in all cases. More precisely, for any $\W$ we have
\[
\int_\W \psi_U - \int_U \psi_U = \int_{\W \triangle U} |\psi_U| \geq \int_{\W \triangle U} \min\{|d(x, \p U)|, c_0\} dx \geq c |\W \triangle U|^2,
\]
where the last step can be seen from the fact that the integral is minimized by tubular neighborhoods of $U$, and then by changing variables and using the smoothness of $\p U$.
 On the other hand, if the boundary of $\W$ is a $C^2$ normal graph $\xi$ over $\p U$ with $|\xi|\leq c_0$, we may change variables and integrate
\[
\int_{\W \triangle U} |\psi_U| \leq C \int_{\p U} \int_0^{|\xi(x)|} t dt d\cH^{n-1} = C \int_{\p U} \int_0^{|\xi(x)|^2/2} 1 dt d\cH^{n-1} \leq C|U \triangle \W|^2.
\]

So, given any two bounded open sets $\W$ and $U$ with $U$ smooth as above, we let
\begin{equation}\label{eqn: dstar two sets}
d_*(\W, U)^2 = \int_\W \psi_{U} - \int_{U} \psi_U + \int |u_\W - u_U|^2
\end{equation}
where $u_\W$ denotes the first eigenfunction of $\W$, normalized so that $u\geq 0$ and $\| u_\W\|_{L^2} = 1$ and extended by zero to be defined on all of $M$. In each of the following examples, we will use $d_*(\W, U)$ to define a distance $d_*(\W)$ to the collection $\Min$ of minimizers of the base energy and consider the functional \eqref{eqn: main functional} with the nonlinearity \eqref{eqn: SP nl} for this suitable definition of $d_*(\W)$.
\begin{remark}[Unique first eigenfunction]\label{rmk: unique eigenfunction}
	{\rm 
Lemma \ref{l:simpleeval} below guarantees that if $\En(\W) \leq \minE + \err_0$ for $\err_0$ is sufficiently small, then $\W$ has  a unique first eigenfunction and so $d_*(\W, U)$ is well-defined for such a set.
 This means that by choosing the parameter $\err$ in \eqref{eqn: main functional} accordingly, $u_\W$ and $d_*(\W, U)$ will be well-defined for any minimizing sequence or minimizer of \eqref{eqn: main functional}. As such, we will generally only define $\nl$ for sets with $\En(\W) \leq \minE + \err_0$.	There is no loss of generality, at least when considering minimizers, to define $\nl(\W) = 1$ if $\En(\W) \geq \minE + \err_0$.
	}
\end{remark}

The simplest example is the case when only one minimizer of the base energy exists. 
\begin{example}[Unique minimizer]\label{ex: one min}
	{\rm
	Suppose that the collection of minimizers $\Min$ of the base energy consists of a single set $U$, and $U$ has boundary of class $C^2$. Then the distance  $d_*(\W)$ of $\W$ to the collection of minimizers of the base energy with $v= |\W|$ is defined by $
d_*(\W) = d_*(\W, U).$
	}
\end{example}
  
In the case when minimizers of the base energy are not unique, to define a distance of a given set $\W$ to the collection of minimizers, we must select the nearest minimizer to $\W$. On Euclidean space, where translation invariance gives rise to nonuniqueness of minimizers, we do this by choosing the ball with the same barycenter as $\W$. 
\begin{example}[Euclidean space]\label{ex: functional rn}
\rm{
Let $(M,g)$ be Euclidean space. Given a bounded open set $\W$, let $R$ be the unique radius so that $|\W| =|B_R|$ and let $x_\W = \fint_\W x\,dx$ be the barycenter of $\W$. Then  the distance  $d_*(\W)$ of $\W$ to the collection of balls $B_R$ (i.e. the minimizers of the base energy with $v= |\W|$) is defined by 
$
d_*(\W) = d_*(\W, B_R(x_\W)).
$
}	
\end{example}

We saw in Example~\ref{ex: balls} that the situation on hyperbolic space and the round sphere is similar to Euclidean space: the collection of minimizers of the base energy comprises geodesic balls of suitable radius centered at {\it any } point. In these cases, we define an appropriate ``set center'' in analogy to the Euclidean barycenter.
\begin{example}[Hyperbolic space]\label{ex: hyp space}
{\rm 
Let $(M,g)$ be hyperbolic space. Given a set bounded open set $\W \subset M$, define the set center $x_\W$ of $\W$ by $x_\W= \text{argmin}_x \int_\W d^2(x,y) \, dm(y)$, which is well-defined.  Then the distance  $d_*(\W)$ of $\W$ to the collection of balls $B_R$ (i.e. the minimizers of the base energy with $v= |\W|$) is defined by 
$
d_*(\W) = d_*(\W, B_R(x_\W)).
$
}	
\end{example}

On the round sphere, we cannot define the notion of set center in analogy to Example~\ref{ex: hyp space} as the minimum may not be uniquely achieved.  We instead consider the following.
\begin{example}[The round sphere]\label{ex: round sphere}
{\rm 
Let $(M,g)$ be the round sphere. Consider its standard embedding  in $\R^{n+1}$. Given any open $\W \ss M$, let $
		y_\W = \fint_{\W} y d \cH^{n}(y),$
	where $y, y_\W \in \R^{n+1}$ and this is a (vector-valued) surface integral. If $y_\W \neq 0$, then define the set center of $\W $ as $x_\W = y_\W /|y_\W|$. The distance  $d_*(\W)$ of $\W$ to the collection of balls $B_R$ (i.e. the minimizers of the base energy with $v= |\W|$) is defined by $d_*(\W) = d_*(\W, B_R(x_\W)).$

}		
\end{example}
In all three of the preceding examples, uniqueness of minimizers of the base energy fails solely due to the isometries of the space, and modulo the isometry group $\Isom$, there is a unique minimizer $U$ with boundary of class $C^2$. In more general situations of this type, we  define a suitable notion of {\it set center}, which is a mapping $\W \mapsto x_\W $ for which there is a unique minimizer $U_\W$ of the base energy with the same set center as $\W$. Because the definition is somewhat technical, we postpone it until Definition~\ref{def: set center} in Appendix~\ref{app A}. In this situation, we define $d_*(\W)$ analogously to the examples above; the Euclidean barycenter and the sets centers for hyperbolic space and the round sphere above are examples of set centers in the sense of Definition~\ref{def: set center}.
\begin{example}[Minimizer is unique modulo isometries]\label{ex: general example}
	{\rm 
	Assume that $\Min$ consists of a single set $U$ up to the action of the isometry group $\Isom$.  Also assume that it has smooth boundary, and that there exists a set center $\W\mapsto x_\W$ adapted to $U$. We define
\begin{equation}\label{e:dstar}
d_*(\W) = d_*(\W, U_\W)
\end{equation}
where $U_\W$ the unique set in $\Min$ with the same set center as $\W$.
	}
\end{example}

 It may be tempting to give {generic} constructions of $d_*(\W)$, independent of set centers, by solving minimization problems along the lines of
$
	U_\W = \text{argmin}\left\{d_*(\W, U) : U \in \Min  \right\}
$ and letting $d_*(\W) = d_*(\W, U_\W).$
However, such variational problems have issues with uniqueness and regularity of the solution map: these are closely related to the convexity and smoothness of the ``distance'' functional being minimized, as well as the structure of the family $\Min$.

\smallskip

All of the examples given above fall into the general class of admissible nonlinearities $\nl$ for which our main existence and regularity results are valid.  Below, $u_\W$ is used to denote the normalized, nonnegative first eigenfunction of $\W$ (see \eqref{e:evalue}); recalling Remark~\ref{rmk: unique eigenfunction}, this function exists and is unique when $\En(\W) \leq \minE + \err_0$.
\begin{definition}[Admissible nonlinearities]\label{def: nl} Fix a (possibly trivial) closed subgroup  $\Isom_0 \leq \Isom$
	of the isometry group of $(M,g)$. A function $\nl$ mapping bounded open subsets of $M$ to $\R$ is said to be \emph{an admissible nonlinearity} with respect to $\Isom_0$ if the following properties hold.
	\begin{enumerate}[({N}1)]
	\item \label{a:nlbdd} $\nl \in [0, 1]$. 
	\item \label{a:nlinv} $\nl$ is invariant under isometries in $\Isom_0$.
	\item \label{a:nllip} If $\W, \W'$ are bounded open sets with $\En(\W) \leq \minE + \err_0$, then
	\[
		|\nl(\W) - \nl(\W')| \leq |\W \triangle \W'| + \int |u_\W - u_{\W'}|.
	\]
	\item \label{a:nlc1} For any $x_0\in M$ and $ r \leq r_0$, take a one-parameter family $\phi_t$  of diffeomorphisms with $\phi_0(x) = x$
	 and $|\p_t \phi_t| \leq 1$ such that $|\{x : \phi_t(x) \neq x \}| \ss B_r(x_0)$. Assume $\phi_t(\W)$ is a set with $\En(\phi_t(\W)) \leq \minE + \err_0$ and $\cH^{n-1}(\p \W) < \infty$; then
	\[
		\limsup_{t \rightarrow 0} \frac{1}{|t|} \left|\nl(\phi_t(\W)) - \nl(\W) - \int (u_{\phi_t(\W)} - u_\W) a_\W - \left[\int_{\phi_t(\W)} b_\W - \int_\W b_\W\right]\right| \leq C_{\nl} r^n,
	\]
	where $a_\W$, $b_\W$ are continuous functions $M \rightarrow \R$ with $|a_\W|, |b_\W| \leq 1$. Moreover, $\|b_\W\|_{C^{1, 1}(\Qb_R)} \leq 1$ and $\|a_\W\|_{C^{0, 1}(\W)}\leq 1$.
\end{enumerate}

\end{definition}
 Assumption~\ref{a:nlbdd} simply says that $\nl$ is nonnegative and uniformly bounded on bounded open sets, and it is convenient to normalize any such functional so that the upper bound is $1$. Assumption~\ref{a:nlinv} should be understood with Example~\ref{ex: general example} in mind; when the base energy $\En$ has a nonunique minimizer due only to a subgroup of isometries $G_0$, then \ref{a:nlinv} is used to compare a set $\W$ to the nearest minimizer of the base energy.
Assumption \ref{a:nllip} heuristically states that $\nl$ is a Lipschitz function $\W \mapsto \R$, where the topology placed on the space of sets is either the $L^1$ distance $|\W \triangle \W'|$ or a distance governed by the $L^1$ difference of their eigenfunctions.
Assumption \ref{a:nlc1} should be viewed as a higher regularity assumption on $\nl$; it is asking that $\nl$ is not only Lipschitz but also $C^1$, with derivative represented by the functions $a_\W$ and $b_\W$.	

Only the first three assumptions \ref{a:nlbdd}, \ref{a:nlinv}, and \ref{a:nllip} are needed for the existence theory and core estimates up through Section \ref{s:measure}. 
The last assumption \ref{a:nlc1} is used  to derive the Euler-Lagrange equation in Section~\ref{s:el} (indeed, the functions $a_\W$ and $b_\W$ appear in the first variation) and for the regularity of minimizers in Sections \ref{s:reg} and \ref{s:highreg}. In particular the extra regularity of $a_\W, b_\W$ is only needed in Section \ref{s:highreg} to obtain better smoothness of $\p \W$ for minimizers.

In Appendix~\ref{app A}, we verify  that the examples given above are admissible nonlinearities. Given an admissible nonlinearity, it is easy to construct others via composition with functions of one variable: 
\begin{remark}\label{rmk: admissible compositions}{\rm
 Given any admissible nonlinearity $\nl$, and any $C^1$ function $\phi : [0, 1] \rightarrow [0, 1]$   with $|\phi'|\leq 1$, the composition $\phi \circ \nl$ is again an admissible nonlinearity. No monotonicity or structure of $\phi$ is required here.
}
\end{remark}

\subsection{Minimizers of the main energy}

\def\UP{\text{UP}}
\def\DO{\text{DO}}
As we did for the base energy, we will fix $R>0$ and $\vmax$ and minimize the main energy over the set $\compset=\compset_{R,\vmax}$ defined in \eqref{e: compset}.
We say that a bounded open set $\W$ is a \emph{minimizer} if $\W \in \compset$ and for any $\W' \in \compset$,
\begin{equation}\label{e:min}
	\Ep(\W) \leq \Ep(\W').
\end{equation}
We say that $\W$ is an \emph{inward minimizer} if \eqref{e:min} holds for any $\W' \in \compset$ with $\W' \ss \W$. We say that $\W$ is an \emph{outward minimizer} if \eqref{e:min} holds for any $\W' \in \compset$ with $\W' \supset \W$.

The following theorem summarizes the main existence and regularity properties for $\Ep$ that are established in this paper. The proof follows by combining Theorem~\ref{t:min} and Theorem~\ref{t:globalexist}  (existence), Proposition~\ref{l:volumeisright} (volume constraint), Theorem~\ref{thm:usefulEL} (free boundary condition), Theorem~\ref{thm:C1a} ($C^{1,\alpha}$ regularity), and Corollary~\ref{c:higherreg} (higher regularity). 

We say that a collection $\mathcal{C}$ of bounded open subsets  of $M$ is \emph{uniformly} $C^{k}$ if there is a constant $C > 2/\text{inj}_M $ such that for every $U \in \mathcal{C}$ and every $x \in \p U$, the set  $\p U\cap B_{1/C}(x)$ may be expressed as a graph with $C^{k}$ norm at most $C$ over a hyperplane in normal coordinates.

\begin{theorem}\label{thm: summary for minimizers} 	Fix $R >0$ and $v \in (0, |M|)$. Fix any $\vmax \in (v, |M|).$ There exist positive constants $\vpar, \tpar$ $\err_0>0$ depending on $R, v,\vmax$ such that the following holds. Let $\nl$ be an admissible nonlinearity with respect to the isometry group $G$ of $(M,g)$ and for fixed $\err <\err_0$, consider the energy functional $\Ep$ defined in \eqref{eqn: main functional} with the parameters $v, \vpar ,$ and $\tpar$ in the base energy $\En$ of \eqref{eqn: main functional}. 

	\begin{enumerate}
		\item There exists an open set $\W$ that minimizes $\Ep$ among sets in $\compset =\compset_{R,\vmax}$. Any such minimizer satisfies the volume constraint $|\W| =v$, is a set of finite perimeter, and has a unique first eigenfunction $u_\W$ up to scaling. 
		\item There exist $C>0$ and $\alpha \in (0,1)$ depending only on $v, \vmax, R$,   a function $\r \geq - C \err$ with $\|\r\|_{C^{0, \a}(\p \W)} \leq C$, and a constant $A_0\in [1/C,C]$ such that the normalized first eigenfunction $u_\W$ satisfies the free boundary condition
	\[
		|\grad u_\W (x)|^2(1 + \r(x)) = A_0
	\]
	for $\cH^{n-1}$-a.e. $x \in \p^* \W \cap \Qb_R$.
		\item Assume further that the collection $\Min$ of minimizers of the base energy  is uniformly $C^{4}$. For every $\alpha \in (0,1)$ and  $r_1 > 0$ there is a $\err_1 = \err_1(v, \vmax, \vpar, r_1, \alpha) > 0$ such that the following holds. If $\err < \err_1$, then for any minimizer $\W$ of $\Ep$, the set $\p \W \cap \{ x \in \Qb_R: d(x, \p \Qb_R) > r_1 \}$ may be parametrized as a $C^{2, \a}$ normal graph (with $C^{2, \a}$ norm bounded by $r_1$) over $\p U$ for some $U \in \Min$. 
	\end{enumerate}

	If $M/G$ is compact, then all constants may be taken to be independent of $R$, the minimization of $\Ep$ may be taken among all bounded open subsets of $M$ with volume at most $\vmax$, and in (3) the entire boundary of any minimizer $\W$ may be written as a $C^{2,\alpha}$ graph over $\p U$ for some $U \in \Min$.
\end{theorem}

Theorem~\ref{thm: summary for minimizers}(1) implies that the minimization of $\Ep$ among sets in $\compset$ is equivalent to the minimization problem
\[
\inf\{ \Ep(\W) : \W \in \compset^*, |\W| =v\}
\]
where we let $\compset^*$ denote either the collection of open bounded subsets of $\Qb_R$ or, in the case that $M/G$ is compact, the collection of open bounded subsets of $M$.\\

In order to better track the dependence of constants in various estimates in the sections to follow, we define two quantities which measure growth properties of the eigenfunction and torsion function of $\W$. Let $\W$ be an open subset of $\Qb_R$ with $\ei(\W) < \l_2(\W)$, and $u_\W$ the first eigenfunction (normalized so that $\int u_\W^2 = 1$ and $u_\W \geq 0$ as usual), while $w_\W$ is the torsion function. Then
\begin{equation}\label{def: UP}
	\UP(\W) = \sup\left\{ \frac{u_\W(x) + \sqrt{\tpar} w_\W(x)}{d(x, \partial \W)}: x \in \W, d(x, \partial \W) \in (0, 1) \right\}
\end{equation}
and
\begin{equation}\label{def: DO}
	\DO(\W) = \inf \left\{ \frac{1}{r}\sup_{x \in B_r(y)} u_\W(x) + \sqrt{\tpar} w_\W(x) : y \in \bar{\W}, r \in (0, 1) \right\}.
\end{equation}
In general, $\UP(\W) \in (0, \infty]$ while $\DO(\W) \in [0, \infty]$, though we will show in Section \ref{ss:basics} that $\DO(\W) < \infty$ by showing that $u_\W$ and $w_\W$ are bounded. We also always have $\DO(\W) \leq \UP(\W)$. The core nonlinear estimates for this type of variational problem, following \cite{AC81}, are that outward minimizers have $\UP(\W) \leq C < \infty$, while inward minimizers have $\DO(\W) \geq c > 0$.

\subsection{Selection Principle}\label{s:selection}
\def\vol{{\textrm{vol}}}
\def\Om{{\Omega}}
\def\pa{{\partial}}

In this section, we establish the main linear stability-implies-nonlinear stability result of this paper, which is Theorem~\ref{thm: quantitative stability} below. This is a generalization of Corollary~\ref{c:intro}.  
 As we noted in the introduction and as will be apparent in the proof, the constant $c$ in Theorem~\ref{thm: quantitative stability} below is not explicit. The proof is essentially self-contained assuming Proposition~\ref{prop: base summary} and Theorem~\ref{thm: summary for minimizers}, with the exception of some continuity statements of Lemma~\ref{l:emin} and Lemma \ref{l:globalsimpleeval}. We operate under one of the following set of assumptions: 

{\it Case 1:} $M/\Isom$ is compact. 

{\it Case 2: } $M$ is arbitrary and $R>0$ is fixed.

Let $\compset^* = \{ \W \text{ open, bounded} \}$ in Case 1 and  $\compset^* = \{ \W \text{ open} : \W \ss \Qb_R \}$ in Case 2. Fix $0 < v < \vmax < |M|$, with $\vmax < |\Qb_R|$ in Case 2. 
Let $\vpar= \vpar(v, \vmax, R)$ be chosen small enough according to Proposition~\ref{prop: base summary} as well as Lemma \ref{l:globalsimpleeval} in Case 1 (Lemma~\ref{l:emin} in Case 2). Then choose  $\tpar_* = \tpar_*(v, \vmax, R, \vpar) > 0$ small enough according to Proposition~\ref{prop: base summary}. Fix any $ 0 < \tpar < \tpar_*$. All constants and parameters below depend on these fixed values.

We are interested in looking at the quantitative stability of minimizers of the base energy $\En$ with a volume constraint: 
\begin{equation}\label{eqn: Base Energy Min 1}
 \inf\{ \ei(\W) + \tpar \tor(\W)  : \W \in \compset^*, |\W | = v\},
\end{equation}
By Proposition~\ref{prop: base summary}, the infimum in \eqref{eqn: Base Energy Min 1} is equal to the infimum $\minE({v})$ of the unconstrained problem \eqref{e: emin def}, and the two minimization problems are equivalent  
in the sense that if a set $\W \in \compset^*$ attains the infimum in \eqref{e: emin def}, then it has $|\W| = v$ and it attains the infimum in \eqref{eqn: Base Energy Min 1}, and conversely.
Let $\Min^*$ denote the collection of minimizers of this variational problem, which is nonempty by Proposition~\ref{prop: base summary}.
In Theorem~\ref{thm: quantitative stability}, we make the following assumptions about $\Min^*$:

In Case 1, assume that $\Min^*$ may be represented as $\{e(U)\}_{e \in \Isom}$ for a set $U \in \Min^*$, the boundary $\p U$ is smooth, and there exists a set center adapted to $U$ defined for sets in $\{\W : \En(\W) \leq \minE(v) + \d \}$  for $\d$ small enough. Recall that set centers are defined in Definition~\ref{def: set center} and that examples to keep in mind are those on simply connected. space forms in Examples~\ref{ex: functional rn}--\ref{ex: round sphere}.

In Case 2,  assume that $\Min^*$ contains only one element $U$ and $U \cc \Qb_R$ and that $\p U$ is smooth.

 More general situations, such as finitely many minimizers (modulo isometries in Case 1) may be considered with suitable modifications to the arguments here.
Given an open bounded set $\W$, let $d_*(\W)$ be the distance to the collection of minimizers as defined in Example~\ref{ex: general example}. In Case 1, we let $U_\W$ denote the unique minimizer with the same set center as $\W.$

\begin{theorem}\label{thm: quantitative stability}
	For $M$ and $\Min^*$ as above and $\e > 0$ small, assume that the inequality
	\begin{equation}\label{eqn: hp local stability}
	\ei(\W) + \tpar \tor(\W) - \minE \geq c d_*(\Omega)^2
	\end{equation}
	holds for all $\Omega \in \compset^*$ with $|\W| = v$ and $\p \W$ expressible as a $C^2$ normal graph over $\p U_\W$ with $C^{2, \a}$ norm bounded by $\e$. Then, up to possibly replacing $c$ with a smaller constant, \eqref{eqn: hp local stability} holds for all $\Omega \in \compset^*$ with $|\W| = v$.
\end{theorem}
\begin{proof}
	Suppose by way of contradiction that the theorem is false. We may find a sequence of sets $\{\W_j\}$ with $\W_j \in \compset^*$  and $|\W_j| = v$  such that 
	\begin{equation}\label{eqn: contradiction hp 1}
	\En(\W_j)= \ei(\W_j) + \tpar \tor(\W_j)   \leq  \minE +  \frac{d_j^2}{j}\,,
	\end{equation}
	where here and in the remainder of the proof, we let $d_j = d_*(\W_j)$. The quantity $d_*(\Omega)$ is bounded above uniformly  for any $\Omega \in \compset^*$,  so \eqref{eqn: contradiction hp 1} implies that
	\[
	\lim_{j \rightarrow \infty} \En(\W_j) = \minE(v).
	\]
	Applying Lemma \ref{l:globalsimpleeval} in Case 1 or Lemma \ref{l:emin} in Case 2, we see that $d_j \to 0$, and so there exists $U \in \Min^*$ such that, after passing to a subsequence, $|\W_j \triangle e_j(U)| \to 0$ as $j \to \infty$ for some $e_j \in \Isom$. 
	
	For each $j \in \mathbb{N}$, consider the minimization problem 
	\begin{equation}\label{eqn: minimization problem for penalized functional}
	\inf\{\Ep(\W) = \En(\W) + \err \nl(\W) : \Omega \in \compset^* \},
	\end{equation} 
	where we let  
	\begin{align*}
	\nl(\W) = \sqrt{d^4_j + \left(d_*(\Omega)^2 - d_j^2\right)^2} - d_j^2 
	\end{align*}
	where $\err < \err_0$ is a small parameter, and $\err_0= \err_0(v, \vmax, R, \vpar, \tpar)$ is chosen so that Theorem~\ref{thm: summary for minimizers} holds. 
	By Theorem~\ref{thm: summary for minimizers}, a minimizer $V_j$ of this variational problem exists and satisfies $|V_j| = v$. 
	Let us take  $\W_j$ as a competitor in \eqref{eqn: minimization problem for penalized functional} in order to relate the  deficits and distances of $\W_j$ and $V_j$. We find
	\begin{align}\label{eqn: competitor}
	\En(V_j)+ \err\1\sqrt{d_j^4 + \left(d_*(V_j)^2 - d_j^2\right)^2} - d_j^2\2
	& \leq \En(\W_j) \\
	\label{eqn: competitor2}
	& \leq \En(U) + \frac{d_j^2}{j}. 		
	\end{align}
	In \eqref{eqn: competitor2} we have 
	applied \eqref{eqn: contradiction hp 1}. Subtracting $\En(U) = \minE(v)$ throughout \eqref{eqn: competitor} and \eqref{eqn: competitor2} (and recalling that $\minE(v)\leq \En(V_j)$ by definition), we directly see that
	\begin{equation}\label{eqn: comparable deficits}
	\En(V_j) - \minE(v) \leq \En(\W_j) - \minE \to 0.
	\end{equation}
	Furthermore, \eqref{eqn: competitor2} tells us that 
	\[
	\sqrt{d_j^4 + \left(d_*(V_j)^2 - d_j^2\right)^2} \leq  \left(1+ \frac{1}{\tau  j}\right)d_j^2.
	\]
	After squaring both sides of this inequality, we find that $ (d_*(V_j)^2 - d_j^2)^2 \leq 3d_j^4/\err j$. This guarantees, first of all, that $V_j \not\in \Min^*$ for $j$ sufficiently large. Moreover, taking the square root of both sides, we find that for sufficiently large $j$, 
	\begin{equation}\label{eqn: comparable a}
	\frac{1}{2}d_j \leq d_*(V_j).
	\end{equation}
So, combining the contradiction hypothesis \eqref{eqn: contradiction hp 1} with the deficit and distance comparisons \eqref{eqn: comparable deficits} and \eqref{eqn: comparable a}, we see that
	\begin{equation}\label{eqn: contra a}
	\En(V_j) - \minE(v) \leq \En(\W_j) - \minE(v) \leq \frac{d_j^2}{j} \leq \frac{ 2d_*(V_j)^2}{j}\,.
	\end{equation}
	By Theorem~\ref{thm: summary for minimizers}, $V_j$ may be represented as a $C^{2, \a}$ normal graph over $\p U_j$ for some $U_j \in \Min$, with $C^{2, \a}$ norm controlled by $o_j(1)$. Moreover, the set centers of $U_j$ and $V_j$ converge, since $ 
		d(x_{U_j}, x_{V_j}) \leq C |U_j \triangle V_j| = o_j(1),$ which immediately implies that $|U_j \triangle U_{V_j}| = o_j(1)$ from the properties of set centers. Here $U_{V_j}$ is the unique minimizer with the same set center as $V_j$. Using the smoothness of the minimizers $U$ and the tubular neighborhood theorem, this means for $j$ large $\p U_j$ (and then also $\p V_j$) may be written as a $C^{2, \a}$ normal graph over $\p U_{V_j}$ with small $C^{2, \a}$ norm.
	Applying the hypothesis \eqref{eqn: hp local stability} to $V_j$, we see that
	\[
	c d_*(V_j)^2 \leq \En(V_j) - \minE,
	\]
	which for $j$ large is a contradiction.This completes the proof.
	%
	%
\end{proof}

Recall from the discussion in Section~\ref{ss:mainprob} that $d_*(\W)^2$ is comparable to
\[
\bar{d}_*(\W)^2 := |\W \triangle U_\W|^2 + \int |u_\W - u_{U_\W}|^2
\]
for $\W$ whose boundary is a (small) normal graph over $\p U$, and controls $\bar{d}_*(\W)^2$ in all cases. This leads to the following corollary, which in particular proves Corollary \ref{c:intro}.
\begin{corollary} Theorem \ref{thm: quantitative stability} remains valid with $\bar{d}_*$ in place of $d_*$.	
\end{corollary}

\section{Eigenfunctions, torsion functions, and the base energy}\label{s:basics}

This section collects some useful facts about eigenfunctions and torsion functions of bounded open subsets of $M$, as well some initial properties about the base energy $\En$ and its optimizers.  Section~\ref{ss:basics} is devoted to general properties of eigenfunctions and torsion functions. 

Section~\ref{ss: key property} contains Proposition~\ref{lem:key}, the key estimate that will allow us to control the nonlinear perturbation in the main energy in the remainder of the paper. This lemma estimates the difference of eigenfunctions on nested domains in terms of their eigenvalue and torsional rigidity differences.  Fundamentally, this lemma highlights why we are able to establish existence and regularity of minimizers of the main energy when the coefficient $\tpar$ in front of the torsional rigidity is positive (and suitably small), but not when $\tpar= 0$ and the leading term in the base energy is only the first eigenvalue.

In Section~\ref{ss:base energy mins}, we establish the existence and initial properties of minimizers of the base energy in $\Qb_R$, while Section~\ref{ss:globalbase} extends some of these properties of minimizers of the base energy globally.

We recall the set $\compset = \compset_{R,\vmax}$ defined in \eqref{e: compset}.
Recall that we always implicitly assume that $\vmax$ and $R$ are chosen so that  $\vmax \leq |\Qb_R|$  and, when $|M|$ is finite, so that $\vmax <|M|$. Finally, we remind the reader that, in addition to the specified dependence in each theorem statement, all constants will depend on $(M,g)$.

\subsection{Eigenfunctions and torsion functions}\label{ss:basics}
Our first lemma collects some basic facts about eigenfunctions and torsion functions. The proof is elementary and so we omit it.

\begin{lemma}\label{l:basic} Let $\W \subset M$ be a bounded open set. Then the following properties hold.
	\begin{enumerate}
		\item $w_\W$ is nonnegative and satisfies $\Lap w_\W \geq - 1_{\W}$ on $M$ in the sense of distributions.
		\item $\Lap w_\W = -1$ on $\W$.
		\item $\int_\W |\nabla w_\W |^2 =\int_\W w_\W =-2 \tor (\W)$.
		\item If $\l_1(\W) < \l_2(\W)$, the normalized eigenfunction $u_\W$, up to replacing with $- u_\W$, is nonnegative and $\Lap u_\W \geq - \l_1(\W) u_\W$ on $M$ in the sense of distributions.
		\item If $u_\W$ is as in (4), $\Lap u_\W = - \l_1(\W) u_\W$ on $\W$.
	\end{enumerate}
\end{lemma}

Next, we show that eigenfunctions of a domain $\W \subset M$ are bounded uniformly in terms of an eigenvalue upper bound.
\begin{lemma}[Eigenfunction upper bound]\label{l:efbdd}
For each $\l_0 > 0$, there is a constant $C_1 = C_1(\l_0)$ such that the following holds. Let $\W$ be an open bounded set and fix $\l \leq \l_0$. For any function $u \in H^1_0(\W)$ with $\|u\|_{L^2} = 1$ satisfying $\Lap |u| \geq - \l |u|$ on $\W$,
\[
	|u| \leq C_1.
\]
In particular, this applies to any eigenfunction of $\W$. 
\end{lemma}

\begin{proof}
	Let $K(x, y, t)$ be the heat kernel on $M$ (see \cite{Davies}), and set
	$
		C_1 = e^{\l_0}\sup_{x \in M} \sqrt{K(x, x, 2)}.
	$
	Since $(M,g)$ has bounded geometry, $C_1$ is finite.
	Extending $u$ by zero to be defined on all of $M$, the function $e^{-t \l}|u(x)|$ is a subsolution to the heat equation on $M$. Using the comparison principle,
	\begin{align*}
		e^{-t \l}|u(x)| \leq \int K(x, y, t) |u(y)| dy
		& \leq \left(\int K^2(x, y, t)dy\right)^{1/2} \|u\|_{L^2} = \sqrt{K(x, x, 2 t)}.
	\end{align*}
	Set $t = 1$ to conclude.
\end{proof}

Next, the torsion function of an open bounded set is bounded in terms of its torsional rigidity.

\begin{lemma}[Torsion function upper bound]\label{l:torsionfunctionbdd} For each $T>0,$ there exists a constant $C_2= C_2(T)$ such that the following holds. Let $\W$ be an open bounded set with $\tor (\W) \geq -T$. 
	\[
		w_\W \leq C_2.
	\]
\end{lemma}

\begin{proof}	
In local coordinates on a ball $B_r(x)$ around any point $x\in M$, we have that $w_\W \geq 0$ and $
		\Lap w_\W \geq - 1,$
	so applying the local maximum principle (\cite{GT}, Theorem 8.17) gives 
	\[
		\sup_{B_{r/2}(x)} w_\W^2 \leq C(r, x)\bigg[\int_{B_{r}(x)} |\grad w_\W|^2 + 1\bigg] \leq C(r, x) [-\tor(\W) + 1].
	\]
In the second inequality we used Lemma~\ref{l:basic}(3).	Since $(M,g)$ has bounded geometry,  the values of $r$ and the constant $C(r, x)=C$ may be taken uniformly for all $x \in M$. We conclude by letting $C_2 = C[T +1]$.
\end{proof}
The next lemma shows that the eigenfunction is controlled by the torsion function.
\begin{lemma}[Torsion function controls eigenfunction] \label{l:torsionbound}
Let $\W $ be an open bounded set with $|\W| < |M|$ such that $\ei(\W) < \l_2(\W)$. Then there is a constant $C_3$ depending only on $C_1$ of Lemma~\ref{l:efbdd} and $\ei(\W)$ such that
	\[
		u_\W \leq C_3 w_\W.
	\]
\end{lemma}

\begin{proof}
	This follows directly by the maximum principle and Lemma \ref{l:efbdd}: $-\Lap u_\W \leq \ei(\W)C_1 = - \Lap \ei(\W)C_1 w_\W$ on $\W$, and both are in $H^1_0(\W)$, giving $u_\W \leq C_1 \l_1(\W) w_\W$ on $\W$.
\end{proof}

As a consequence of this lemma, if we write $\Lap u_\W =- \l_1(\W) u_\W + \mu$ in the sense of distributions, with $\mu$ a nonnegative measure, and similarly $ \Lap w_\W =- 1_{\W} + \nu$ where $\nu$ is a nonnegative measure, then $\mu \leq C \nu$. The opposite inequality $w_\W \leq C u_\W$ is much more subtle, and will be discussed for \emph{minimizing} $\W$ in Proposition~\ref{l:bdryharnackef}.

We now establish two forms of a Poincar\'{e} inequality on the sets $Q_R$ defined at the beginning of Section~\ref{s:setup}; the key point here is that they depend only on the parameters $R$ and $\vmax$. 

%

\begin{lemma}[Poincar\'{e} Inequality]\label{l:poincare}  

Fix $\vmax < |M|$ and $R>0$, and let $X$ be the completion of the space $ \{u \in C^\infty_0(\Qb_R) : |\{|u| > 0 \}| \leq \vmax\}$ with respect to the seminorm $\| \nabla u\|_{L^2(\Qb_R)}$.
Then the embedding 
\[
X  \rightarrow L^p(\Qb_R)
\] is continuous for $p \in [1, \frac{2n}{n-2}]$ and compact for $p < \frac{2n}{n-2}$. If $\vmax < |\Qb_R|$, then the same is true as well letting $X$ be the completion of the space $\{ u \in C^\infty(\Qb_R) : |\{|u| > 0 \}| \leq \vmax, \ \| \nabla u\|_{L^2(\Qb_R)} <+\infty\}$ with respect to the seminorm $\| \nabla u\|_{L^2(\Qb_R)}$.
\end{lemma}

\begin{proof}
 Let $H^1(Q_R)$ be the completion of the space 
$\{ u\in C^{\infty}(\Qb_R) : \| u\|_{H^1(\Qb_R)} < +\infty \}$ 
with respect to the norm $\| u\|_{H^1(\Qb_R)}:= \| u\|_{L^2(\Qb_R)} + \| \nabla u\|_{L^2(\Qb_R)}$. 
It is well known (see for instance \cite[Theorem 10.2]{HebeyBook} or  \cite[Theorems 2.30 and 2.34]{AubinBook}) that on $Q_R$ (or more generally, any compact Riemannian manifold with boundary), $H^1(\Qb_R)$ embeds continuously into $L^p(\Qb_R)$ for  $p \in [1, \frac{2n}{n-2}]$ and the embedding is compact for $p < \frac{2n}{n-2}$. 

Suppose by way of contradiction that the second embedding stated in the lemma fails for $p=2n/(n-2)$. Take a sequence of functions $\{u_k\}$ with $\| u_k\|_{L^p(\Qb_R)} = 1$ and $| \{ u_k >0\}| \leq \vmax$ such that $\| \nabla u_k\|_{L^2(\Qb_R)}\to 0$. We see from  H\"{o}lder's inequality that  $\| u_k\|_{L^2(\Qb_{R})} \leq C(\Qb_R)$. So, from the continuous embedding $H^1(\Qb_R)\hookrightarrow L^p(\Qb_R)$ we find that $u_k \to c$ in $H^1(\Qb_{R})$, $L^p(\Qb_R)$ and pointwise a.e. in $\Qb_{R}$ for some constant $c \in \mathbb{R}$. On one hand, $\| c\|_{L^p(\Qb_R)}=1$, while on the other hand, Fatou's lemma implies that $|\{|c|>0\}| \leq \vmax$  and so $c= 0$. We reach a contradiction. 

Now, for $p < 2n/(n-2)$, the continuous embedding  then follows by H\"{o}lder's inequality, while the compactness of the embedding follows from the compact embedding of  $H^1(\Qb_R)$ into $L^p(\Qb_R)$ and the continuous embedding of $\{u \in \dot{H}^1(\Qb_R) : |\{|u| > 0 \}| \leq \vmax\}$ into $L^2(\Qb_R)$ that we have just established.

We prove the first embedding of the lemma analogously. If $M$ is compact, then the argument above carries over directly. 
If $M$ is noncompact, we again argue by way of contradiction and take a sequence of functions with $\| u_k\|_{L^p(\Qb_R)} = 1$ and $\| \nabla u_k\|_{L^2(\Qb_R)}\to 0$. 
Extend each $u_k$ by zero to be defined on $\Qb_{2R}$. By H\"{o}lder's inequality, $\| u_k\|_{L^2(\Qb_{2R})} \leq C$ and so from the continuous embedding $H^1(\Qb_{2R})\hookrightarrow L^p(\Qb_{2R})$ we see that $u_k \to c$ in $H^1(\Qb_{2R}), L^p(\Qb_{2R})$ and pointwise a.e. in $\Qb_{2R}$. On one hand, $\| c\|_{L^p(\Qb_{2R})} =1,$ while on the other hand, $u_k \equiv 0$ on $\Qb_{2R}\setminus \Qb_R$ and so we see that $c=0$. We reach a contradiction. Again, the continuous embedding for $p < 2n/(n-2)$ then follows by H\"{o}lder's inequality, while the compactness of the embedding follows from the compact embedding of  $H^1_0(\Qb_R)$ embeds continuously into $L^p(\Qb_R)$ and the continuous embedding of $\dot{H}^1_0(\Qb_R)$ into $L^2(\Qb_R)$ that we have just established. 
\end{proof}

As an immediate consequence of Lemma~\ref{l:poincare}, we may bound the torsional rigidity $\tor(\W)$ from below uniformly for all domains in $\compset = \compset_{R, \vmax}$, the class of admissible domains defined in \eqref{e: compset}. Naturally, if $M$ has bounded diameter, then the lower bound is independent of $R$.

\begin{corollary}[Torsional rigidity lower bound]\label{c:tornotminusinf}
	Fix $R>0$ and $\vmax < |M|$. There is a constant $C(R, \vmax) < \infty$ such that for any $\W \in \compset$,
	\[
		\tor(\W) \geq - C(R, \vmax).
	\]
\end{corollary}

\begin{proof}
	Fix $\W \in \compset$ and let $w_\W$ be the corresponding torsion function, so that $\tor(\W) = \int_\W \frac{1}{2} |\nabla w_\W|^2 - w_\W$. Then for any $\e > 0$,
	\[
	\int |w_\W|  \leq C(R, \vmax) \|\grad w_\W\|_{L^2} \leq \e \int |\grad w_\W|^2 + C(\e, R, \vmax),
	\]
	where the first inequality used Lemma \ref{l:poincare} on $w_\W$ extended by $0$ and the fact that $|\{ w_\W > 0\}| \leq |\W| \leq \vmax$. Take $\e = \frac{1}{4}$ and reabsorb to get
	\[
	\tor(\W) = \int_\W \frac{1}{2} |\nabla w_\W|^2 - w_\W \geq - C(R, \vmax) + \int_\W \frac{1}{4} |\nabla w_\W|^2 .
	\]
	This concludes the proof.
\end{proof}

\begin{remark}\rm{
Notice that from  Corollary \ref{c:tornotminusinf}, for any $\Omega \in \compset =\compset_{R,\vmax}$, the constant $C_2$ in Lemma~\ref{l:torsionfunctionbdd}  depends only on $R$ and $\vmax.$ 
}
\end{remark}

The final lemma of this section is an elementary fact about the energy of the difference between any normalized nonnegative function and the first eigenfunction.
\begin{lemma}\label{l:poincarestab}
	Let $\W$ be an open bounded set with $\ei(\W) + \a \leq \l_2(\W) $ for some $\a > 0$, and $v \in H^1_0(\W)$ with $\int v^2 = 1$ and $v\geq 0$. Then
	\[
		\int |\grad (u_\W - v)|^2 \leq \Big(1 + \frac{2 \ei(\W)}{\a}\Big)\Big[ \int |\grad v|^2 - \ei(\W)\Big].
	\]
\end{lemma}
\begin{proof}
	Let us use the shorthand $\lambda_1 = \lambda_1(\W)$ and $\lambda_2 = \lambda_2(\W)$, and write $v = \beta u_\W +v_\perp $ where $\beta = \int u_\W v$. Because $\int u_\W v_\perp =0$ by definition, we observe that 
	\begin{align*}
	\lambda_2(1-\beta^2)  =\lambda_2 \int v_\perp^2 &\leq \int |\nabla v_\perp|^2 = \int |\nabla v|^2-  \lambda_1\beta^2 = \int |\nabla v|^2 - \lambda_1 + \lambda_1(1-\beta^2).
	\end{align*}
	In particular, $\alpha(1-\beta)^2 \leq \int |\nabla v|^2 - \lambda_1$ and $\int |\nabla v_\perp|^2 \leq (1 + \lambda_1/\alpha)(\int |\nabla v|^2 - \lambda_1)$. So, since $u_\W - v = (1-\beta) u_\W + v_\perp $ and $0 \leq \beta \leq 1$, we see that 
	\[
	\int |\nabla (u_\W - v)|^2 = (1-\beta)^2 \lambda_1 + \int| \nabla v_\perp|^2  \leq \Big( 1+ \frac{2\lambda_1}{\alpha}\Big) \Big( \int |\nabla v|^2 - \lambda_1 \Big)\,.
	\]
	This concludes the proof.
 \end{proof}

\subsection{The key estimate on the nonlinear perturbation}\label{ss: key property}
In order to study the existence and regularity of minimizers of the main energy $\Ep$, we will require some sharp estimates on how an admissible nonlinearity $\nl$ changes under a change in domain from $\W$ to $\W'$. From assumption \ref{a:nllip} in Definition~\ref{def: nl}, this is controlled by the sum of the symmetric difference $|\W \triangle \W'|$, which is easy to estimate, and the $L^1$ difference of eigenfunctions $\int |u_\W - u_{\W'}|$, which is not significantly more difficult to estimate. In fact, we do not know how to control this quantity in terms of the eigenvalue difference $|\ei(\W') - \ei(\W)|$ alone, at least in the case of outward perturbations. The following lemma does bound $\int |u_\W - u_{\W'}|$, but the difference $|\tor(\W) - \tor(\W')|$ appears on the right-hand side of the estimate as well.  This is the fundamental reason for introducing the torsional rigidity term into the main energy functional $\Ep$.

This lemma will be the crucial ingredient in proving the main estimates of Sections \ref{s:lb} and \ref{s:ub} and establishing our main existence result in Section \ref{s:exist}. At least for the estimates, the specific form of the lemma, with \emph{linear} dependence on $|\tor(\W) - \tor(\W')|$ and $|\ei(\W') - \ei(\W)|$, is essential; the (much) simpler sublinear versions of this inequality are insufficient.

\begin{proposition} \label{lem:key} Fix $0 < \vmax <|M|$,  $\a>0$ and $\lambda_0>0$. Let $\W \ss  \W'$ be bounded open sets with   $|\W'|\leq \vmax$, $\l_2(\W') \geq \ei(\W') + \a$, and $\lambda_1(\W) \leq \lambda_0$. Let $f :M \to \R$ be a bounded function. Then  there exists a constant $C_4 = C_4(\a, \vmax, \lambda_0)$ such that 
	\begin{equation}\label{e:key}
		\Big| \int f (u_{\W'} - u_\W) \Big| \leq C_4 \|f\|_{L^\infty} \left(\tor(\W) - \tor(\W') + \ei(\W) - \ei(\W')\right).
	\end{equation}
\end{proposition}

In the case that the first eigenvalue of $\W$ is not simple, then Proposition~\ref{lem:key} holds trivially for any first eigenfunction $u_\W$. To see this, note that $\ei(\W) \geq \ei(\W')$ and $\l_2(\W) \geq \l_2(\W')$ from $\W \ss \W'$. So, if $\ei(\W)$ is not simple, then $\ei(\W) = \l_2(\W)$ and so $\ei(\W) - \ei(\W') \geq \a$, meaning the estimate here holds automatically for \emph{every} normalized first eigenfunction $u_\W$ (using Lemma \ref{l:efbdd}). So, in the proof below, we will assume that $\ei(\W) < \l_2(\W)$, in which case $u_\W$ is uniquely defined.

\begin{proof}	
	We first consider the case of $\W$ compactly contained inside $\W'$ and having smooth boundary, but all constants will be independent of the nature or regularity of $\p \W$. Let $\beta = \int f u_{\W'}$ and  write $f = \beta u_{\W'} + f_\perp$. We estimate the contributions of $\beta u_{\W'}$ and  $f_\perp$ to the left-hand side of \eqref{e:key} separately. To estimate the contribution from the first term, use the normalization of the eigenfunctions $u_\W, u_{\W'}$:
	\begin{align*}
		\int \beta u_{\W'}(u_{\W'} - u_\W) &= \beta \int u_{\W'}^2 - u_{\W'} u_{\W} = \frac{\beta}{2} \int u_{\W'}^2 + u_{\W}^2 - 2 u_{\W'} u_{\W}=\frac{\beta}{2} \int |u_\W - u_{\W'}|^2.
	\end{align*}
	Applying Lemma \ref{l:poincarestab} and assuming without loss of generality that $\alpha \leq  \lambda_1(\W')$, we have
	\[
		\int |u_\W - u_{\W'}|^2 \leq \frac{1}{\ei(\W')} \int |\grad (u_\W - u_{\W'})|^2 \leq \frac{3}{\a} \left[\int |\grad u_\W|^2 - \ei(\W') \right] = \frac{3}{\a}\left[\ei(\W) - \ei(\W')\right]
	\]
	and therefore
	\begin{equation}  \label{e:1fbound}
	 \Big|\int \beta u_{\W'}(u_{\W'} - u_\W)\Big| \leq \frac{3}{\alpha} \|f\|_{L^2} \left[\lambda_1(\W) - \lambda_1(\W')\right].
	\end{equation}

	Now, the remainder $f_\perp$ has $\|f_\perp\|_{L^2} \leq \|f\|_{L^2} \leq \sqrt{\vmax} \|f\|_{L^\infty}$ and $\int f_\perp u_{\W'} = 0$. This means we may solve for the potential
	\begin{equation}\label{eqn: potential pde}
	\begin{cases}
		- \Lap q -\ei(\W') q = f_\perp & \text{ on } \W'\\
		q = 0 & \text{ on } \p \W',
	\end{cases}
	\end{equation}
	which has a unique solution $q \in H^1_0(\W')$ with $\int q u_{\W'} = 0$, and enjoys the estimate $\|q\|_{H^1_0(\W')} \leq C(\a)\|f_\perp\|_{L^2}$ (this follows from the Fredholm alternative). From Lemma \ref{l:efbdd} we have that $\|u_{\W'}\|_{L^\infty} \leq C$, so
	\[
		\|f_\perp\|_{L^\infty} \leq \|f\|_{L^2} \|u_{\W'}\|_{L^\infty} + \|f\|_{L^\infty} \leq C \|f\|_{L^\infty}.
	\]
	So, from \eqref{eqn: potential pde}, we see that $- \Lap q - \ei(\W') q \in L^{\infty}$, we have from applying \cite[Theorem 8.17]{GT} in charts that
	\[
		\|q\|_{L^\8} \leq C \big[\|f_\perp\|_{L^\infty} + \|q\|_{L^2}\big] \leq C \|f\|_{L^\infty}.
	\]
	Now rewriting the PDE as $- \Lap q = \ei(\W') q + f_\perp \in L^\8$, applying the comparison principle with a multiple of the torsion function $w_{\W'}$ gives
	\[
		|q| \leq C\|f\|_{L^\infty} w_{\W'}.
	\]
	This is the key fact about $q$; note also that from standard elliptic regularity $q$ lies in $C(V) \cap W^{2, 2}(V)$ for any $V$ compactly contained in $\W'$.
	
	We proceed in our estimate:
	\begin{equation}\label{e:keyproof2}
		\int f_\perp (u_{\W'} - u_\W) = \int f_\perp u_\W = \int\left(-\Lap q - \ei(\W') q\right) u_\W.
	\end{equation}
	From Lemma \ref{l:basic}, we have $\Lap u_\W = -\ei(\W) u_\W + \mu$ in the sense of distributions, where $\mu$ is a nonnegative measure supported on $\p \W$: i.e. for any function $\phi \in C_c^\infty(M)$, we have that
	\begin{equation}\label{e:keyproof1}
		\int \left(\Lap \phi + \ei(\W)\phi\right) u_\W = \int \phi d\mu.
	\end{equation}
	Take an open set $V$ with $\W \cc V \cc \W'$ and a sequence $\phi_k \in C^\infty_c(M)$ with $\|\phi_k - q\|_{W^{2, 2}(V) \cap C(V)} \rightarrow 0$; then passing to the limit in \eqref{e:keyproof1} leads to
	\[
		\int \left(\Lap q + \ei(\W)q\right) u_\W = \int q d\mu.
	\]
	Continuing on and applying this to \eqref{e:keyproof2},
	\[
	\int q \left(-\Lap u_\W - \l_1(\W') u_\W\right) = \int -q d\mu + \left(\l_1(\W) - \l_1(\W')\right)\int q u_\W.
	\]
	The second term is controlled by
	\[
		\left(\l_1(\W) - \l_1(\W')\right)\,\Big|\int q u_\W \Big| \leq \left(\l_1(\W) - \l_1(\W')\right)\|q\|_{L^2}\,\|u_\W\|_{L^2} \leq C (\l_1(\W) - \l_1(\W')) \|f\|_{L^\infty},
	\]
	so we need only focus on the first. Recall from Lemma \ref{l:torsionbound} that writing $\Lap w_{\W} = -1_{\W} + \nu$ on $\W'$ for a nonnegative measure $\nu$, then $\mu \leq C \nu$. As such, 
	\[
	\Big|\int q d\mu\Big| \leq \int |q|d \mu \leq C\|f\|_{L^\infty}\int w_{\W'} d\nu = C\|f\|_{L^\infty}\int \left( w_{\W'} + w_{\W}\right)  d\nu,
	\]
	in the final identity using that $w_\W$ vanishes on the support of $\nu$.
		
		Let $\psi$ be a $C^2_c(M)$ function with $\psi = w_\W + w_{\W'}$ on $\bar{\W} \cc \W'$; this exists as $\W$ is smooth and so $w_\W$ is smooth up to the boundary $\p \W'$, and then by using e.g. Whitney's extension theorem. Applying the distributional definition of $\nu$, we have that
		\[
			\int \left( w_{\W'} + w_{\W}\right)  d\nu = \int \psi d\nu = \int w_\W \Lap \psi + 1_\W \psi = \int -2 w_\W + 1_\W (w_\W + w_{\W'}),
		\]
		using that $ \Lap \psi = \Lap (w_\W + w_{\W'}) = -2$ on the support of $w_\W$ in the last step. Using the positivity of $w_{\W'}$,
		\[
			\int -2 w_\W + 1_\W (w_\W + w_{\W'}) \leq \int w_{\W'} - w_\W.
		\]

	Recall that $\tor(\W)  = - \frac{1}{2}\int |\n w_{\W}|^2 = - \frac{1}{2} \int w_\W$ and similarly for $\W'$, so we have shown that
	\[
		\Big|\int f_\perp (u_{\W'} - u_\W)\Big| \leq C\|f\|_{L^\infty}[\ei(\W) - \ei(\W') +  \tor(\W) - \tor(\W') ].
	\]
	Together with \eqref{e:1fbound} this gives the conclusion.
	
For the case of a general $\W$ without restrictions on the smoothness or location of $\p \W$, let $\W_k \ss \W_{k+1} \cc \W$ be an exhaustion of $\W$ by smooth sets (i.e. $\W = \cup_k \W_k$; see e.g. \cite[Chapter 5, Theorem 4.20]{EE}). Then apply \eqref{e:key} to $\W_k$. We have that $\ei(\W_k) \geq \ei(\W)$ and $\tor(\W_k) \geq \tor(\W)$ from set inclusion. Take first eigenfunctions $u_{\W_k} \rightharpoonup u \in H^1_0(\W)$ weakly (and therefore strongly in $L^2(\W)$) and locally in $C^2$ topology for some $u \in H^1_0(\W)$ along a subsequence, and similarly $w_{\W_k} \rightharpoonup w \in H^1_0(\W)$. Passing the PDEs satisfied by these to the limit, we see that $- \Lap u  = \l u$ and $-\Lap w = 1$ on all of $\W$, where $\l = \lim_k \l(\W_k) \geq \ei(\W)$ is some number; moreover, $u, w \geq 0$ on $\W$ and $\int u^2 = 1$. This implies that $u$ is an eigenfunction of $\W$ and $\l$ is an eigenvalue; however, as $\int u u_\W > 0$ it must be the case that $u = u_\W$ and $\l = \ei(\W)$ (recall we are assuming that $\ei(\W)$ is simple). Similarly, $w_\W = w$ as the solutions to the corresponding PDE are unique. This means that 
	\[
		\tor(\W_k) = - \frac{1}{2}\int w_{\W_k} \rightarrow - \frac{1}{2} \int w_{\W} = \tor(\W)
	\]
	while
	\[
		\int f u_{\W_k} \rightarrow \int f u_\W,
	\]
	both using convergence in $L^2$. In particular, both sides of \eqref{e:key} pass to the limit.
\end{proof}

\subsection{Base energy minimizers in $\Qb_R$}\label{ss:base energy mins}
We now move toward minimizing the base energy $\En$ among subsets of a fixed $\Qb_R$. In this section we establish existence, volume bounds and connectedness properties of minimizers in the class $\compset = \compset_{R, \vmax}.$

We start with an existence and compactness theorem for minimizers. The statement is formulated in a way to give information on $w_\W$ if and only if $\tpar > 0$. We recall that the base energy $\En$, the collection of sets over which we minimize $\compset$, and the collection of minimizers $\Min$ are defined in \eqref{e:baseenergy}, \eqref{e: compset}, and \eqref{eqn:min} respectively. Recall from \eqref{e: emin def} that $\minE = \minE(v, \vmax, \eta, \tpar,  R)$ denotes the infimum of the base energy.

\begin{lemma}\label{l:emin}  Fix  $R>0$ and $v, \vmax $ with  $0<v < \vmax$, as well as  the parameters $\eta>0$ and $ \tpar >0$ in the base energy. Then:
	\begin{enumerate}
		\item $\minE > -\infty$.
		\item Let $\W_k \in \compset$ and $u_k, w_k \in H^1_0(\W_k)$ with $\int u_k^2 = 1$ and
		\begin{equation}\label{e: min sequence}
		\lim_k \int |\grad u_k|^2 + \tpar \int \frac{1}{2}|\grad w_k|^2 - w_k + \fv(|\W_k|) = \minE.
		\end{equation}
		Then there is a subsequence $\W_{k_j}$ and a set $\W \in \compset$ such that
		\[
		\|u_{k_j} - u_\W \|_{H^1_0(\Qb_R)} + \sqrt{\tpar} \|w_{k_j} - w_\W \|_{H^1_0(\Qb_R)} + |\W \triangle \W_{k_j}|  \rightarrow 0,
		\]
		where $u_{\W}$ is a normalized first eigenfunction of $\W$ and $w_\W$ is the torsion function of $\W$.
		\item $\Min$ is nonempty, and for any $\W \in \Min$, the functions $u_\W$ and $\sqrt{\tpar} w_\W$ are  continuous on $M$ after extension by $0$ to be defined on all of $M$ (for any choice of $u_\W$ normalized first eigenfunction).
	\end{enumerate} 
\end{lemma}

\begin{proof}	
	Notice that $\int |\grad u_k|^2 \geq 0$ and $\fv(|\W_k|) \geq -\eta v$ by definition. Moreover,  $\tor(\W)$ is bounded from below uniformly in $\W \in \compset$ by Corollary \ref{c:tornotminusinf}. Summing these terms and taking the infimum shows that $\minE >-\infty,$ thus establishing (1).
	\\

{\it Step 1: Basic compactness.}	Let $\W_k$ be a minimizing sequence as in claim (2) of the lemma. 
As we noted above,  each of the terms $\int |\grad u_k|^2, \tpar \int \frac{1}{2}|\grad w_k|^2 - w_k, $ and $\fv(|\W_k|)$ in the energy are bounded from below individually.
 Because $\minE$ is also bounded from above,  these three terms are bounded above individually as well. This immediately gives that $\|u_k\|_{H^1_0(\Qb_R)}$ is uniformly bounded, while applying Lemma \ref{l:poincare} as in the proof of Corollary \ref{c:tornotminusinf} gives that
	\[
	C \geq \tpar \int \frac{1}{2}|\grad w_k|^2 - w_k \geq \tpar \left[ - C(R, \vmax) + \frac{1}{4}\int |\grad w_k|^2 \right],
	\]
	giving $\sqrt{\tpar} \|\grad w_{k}\|_{L^2(\Qb_R)} \leq C$ and so $\sqrt{\tpar} \|w_{k}\|_{H^1_0(\Qb_R)} \leq C$ from Lemma \ref{l:poincare} again.

	Passing to subsequences, we have that $u_{k} \rightarrow u$ weakly in $H^1_0(\Qb_R)$, strongly in $L^2$, and a.e., and similarly for $ w_{k} \rightarrow w$ when $\tpar > 0$. This implies that $1_{\{ |u| + \tpar |w| > 0 \}} \leq \liminf_k 1_{\{ |u_{k}| + \tpar |w_{k}| > 0 \}}$, so by Fatou's lemma $|\{ |u| + \tpar |w| > 0 \}| \leq \liminf_k |\W_k|$ and
	\[
		\fv(|\{ |u| + \tpar |w| > 0 \}|) \leq \liminf_{k \rightarrow \infty} \fv(|\W_k|).
	\]
	Using the lower semicontinuity of the norm under weak convergence, we have that
	\begin{equation}\label{eqn: LSC a}
		\int |\grad u|^2 + \tpar \int \frac{1}{2}|\grad w|^2 - w + \fv(|\{ |u| + \tpar |w| > 0 \}|) \leq \minE
	\end{equation}
	and $\int u^2 = 1$. 	From Lemmas \ref{l:efbdd} and \ref{l:torsionfunctionbdd}, we have $|u_k|\leq C_1$ and $|w_k|\leq C_2$; these pass to the limit to give $|u| + \sqrt{\tpar} |w|\leq C$. 
	\\
	
	{\it Step 2: Continuous representatives. }We will now show that $u, w$ admit continuous representatives; this will imply that the set $\W := \{ |u| + \tpar |w| > 0 \}$  is open (and hence in $\compset$), $u$ is a first eigenfunction of $\W$, $w = w_\W$, and the inequality in \eqref{eqn: LSC a} is an equality.\\

{\it Step 2a. }	To this end, we first claim that $u$ and $w$ have the following ``almost minimality'' property. Fix a small ball $B_{2r}(x)$, $x \in \Qb_R$. We claim that 
		\begin{equation}\label{eqn: almost min}
	\int |\grad u|^2 + \tpar\int \frac{1}{2}|\grad w|^2 - w \leq \frac{\int |\grad a|^2}{\int a^2} + \tpar \int \frac{1}{2}|\grad b|^2 - b + Cr^n
	\end{equation}
	for any pair of functions $a, b$ with $a=u$ and $b=w $ on $B_r(x)^c$ and such that $u -a$ and $w - b$  are in $H^1_0(B_r(x) \cap \Qb_R)$.
	
	Indeed, let $\psi$ be a smooth cutoff function which is $1$ on $B_r(x)$ and supported on $B_{2r}(x)$. Let $\W_k' = \W_k \cup B_{r}(x)$ and  define the functions $a_k, b_k \in H^1_0(\W_k')$ by  
	\begin{align*}
	a_k &= \psi a + (1 - \psi) u_{k},\\
	b_k &= \psi b + (1 - \psi) w_{k}.
	\end{align*} Note that $a_k \to a$ and $b_k \to b$ in $L^2$. Moreover,
	$
		\fv(|\W'_k|) \leq \fv(|\W_k|) + C r^n.
	$
	So, we may use $a_k/\|a_k\|_{L^2}$ as a competitor for $u_{\W'_k}$ and $b_k$ for $w_{\W'_k}$ inside of \eqref{e:evalue} and \eqref{e:torsion} to give
	\[
		\minE \leq \En(\W'_k) \leq \frac{\int |\grad a_k|^2}{\int a_k^2} + \tpar \int \frac{1}{2}|\grad b_k|^2 - b_k + \fv(|\W_k|) + C r^n.
	\]
	Recalling that $u_k, w_k$ were chosen to satisfy \eqref{e: min sequence}, this gives that
	\begin{equation}\label{e:emin1}
		0 \leq \liminf_{k \rightarrow \infty} \frac{\int |\grad a_k|^2}{\int a_k^2} - \int |\grad u_k|^2 + \tpar \left[\int \frac{1}{2}|\grad b_k|^2 - b_k  - \int \frac{1}{2}|\grad w_k|^2 - w_k\right] + C r^n,
	\end{equation}
	We first focus on the term in brackets. We have
		\begin{align*}
		\int \frac{1}{2}&|\grad b_k|^2   - \int \frac{1}{2}|\grad w_k|^2 = \frac{1}{2}\int |\grad w_{k}|^2 \left[( 1 - \psi)^2 - 1\right] + \psi^2 |\grad b|^2 + 2 \psi (1 - \psi) \grad b \cdot \grad w_{k}   + o_k(1),
	\end{align*}
%
	where the $o_k(1)$ term contains all error terms containing $\grad \psi$ and can be seen to converge to $0$ using the weak convergence of $\grad w_k \rightarrow \grad w$, the strong convergence of $w_k \to w$, and the fact that $b=w$ on the support of $\grad \psi$. Therefore, taking limits in the term in brackets in \eqref{e:emin1}, since $w =b$ when $1-\psi$ is nonzero, we have
	\begin{align*}
	\limsup_{k \rightarrow \infty} &\int \frac{1}{2}|\grad b_k|^2 - b_k  - \int \frac{1}{2}|\grad w_k|^2 - w_k\\
	& \leq \frac{1}{2}\int |\grad w|^2 \left[( 1 - \psi)^2 - 1\right] + \psi^2 |\grad b|^2 + 2 \psi (1 - \psi) |\grad w|^2  + \int w - b\\
	& =\frac{1}{2} \int \psi^2 [|\grad b|^2 - |\grad w|^2] + \int w - b\\
	& = \int \frac{1}{2}|\grad b|^2 - b - \int \frac{1}{2}|\grad w|^2 - w
	,
	\end{align*}
	In the inequality we have used that $\int \phi |\grad w|^2 \leq \liminf_{k \rightarrow \infty} \int \phi |\grad w_{_k}|^2$ for $\phi \geq 0$ from weak convergence (and setting $\phi = 1 - (1 - \psi^2)$).  
	A similar computation for the $a_k$ terms gives
	\[
	\limsup_{k \rightarrow \infty} \frac{\int |\grad a_k|^2}{\int a_k^2} - \int |\grad u_k|^2 \leq \frac{\int |\grad a|^2}{\int a^2} - \int |\grad u|^2.
	\]
	Substituting these conclusions into \eqref{e:emin1}, we arrive at \eqref{eqn: almost min}, thus proving the claim.
\\

	{\it Step 2b.} Now let $a$ be the harmonic replacement of $u$ on $B_r(x) \cap \Qb_R$ (i.e. the minimizer of $\int_{B_r(x) \cap \Qb_R} |\grad a|^2$ with data $u$ on $\partial [B_r(x) \cap \Qb_R]$), and $b$ the harmonic replacement of $w$. Then
	\[
	\int_{B_r(x)} |\grad u|^2 + \tpar \frac{1}{2} |\grad w|^2 \leq C r^n + \int_{B_r(x)} |\grad a|^2 + \tpar \frac{1}{2} |\grad b|^2,
	\]
	estimating the non-gradient terms using $|u|, |w|, |a|, |b| \leq C$ and absorbing them into the $Cr^n$ term. Rewriting and using that   $\int \langle \grad a, \grad (u - a)\rangle  = 0$  from the equation on $a$ (and similarly with $b$),
	\[
	\int_{B_r(x)} |\grad (u - a)|^2 + \tpar \frac{1}{2} |\grad (w - b)|^2 \leq C r^n.
	\]
	This is valid for any $B_{2r}(x)$, and implies that $u, \sqrt{\tpar} w$ are $C^{0, \a}(\Qb_R)$ (see \cite{Camp65}).\\
	
{\it Step 3: Conclusion.}	Finally, as $\W$ is open, $\W \in \compset$ and $\minE \leq \En(\W) \leq \int |\grad u|^2 + \int \frac{1}{2} |\grad w|^2 - w + \fv(|\W|) \leq \minE$. This implies that $w = w_\W$, $u$ is a first eigenfunction of $\W$, and 
	\[
		\lim_{k \rightarrow \infty} \int |\grad u_k^2|  = \int |\grad u|^2 \qquad \lim_{k \rightarrow \infty} \int |\grad w_k^2|  = \int |\grad w|^2 \qquad \lim_{k \rightarrow \infty} |\W_k| = |\W|.
	\]
	This gives the strong convergence of $u_k$ and $w_k$ in $H^1_0(\Qb_R)$. To see the convergence of $\W_k$, recall that $1_{\W} \leq \liminf_k 1_{\W_k}$; integrating over $\W$ and applying Fatou's lemma gives $|\W|\leq \liminf_k |\W \cap \W_k|$, or $|\W \sm \W_k|\rightarrow 0$. Together with $|\W_k|\rightarrow |\W|$ this guarantees $|\W_k \sm \W| \rightarrow 0$ as well.
	
	For (3), (1) and (2) immediately give that $\Min$ is nonempty by applying with $u_k = u_{\W_k}$, an eigenfunction of $\W_k$, and $w_k = w_{\W_k}$, where $\W_k$ is any sequence in $\compset$ with $\En(\W_k) \rightarrow \minE$. Choosing $\W_k = \W$ and $u_{\W_k} = u_{\W}$ any first eigenfunction and applying the continuity argument above gives that $u_\W$ and $\sqrt{\tpar} w_\W$ are continuous.
\end{proof}

Recall that the volume penalization term $f_{v,\eta}$ in the base energy does not immediately guarantee that a minimizer $\Omega \in \Min$  satisfies the desired volume constraint $|\Omega| \leq v$. However, the following lemma provides an initial upper bound on the volume of a minimizer, provided that the parameter $\vpar$ in the volume penalization term is taken to be sufficiently small. We will eventually show in Proposition~\ref{l:volumeisright} that and $\W \in \Min$ actually satisfies the desired volume constrain $|\W| = v$.

\begin{lemma}\label{l:volumenottoobig} Fix $R>0$ and $v, \vmax$ with $0 < v < \vmax$, and $0<\tpar \leq 1$.  There exists an $\vpar_0 = \vpar_0(v, \vmax, R) > 0$ so that for $\vpar < \vpar_0$, if $\W \in \Min$ then
		\[
			|\W| < \frac{\vmax + v}{2}.
		\]
\end{lemma}

\begin{proof}
	We have already verified in the proof of Lemma \ref{l:emin} that
	$
	\En(\W) \geq - C + \fv(|\W|)
	$
	for any $\W \in \compset$; this estimate is uniform for any $\tpar\leq 1$. By choosing some smooth $\W$ with precisely $|\W| = v$, we see that $\En(\W) \leq C$ with $C$ independent of $\eta$ (and $\tpar$). We therefore have, for $\W \in \Min$,
	\[
	\frac{|\W| - v}{\vpar}	\leq \fv(|\W|) \leq C.
	\]
	Choosing $\vpar \leq \frac{\vmax - v}{2C}$ gives the conclusion.
\end{proof}
 
\begin{remark}\label{rmk: eta parameter}
	{\rm 
	In the remainder of the paper, we will {\it always} assume, without further comment, that the parameter in the base energy $\En$ and the main energy $\Ep$ is taken such that    $\eta \leq \eta_0$, where $\eta_0 = \eta_0(v,\vmax, R)$ is the constant obtained in Lemma~\ref{l:volumenottoobig}. 
	}
\end{remark} 

In the next lemma, we show that the coefficient $\tpar$ in front of the torsional rigidity in the base energy can be taken to be  sufficiently small to guarantee that a minimizer $\Omega \in \Min$ is connected and its first eigenvalue is simple. Although the simplicity of the first eigenvalue follows from the connectedness, we prove these two properties separately so that the proof can be immediately generalized to prove Lemma~\ref{l:simpleeval} below.

\begin{lemma}\label{l:simpleevalaux}  Fix $R>0$, $v < \vmax $, and $\vpar>0$. There exists a $\tpar_0(v, \vmax, \vpar, R) > 0$ such that if $\W \in \Min$ with $\tpar < \tpar_0$, then $\W$ is connected, $u_\W > 0$ on $\W$ (up to changing sign), and
	\[
		\l_2(\W) > \ei(\W).
	\]
\end{lemma}

\begin{proof}
We first prove that these three properties holds when $\tpar =0$.

{\it Spectral gap when $\tpar =0$:} First, suppose by way of contradiction that $\l_2(\W) = \ei(\W)$.  Then there are two functions $u_1, u_2$ with $\int |\grad u_i|^2 = \ei(\W)$, $\int u_i^2 = 1$, and $\int u_1 u_2 = 0$.  By Lemma \ref{l:emin},  they are both continuous, and we may assume without loss of generality that $u_i \geq 0$ by replacing $u_i$ with $|u_i|$ in the Rayleigh quotient characterization of eigenfunctions. 
	 So, the orthogonality shows that $\{u_1 >0\}$ and $\{ u_2>0\}$ are disjoint open sets.	Now let $\W' = \{u_1 > 0\}$, noting that $\ei(\W') \leq \int |\grad u_1|^2 = \ei(\W)$. On the other hand, $|\W'| < |\W|$, so $\fv(|\W'|) < \fv(|\W|)$ and $\En(\W') < \En(\W) = \minE$. This is a contradiction, and shows that $\l_2(\W) > \l_1(\W)$.

{\it Connectedness and positivity of the eigenfunction when $\tpar =0$:} Next, suppose by way of contradiction that $\W$ is not connected. Then, there is some connected component  $\Om'$ on which the first  eigenfunction $u_\W$ of $\W$  (which is unique by the previous step) is positive. Then $u_{\Om}|_{\Omega'}$ is a first eigenfunction of $\W'$, so $\lambda_1(\W) = \lambda_1(\W')$, and $|\W'| < |\W|$.
So, just as in the previous step, we take $\W'$ as a competitor and find that $\En(\W')< \En(\W)$. This is a contradiction, showing that $\W$ is connected. Then, the strong maximum principle implies that $u_\W> 0$ in $\W.$
\\

{\it Setup for $\tpar >0$:}
Now we prove these three properties when $\tpar > 0$ is sufficiently small. We prove the spectral gap and the connectedness again arguing by contradiction. As in the previous part of the proof, the positivity of the first eigenfunction follows immediately from connectedness and the maximum principle. Suppose by way of contradiction that there is a sequence of numbers $\tpar_k \rightarrow 0$ and a corresponding sequence of sets $\W_k \in \Min_{\tpar_k}$ such that either $\l_2(\W_k) = \l_1(\W_k)$ or $\W_k$ is disconnected. 

Let us begin by making some general convergence and compactness observations about this sequence. By minimality, for any fixed competitor $\W \in \compset$, we have 
	\begin{align*}
	\ei(\W_k) + \tpar_k \tor(\W_k) + \fv(\W_k) & \leq \ei(\W) + \tpar_k \tor(\W) + \fv(\W)\\
	& \rightarrow \ei(\W) + \fv(\W)
	\end{align*}
	 From Corollary \ref{c:tornotminusinf} we know that  $\tor(\W) \geq - C$, and so in  each of $\ei(\W_k), \tor(\W_k),$ and $ \fv(\W_k)$ are bounded uniformly in $k$, so the middle term $\tpar_k \tor(\W_k) \rightarrow 0$. This means
	\begin{align*}
		\minE(\tpar_k)& = \ei(\W_k) + \fv(\W_k) + o_k(1)\\ & \leq \inf\{\ei(\W) + \fv(\W) : \W \in \compset \} + o_k(1) = \minE(0) + o_k(1).
	\end{align*}
	In particular, $\minE(\tpar_k) \rightarrow \minE(0)$ and $\lim_k \ei(\W_k) + \fv(\W_k) = \minE(0)$. For each $k$, fix  a first eigenfunction $u_{\W_k}$ for $\W_k$; as above we may take them nonnegative.  Applying the compactness claim of Lemma \ref{l:emin}, there is an $\W \in \Min_0$ with $\|u_{\W_k} - u_{\W}\|_{H^1_0(\Qb_R)} \rightarrow 0$ and $|\W_k \triangle \W| \rightarrow 0$. Since $\Omega \in \Min_0$, i.e. it is a minimizer for $\tpar =0$,  it satisfies the three properties of the lemma shown above in the case  $\tpar =0$.
	
	{\it Spectral gap for $\tpar >0$ sufficiently small:} 	Returning to our main contradiction argument, first suppose that $\l_2(\W_k) = \ei(\W_k)$ along some subsequence (that we do not relabel). So, for each $k$ we have a first eigenfunction $v_k\in H^1_0(\W_k)$ with $v_k \geq 0$ and  $\int v_k^2 = 1$ that is orthogonal to the previously selected first eigenfunction: $\int v_k u_{\W_k} = 0$, and $\int |\grad v_k|^2 = \ei(\W_k) \rightarrow \ei(\W)$. We may apply Lemma \ref{l:emin} to $v_k$ instead of $u_{\W_k}$ to obtain another set $\W' \in \Min_0$ and a continuous first eigenfunction $u_{\W'}$ of $\W'$ such that $\|v_k - u_{\W'}\|_{H^1_0(\Qb_R)} \rightarrow 0$ and $|\W_k \triangle \W'| \rightarrow 0$.  In particular,  this tells us that $|\W \triangle \W'|= 0$.  So the set $\W'':=\W\cup \W'$ is in the competitor class $\compset$, and satisfies $\ei(\W'') \leq \ei(\W)$, and $|\W''| = |\W|$. In particular, $\W'' \in \Min_0$ and $\ei(\W'') = \ei(\W)$.  In turn, this means both $u_{\W'}$ and $u_{\W}$ are first eigenfunctions of $\W''$. Moreover, they are orthogonal because 
	\[
		\int u_\W u_{\W'} = \lim_{k\rightarrow \infty} \int u_{\W_k} v_k = 0.
	\]
	Thus $\l_2(\W'') = \ei(\W'')$, which contradicts the previously established spectral gap for the case when $\tpar =0$.

{\it Connectedness for $\tpar >0$ sufficiently small:}	Next, let us assume that along some (unrelabeled) subsequence,  the sets $\W_k$ are disconnected. For each $k$, denote by $E^i_k$ the connected components of $\W_k$. 

As a first step, we notice that $|E^i_k|$ converges either to $|\W|$ or to $0$. Indeed, let $u_k \geq 0$ be the first eigenfunction of $\W_k$, and $E^1_k$ be the (unique) connected component of $\W_k$ on which it is nonzero. We know that $u_k \rightarrow u_{\W}$ a.e. Since $\W \in \Min_0$, i.e. $\W$ is a minimizer for $\tpar =0$, we know from the first part of the proof that $u_\W$ is strictly positive on $\W$. This gives that $|E_k^1| \rightarrow |\W|$ while $|E_k^i| \rightarrow 0$ for any other $i$.

	Next, we claim that for $k$ sufficiently large, i.e. $\tpar_k$ sufficiently small, we have $E_k^1 = \W_k$ , and thus $\W_k$ is connected.
	Set $|\W_k| = |E_k^1| + t_k$; we have shown that  $t_k \rightarrow 0$. Assume for contradiction that $t_k > 0$ for all $k$ along some subsequence. The basic idea  of the arguement is that the torsional rigidity is additive on connected components, and a component of small nonzero volume contributes too much torsional rigidity for a set to be minimizing.

More specifically,  take $E_k^1$ as a competitor for $\W_k$ in the minimization of $\En_{\tpar_k}$.
	Since $\ei(\W_k) = \ei(E_k^1)$,  we find
	\[
		\tpar_k \, \tor(\W_k) + \fv(|\W_k|) \leq \tpar_k\,  \tor(E_k^1) + \fv(|E_k^1|).
	\]
	Moreover, $\tor(A \cup B) = \tor(A) + \tor(B)$ for any disjoint sets $A, B$. Applying this fact to $A= E_k^1$ and $B = \W_k \setminus E_k^1$, we find that 
	\begin{equation}
		\label{rev 1}
		\tpar_k\,  \tor (\W_k \sm E_k^1)  +  \fv(|\W_k|) \leq \fv(|E_k^1|).
		\end{equation}

	Consequently, if we define the quantity
	\[
		a(t) := \inf \{\tor(E) : E \in \compset, |E|\leq t\},
	\]
	then rearranging  \eqref{rev 1} gives us an upper bound for $a(t_k)$:
	\begin{equation}
		\label{rev 2}
		\tpar_k\,  a(t_k) \leq \fv(|E_k^1|) - \fv(|\W_k|) \leq - \vpar t_k\,.
		\end{equation}
	
	On the other hand,  we establish a lower bound for $a(t)$ that will ultimately give us a contradiction.

Indeed, from Lemma~\ref{l:basic}(3), we have $\int |\grad w_E|^2 = \int w_E =-2\tor(E)$ (using $w_E$ as a test function for itself). Therefore,  we have 
	\[
		- \tor(E) = \frac{1}{2}\int w_E \leq |E|^{1/2} \|w_E\|_{L^{2}} \leq C |E|^{1/2} \|\grad w_E\|_{L^2} \leq C|E|^{1/2} \sqrt{-\tor(E)},
	\]
	where the constant $C=C(R,\vmax)$ introduced in the second-to-last inequality comes from the Poincar\'e inequality of Lemma \ref{l:poincare}. Dividing and squaring,
	$
		\tor(E) \geq - C |E|,
	$
	so
	\begin{equation}\label{e:simpleevalaux1}
		a(t) \geq - C t.
	\end{equation}
	Together \eqref{rev 2} and  \eqref{e:simpleevalaux1} tell us that $
	\eta  t_k\leq  - \tpar_k \, a(t_k) \leq C\, \tpar_k\, t_k$, and thus 
	$0< \eta \leq  C \tpar_k.$
	For large $k$, this is a contradiction. We conclude that $\W_k$ is connected.
\end{proof}

The same compactness argument can be applied to ``approximate'' minimizers instead, except for the full connectedness conclusion. We omit the proof.

\begin{lemma}\label{l:simpleeval} Fix $R>0,$ $0<v < \vmax$,  $\vpar >0$ and $\e > 0$. There exist  $\tpar_0 = \tpar_0(v, \vmax, \vpar, R) > 0$ and  $\d = \d(v, \vmax, R, \vpar, \e) > 0$ such that if $\tpar < \tpar_0$, $\W \in \compset$, and
	\[
	\En(\W) \leq \minE + \d,
	\]
	then:
	\begin{enumerate}
		\item There is an $\W' \in \Min$ with $|\W' \triangle \W| < \e$, $\|u_{\W'} - u_\W\|_{H^1} \leq \e$, and $\sqrt{\tpar} \|w_{\W'} - w_\W\|_{H^1} \leq \e$.
		\item $\l_2(\W) > \ei(\W) + c(v, \vmax, R)$.
		\item One connected component $E$ of $\W$ has $|E| \geq |\W| - \e$, while every other (if any) has $|E|\leq \e$.
	\end{enumerate}
\end{lemma}

Note that under the assumptions of this lemma, it is simply false that $\W$ must be connected: one may always add a ball of small measure to $\W$, which perturbs $\En$ continuously.

\subsection{Global minimizers of the base energy}\label{ss:globalbase}
In general, the constants in the previous section necessarily depend on $R$, and there is no reason to expect, for instance, existence of global minimizers on noncompact Riemannian manifolds with arbitrary (bounded geometry) behavior at infinity. In this section, we turn our attention to Riemannian manifolds $(M,g)$ where $M/\Isom$ is compact. In this case, most of the above estimates remain valid with constants independent of $R$, and minimizers of the base energy $\En$ among all open bounded sets exist and have bounded diameter.  For the present section, we focus only on estimates which will be relevant in later sections; in particular, the existence of minimizers is not treated here (but does follow from Theorem \ref{t:globalexist} later). 


Of course, in the case when $(M,g)$ is compact, then for $M \subset Q_R$ for sufficiently large $R$ , and thus the constants in the previous section may be taken independent of $R$ trivially. So, to simplify statements, in this section we will assume that $(M,g)$ is noncompact but $(M/G,g )$ is compact. Note in this case that $M$ has infinite volume.

For our first estimate, we show some initial bounds on the infimum of the energy and on the torsion function of any open bounded set in terms of its volume. Observe that $\minE(R)$ is a nonincreasing function of $R$, as the class $\compset$ is increasing in $R$. Set $\minE(\infty) = \lim_{R \rightarrow \infty} \minE(R)$. 

\begin{lemma} \label{l:globalElb} Assume $M /\Isom$ is compact and fix $v < \vmax < \infty$. Then there exists $C= C(\vmax)<\infty$ such that for any bounded open set $\W$ with $|\W|\leq \vmax$, we have $\minE(\infty) \geq - C$ and $w_\W \leq C$.
\end{lemma}

\begin{proof}
	First, we claim there exists a large $S \gg 0$ such that
	\[
		\bigcup_{e \in \Isom} e(\Qb_S) = M.
	\]
	If not, for every $k$ there is an $x_k \in M$ such that $e(x_k) \notin \Qb_k$ for any $e \in \Isom$. There exists, however, a subsequence $x_{k_j}$ and isometries $e_j$ with $e_j(x_{k_j}) \rightarrow x \in M$. If $x \in \Qb_J$ for some $J$, then as $\Qb_J$ is open $e_j(x_{k_j}) \in \Qb_J$ for $j$ large enough; this is impossible if $k_j \geq J$, so $x \notin \Qb_J$. This contradicts $M  = \cup_R \Qb_{R}$. Using the Vitali covering lemma and possibly taking $S$ larger, there exists a countable collection $L \ss \Isom$ such that $\cup_{e \in L} e(\Qb_S) = M$ and $\{e(\Qb_S)\}_{e \in L}$ have finite overlap. Up to possibly increasing $S$ depending on $\vmax,$ we also assume that $|\Qb_S| > \vmax$.

	Now, fix any $\W$ open, bounded, with $|\W|\leq \vmax$. For any $e \in L$, we apply the Poincar\'{e} inequality of Lemma \ref{l:poincare} to $e(\Qb_S)$ to obtain that
	\begin{equation}\label{e:globalElb}
		\|w_\W \|_{L^2(\W \cap e(\Qb_S))} \leq C \|\grad w_\W\|_{L^2(\W \cap e(\Qb_S))}
	\end{equation}
	for a constant $C= C(S, \vmax) = C(\vmax)$, using that $|\W \cap e(\Qb_S)|\leq \vmax < |\Qb_S|$. Summing over all $e \in L$ and using the finite overlapping property gives
	$
		\int_\W w_\W^2 \leq C \int_\W |\grad w_\W|^2.
	$
	Then
	\[
		\int_\W w_\W \leq  C \, |\W|^{1/2}  \|\grad w_\W\|_{L^2(\W)} \leq C \, \|\grad w_\W\|_{L^2(\W)},
	\]
	so
	$
		\tor(\W) = \int \frac{1}{2}|\grad w_\W|^2 - w_\W \geq - C.
	$
	On the other hand $\ei(\W) > 0$, while $\fv(\W) \geq - \vmax$, so $\En(\W) \geq - C(\vmax)$. Taking the infimum over all $\W$ gives the first conclusion.
	
	To see that $w_\W \leq C$, we now have from \eqref{e:globalElb}, elliptic estimates \cite[Theorem 8.17]{GT}, and a basic covering argument that
	\[
		\|w_\W\|_{L^\infty(\W \cap e(\Qb_S))} \leq C[\|w_\W\|_{L^2(\Qb_{2S})} + 1] \leq C.
	\]
	Applying to every $e$ gives the estimate.
\end{proof}

We now establish an analogue of Lemma~\ref{l:simpleeval} that is independent of $R$:  that any low energy set can be well approximated by a minimizer (on some $\Qb_S$ for $S$ uniform) and has one ``large'' connected component. We use a concentration compactness argument to handle the loss of compactness coming from the isometries of the space. 

\begin{lemma} \label{l:globalsimpleeval} Assume $M /\Isom$ is compact. Fix  $v < \vmax < \infty$, $\e > 0$, and $\vpar >0$. There exist  $\tpar_0 = \tpar_0(v, \vmax, \vpar) > 0$, and a $\d = \d(v, \vmax, \vpar, \e) > 0$ such that if $\tpar < \tpar_0$, then the following holds. For any open bounded set  $\W \subset M$ with $|\W |\leq \vmax$  and $\En(\W) \leq \minE(\infty) + \d$:
	\begin{enumerate}
		\item There is an $S = S(v, \vmax, \vpar, \e)$ and a $U \in \Min(S)$ (i.e. a minimizer of $\En$ over open $\W \ss \Qb_S$ with $|\W|\leq \vmax$) such that $|U \triangle \W| < \e$, $\|u_{U} - u_\W\|_{H^1} \leq \e$, and $\sqrt{\tpar }\|w_{U} - w_\W\|_{H^1} \leq \e$.
		\item $\l_2(\W) > \ei(\W) + c(v, \vmax)$
		\item One connected component $E$ of $\W$ has $|E| \geq |\W| - \e$, while every other (if any) has $|E|\leq \e$.
	\end{enumerate}
\end{lemma}

\begin{proof}
	Let $L \ss \Isom$ and $S$ be as in the proof of Lemma \ref{l:globalElb}. We use the shorthand $\bar{Q} =\Qb_{S}$. Let $\{\psi_e\}_{e \in L}$ a partition of unity subordinate to the cover of $M$ by $\{e(\bar\Qb) \}_{e \in L}$ with $\psi_e = \psi \circ e^{-1} \in C^\infty_c(\bar\Qb)$ translates of each other.	
	Throughout this proof, $C$ and $c$ denote  constants depending only on $\vmax$ (and, as usual, on $g$) whose values change from line to line. Let $\delta \leq 1$ be a fixed number to be specified in the proof and let $\W$ be as above. Applying Lemma~\ref{l:globalElb}, we know that 
	\begin{equation}\label{e: global bound}
	\int_\W |\grad u_\W|^2 + u_\W^2 + |\grad w_\W|^2 + w_\W^2 \leq C.
\end{equation}
	The first step of the proof is a concentration compactness argument that allows us to replace $\Omega$ by a uniformly bounded set, losing only a small amount of mass and increasing the base energy only a small amount. \\

{\it Step 1:} There exist $\tpar_0$ and $\bar J = \bar J(\vmax, \delta) \in \mathbb{N}$ such that we may choose $x_0 \in M$ and  $j \leq \bar J$ depending on $\W$  such that
	\begin{equation}\label{e:globalsimpleeval1}
		\begin{cases}
			\En(\W \cap B_{8 S j}(x_0)) \leq \minE +  2 \d \\
			|\W \sm B_{8 S j}(x_0)| \leq 4 \vpar \d.
		\end{cases}
	\end{equation} 
	Because this step is rather involved, we divide it into several substeps.
\\	
	
{\it Step 1a: Selection of $x_0$.} We show that $\Omega$ has a (uniformly) nontrivial amount of mass in $e_0 (\bar \Qb)$ for some $e_0 \in L$, and choose $x_0$ in this set. More specifically, for any $e \in L$, we apply Lemma \ref{l:poincare} to $\W \cap e(\bar\Qb)$ to find
	\[
		\int_{e(\bar\Qb)} u_\W^2 \leq \left|\W \cap e(\bar\Qb)\right|^{\frac{2}{n}} \Big(\int_{e(\bar\Qb)} u_\W^\frac{2n}{n-2}\Big)^{\frac{n-2}{n}} \leq C \left|\W \cap e(\bar\Qb)\right|^{\frac{2}{n}} \int_{e(\bar\Qb)} |\grad u_\W|^2.
	\]
	Summing over $e \in L$, recalling the finite overlap property of $L$ and \eqref{e: global bound}, we have
	\[
		1 = \int_{\W} u_\W^2 \leq C\, \Big(\sup_{e \in L}|\W \cap e(\bar\Qb)|^{\frac{2}{n}}\Big) \int_\W |\grad u_\W|^2 \leq C \, \Big(\sup_{e \in L}|\W \cap e(\bar\Qb)|^{\frac{2}{n}}\Big).
	\]
 Consequently, $|\W \cap e_0(\bar \Qb)| \geq c$ for some $e_0$. We let $x_0$ be any fixed point in $ e_0(\bar\Qb)$.
\\

{\it Step 1b: Selecting $j$ for each $J \in \mathbb{N}$.} Next, we show how, for each $J\in \mathbb{N}$, to suitably choose $j \leq J$ that will ultimately lead to estimates of the form \eqref{e:globalsimpleeval1} but with a $C/J$ error; we will later choose $J$ depending on $\delta $ to absorb this term.
For any $k \in \mathbb{N}$, we denote by $L_k$  the finite  collection  of isometries such that $e(\bar \Qb)$ lies in the annulus $B_{16 k S}(x_0) \setminus B_{8 k S}(x_0) $. That is,
	\[
	L_k = \left\{e \in L: e(\bar\Qb) \cap \left[B_{8 k S}(x_0) \cup (M \sm B_{16 k S}(x_0))\right] = \emptyset \right\}.
	\]
Since $\diam \bar\Qb \leq 2S$ by assumption, the collections $L_k$ are  pairwise disjoint, and so from \eqref{e: global bound} we find that 
	\[
		\sum_{k = 1}^\infty  \ \sum_{e\in L_k} \int_{e(\bar \Qb)\cap \W} |\grad u_\W|^2 + u_\W^2 + |\grad w_\W|^2 + w_\W^2  \leq C.
	\]
	Since this sum is finite, it follows that for every $J > 0$, we may find some $j \leq J$ such that
	\[
		\sum_{e\in L_j} \int_{e(\bar\Qb)\cap \W} |\grad u_\W|^2 + u_\W^2 + |\grad w_\W|^2 + w_\W^2 \leq \frac{C}{J}.
	\]
	In other words, the functions $u_\Omega$ and $w_\Omega$ have most of their mass and energy outside of the annulus corresponding to $L_j$. Let this $ j=j(J)$ be fixed.\\

	{\it Step 1c: Truncation and intermediate bounds.} In this step, we consider $J \in \mathbb{N}$ fixed and $j=j(J)$ as determined in the previous step. Let us now define two smooth truncations of the eigenfunction $u_\W$, the ``inner part'' $u_I$ in the ball $B_{8jS}$ and the ``outer part'' $u_O$ in the complement of $B_{16Sj}$, by setting 
	\begin{align*}
	u_I = \sum_{e \in L_I} \psi_e u_\W & \qquad \qquad	L_I =\left\{e \in L \  :\  e(\bar\Qb) \cap B_{8j S}(x_0) \neq \emptyset\right\},\\
	u_O = \sum_{e \in L_O} \psi_e u_\W & \qquad \qquad	L_O = \left\{e \in L\  : \ e(\bar\Qb) \cap (M \sm B_{16j S}(x_0)) \neq \emptyset\right\}\,.
	\end{align*}
	Notice that the functions $u_I,$ $u_O$ have disjoint support and the remaining ``annular part'' of $u_\W$ is given by $ u_\W -u_I + u_O = \sum_{e \in L_k} \psi_e u_\W$. The truncations $u_I$ and $u_O$ capture most of the energy of $u_\W$ in the following sense: 
	\begin{align*}
		\left|\int |\grad u_\W|^2 - |\grad u_O|^2 - |\grad u_I|^2\right| = \Big|\int \langle \grad  [u_\W - u_O & - u_I], \grad [u_\W + u_O + u_I] \rangle\Big| \\
		& \leq \sum_{e \in L_j}\int_{e(\bar \Qb)} \1|\grad \psi_e| |u_\W| + \psi_e |\grad u_\W|\2 \cdot C |\grad u_\W|\\
		& \leq C \sum_{e \in L_j} \int_{e(\bar \Qb)} |u_\W|^2 + |\grad u_\W|^2 \leq \frac{C}{J}.
	\end{align*} 
	In particular, this together with the Poincar\'{e} inequality applied to $u_O$ provide the following upper bound for the energy of $u_I$:
	\begin{equation}\label{e: uI energy bound}
	\int| \grad u_I|^2 \leq \lambda_1(\W) - \int |\grad u_O|^2+ \frac{C}{j} \leq \lambda_1(\W)\Big(1 - \int u_O^2\Big) + \frac{C}{J}.
	\end{equation}
	In a similar fashion, we find that 
	\begin{equation}\label{e: mass}
		\int u_\W^2 - u_O^2 - u_I^2 \leq \frac{C}{J}. 
	\end{equation}	
The first main estimate of this sub-step is the following: there exists $C$ such that		 
\begin{equation}\label{e: eval truncation a}\begin{split}
\text{ if }\quad  \int u_I^2\geq \, \frac{1}{4} \qquad \text{ then }	\quad \ei\left(\W \cap B_{8 S j}(x_0)\right) &  \leq \ei(\W) + \frac{C}{J}\,,\\
\text{ if }\quad  \int u_O^2\geq\,  \frac{1}{4} \qquad \text{ then }	\quad \ei\left(\W \setminus B_{8 S j}(x_0)\right) &  \leq \ei(\W) + \frac{C}{J}\,.
	\end{split}\end{equation}
Note that \eqref{e: mass} guarantees that one of the two alternatives in \eqref{e: eval truncation a} must hold provided $J$ is sufficiently large. 
	If  $\int u_I^2\geq 1/4$ then \eqref{e: mass} implies  that $(1- \int u_O^2 )/\int u_I^2 \leq 1+ C/J$. So, if we multiply \eqref{e: uI energy bound} by $(1- \int u_O^2 )/\int u_I^2$ and divide both sides by $1- \int u_0^2$, we arrive at
	\[
		\frac{\int |\grad u_I|^2}{\int u_I^2} \leq \ei(\W) + \frac{C}{J}.
	\]
	Here we have also used the fact that $\ei(\W)$ is uniformly bounded; recall \eqref{e: global bound}.
So,  the first estimate of \eqref{e: eval truncation a} follows by taking $u_I$ as a test function for $\ei(\W \cap B_{8 S j}(x_0))$. The proof of the second estimate of \eqref{e: eval truncation a} is the same, using the analogous estimate to \eqref{e: uI energy bound} for $u_O$.

The second main estimate of this substep is the following:
\begin{equation}	
\label{e: torsion truncation}\begin{split}
\tor(\W \cap B_{8 S j}(x_0)) &\leq\tor(\W) + \frac{C}{J} + C |\W \sm B_{8 S j}(x_0)|,\\
\tor(\W \setminus B_{8 S j}(x_0)) &\leq\tor(\W) + \frac{C}{J} + C |\W \cap B_{8 S j}(x_0)|
\end{split}\end{equation}
The estimate \eqref{e: torsion truncation} follows similarly to \eqref{e: eval truncation a}. Splitting the torsion function $w_\W$ into an ``inner part'' $w_I,$, and ``outer part'' $w_O$, and an ``annular part'' $w_\Omega - w_I - w_O$ exactly as above leads to
	\[
		- \frac{1}{2}\int w_I + w_O \leq \tor(\W) +  \frac{C}{J};
	\]
	we omit the details. 
	 Recalling that $w_\W \leq C$ from Lemma \ref{l:globalElb}, we find that
	\begin{align*}
		\tor(\W \cap B_{8 S j}(x_0)) \leq -\frac{1}{2} \int w_I 
		& \leq \tor(\W) + \frac{C}{J} + \frac{1}{2} \int w_O\\
		& \leq \tor(\W) + \frac{C}{J} + C |\W \sm B_{8 S j}(x_0)|,
	\end{align*}
thus proving the first estimate in \eqref{e: torsion truncation}. The second estimate in \eqref{e: torsion truncation} is proven in the same way. \\

{\it Step 1d: Energy comparison and selection of $\bar J$ and $\tpar_0$.} Fix any $J \in \mathbb{N}$ and let $j=j(J)$ be as determined in Step 1b. Let us first consider the case that $
\int u_I^2 \geq \frac{1}{4}.$
In this case, we may use the first  estimates in \eqref{e: eval truncation a} and \eqref{e: torsion truncation} to compare $\En(\W)$ to $\En(\W \cap B_{8 S j}(x_0))$, which gives
	\begin{align*}
		\En(\W) \leq \minE + \d
		&\leq  \En(\W \cap B_{8 S j}(x_0)) + \d \\
		& \leq \ei(\W) + \tpar \tor(\W) + \frac{C}{J} + \fv (|\W \cap B_{8 S j}(x_0)|) + \d + C \tpar|\W \sm B_{8 S j}(x_0)| \\
		& \leq \En(\W) + \frac{C}{J} + \fv(|\W \cap B_{8 S j}(x_0)|) - \fv(|\W|) + \d + C \tpar |\W \sm B_{8 S j}(x_0)|\\
		& \leq \En(\W) + \frac{C}{J} + \d - [\vpar - C \tpar]|\W \sm B_{8 S j}(x_0)|.
	\end{align*}
At this point, we can fix our remaining parameters. Choosing $\tpar_0 < \frac{\vpar}{2C}$, the last term is negative and this leads to
	\begin{equation}\label{e:globalsimpleeval1 b}
		\begin{cases}
			\En(\W \cap B_{8 S j}(x_0)) \leq \minE + \frac{C}{J} + \d \leq 2 \d \\
			|\W \sm B_{8 S j}(x_0)| \leq 4 \vpar \d.
		\end{cases}
	\end{equation}
Finally, we choose $J = \bar J$ large enough in terms of $\d$ so that $\frac{C}{J} \leq \d$, and \eqref{e:globalsimpleeval1} follows. This concludes the proof of Step 1 in the case when $
\int u_I^2 \geq \frac{1}{4}.$ holds. Now, assume instead that $
\int u_O^2 \geq \frac{1}{4}.$ We may argue in the analogous fashion to find that 
	\begin{equation}\label{e:globalsimpleeval1 c}
		\begin{cases}
			\En(\W \sm B_{8 S j}(x_0)) \leq \minE + \frac{C}{J} + \d \leq 2 \d \\
			|\W \cap B_{8 S j}(x_0)| \leq 4 \vpar \d.
		\end{cases}
	\end{equation}
This contradicts Step 1a for $\delta$ chosen sufficiently small. We conclude that $
\int u_I^2 \geq \frac{1}{4}$ holds and thus have completed the proof of Step 1. \\

{\it Step 2: Conclusion.}	We are now in a position to prove the first conclusion of the lemma. Let $\W' = \W \cap B_{8 S j}(x_0)$ be the truncation of $\W$ obtained in Step 1. Let $S'$ be chosen large enough that $e_0^{-1}( B_{8 S J}(x_0)) \ss \Qb_{S'}$; note that this depends only on $J$ and $S$, so only on $\vmax$ and $\d$. Applying Lemma \ref{l:emin} to $e_0^{-1}(\W')$ on  $\Qb_{S'}$, there is a $U\in \Min(S')$ with $\|u_{U} - u_I\|_{H^1}, \sqrt{\tpar} \|w_U - w_I\|_{H^1} \leq \e/2$. On the other hand, from the Caccioppoli inequality and letting $L'=\{e \in L : e(\Qb_S) \cap B_{8j S}(x_0) = \emptyset\}$,
	\[
		\int |\grad (u_\W - u_I)|^2 \leq C \sum_{e \in L'} \int |\grad \psi_e|^2 |u_\W|^2 + \psi_e^2 |\grad u_\W|^2 \leq C \int_{\W \sm B_{8 S j}(x_0)} |u_\W|^2  \leq C \d,
	\]
	using that $|u_\W|\leq C$ and the volume estimate in \eqref{e:globalsimpleeval1} at the end. Similarly, $\int |\grad (w_\W - w_I)|^2 \leq C \d$, and we obtain conclusion (1) from the triangle inequality.
	
	To prove  (2), note that applying Lemma \ref{l:simpleeval} gives that $\l_2(\W') \geq \ei(\W) + c_0$ for $c_0 = c_0(S', v, \vmax)$, and so it remains to show that $\l_2(\W') \leq \l_2(\W) + c_0/2$. To this end, let $u_2$ be a second eigenfunction for $\W$.  from Lemma \ref{l:efbdd} we have that $|u_2|\leq C(\ei(\W)) \leq C$ unless $\l_2(\W) > 1 + \ei(\W)$ (in which case we would be done). Then set
	$
	u_{2,I} = \sum_{e \in L_I} \psi_e u_2
	$
	to be the ``inner part,'' i.e. the analogue of $u_I$ for $u_2$. As $ \Lap |u_2| \geq - \l_2(\W) |u_2|$ in the sense of distributions, the same Caccioppoli inequality argument gives that
	\[
		\int |\grad (u_2 - u_{2,I})|^2 \leq C \int_{\W \sm B_{8 S j}(x_0)} |u_2|^2  \leq C \d.
	\]
	Hence
	\[
		\frac{\int |\grad u_{2,I}|^2}{\int u_{2,I}^2} \leq \l_2(\W) + C \d \qquad \text{and} \qquad \frac{|\int u_{2,I}\, u_I |}{\|u_{2,I}\|_{L^2}\|u_I\|_{L^2}}\leq C \d,
	\]
	implying that $\l_2(\W') \leq \l_2(\W) + C \d.$
	Taking $\d$ small enough completes the proof of (2).
	
	To check the final conclusion (3) of the lemma, note that from Lemma \ref{l:simpleeval}, we know that one connected component $A$ of $\W'$ has $|A| > |\W'| - \e/2$, while every other has $|A| \leq \e/2$ (choosing $\d$ small enough in terms of $\e$). As $|\W \sm B_{8 S j}(x_0)| \leq C \d$, we may conclude by choosing $C\d < \e/2$.
\end{proof}

\section{The lower bound} \label{s:lb}
In this section we consider inward minimizers of the main functional $\Ep$ defined in \eqref{eqn: main functional}. The main result of this section is Theorem~\ref{t:lb}, which provides a linear lower bound for the growth of the function $ u_\W + \sqrt{\tpar} w_\W$ away from the boundary of an inward minimizer $\W$. The theorem is phrased in terms of  a lower bound on the quantity $\DO(\W)$ defined in \eqref{def: DO}. One important consequence of this theorem is a lower volume density estimate for inward minimizers shown in Corollary~\ref{c:badldensitybd}. Throughout this section, we fix $R>0$, $0 < v< \vmax$, and $0<\eta\leq \eta_0$ (recall Remark~\ref{rmk: eta parameter}). 

\begin{theorem} \label{t:lb}
	 There are  constants  $\d_m, \tpar_m, c_m> 0$ depending only on $R, v, \vmax,$ and $ \vpar$ and a constant $ \err_m(R, v, \vmax,\vpar, \tpar)>0$   such that if we fix  $\tpar < \tpar_m$ and then $\err < \err_m$, the following holds. Let $\W$ be an inward minimizer of $\Ep$ on $\Qb_R$ satisfying  $\En(\W) \leq \minE + \d_m$. Then we have
	\[
		\DO(\W) \geq c_m
	\]
	where $\DO(\W)$ is defined in \eqref{def: DO}. If $M/\Isom$ is compact, all constants may be taken independent of $R$.
\end{theorem}
 It is worth noting that Theorem~\ref{t:lb} and Lemma~\ref{l:torsionbound} show that the torsion function $w_\W$ grows at least linearly away from the boundary, but we cannot immediately deduce an analogous lower bound for the growth of the eigenfunction $u_\W$. It will require some delicate Green's function estimates to eventually show in Section~\ref{s:measure} that, for minimizers, the first eigenfunction also satisfies this type of linear lower bound. 

The core arguments in the proof of Theorem~\ref{t:lb} below are from David and Toro \cite{DT15}. We present the details to show how to apply the key estimate of Proposition~\ref{lem:key} in order to handle the nonlinear term $\nl$, as well as to verify the dependence on all of the parameters.  The proof will be established by iteratively applying the following lemma.

\begin{lemma} \label{l:davidtoroiteration} There exist constants  $\d_m, \tpar_m , c_m> 0$ depending only on $R, v, \vmax,$ and $\vpar$ and a constant  $\err_m(R, v, \vmax, \vpar , \tpar) > 0$  such that if  we fix  $\tpar < \tpar_m$ and then $\err < \err_m$,   the following holds.
	 Let $\W$ be an inward minimizer on $\Qb_R$ satisfying  $\En(\W) \leq \minE + \d_m$. For any $x\in \Qb_R$, and $\sqrt{r} \leq c_m$, if
	\begin{equation}\label{eqn: dt hyp}
		\sup_{B_r(x)} u_\W + \sqrt{\tpar} w_\W \leq c_m r
	\end{equation}
	then
	\[
		\sup_{B_{r/2}(x)} u_\W + \sqrt{\tpar} w_\W \leq \frac{1}{4} c_m r.
	\]
	If $M/\Isom$ is compact, all constants may be taken independent of $R$.
\end{lemma}
\begin{proof} The basic idea of the proof is make use of the inward minimality property of $\W$ using an energy  competitor $\W'$ obtained by removing a small ball from $\W$.

{\it Step 1: The competitor $\Omega'$.}
	Let $c_m \leq \min\{ \text{inj}_M, 1\}$ be a fixed  number to be specified later in the proof. For $r \leq c_m^2$ and $x \in \Qb_R$, we consider the competitor 
	$
	\W' = \W \sm B_{3r/4}(x).
	$
 Choose $\d_m,$ and $\tpar_m$ according to Lemma \ref{l:simpleeval}, so that $\W$ has a spectral gap $\l_2(\W) > \ei(\W) + c(R, \vmax, v)$ of a definite size. This allows us to apply the key estimate of Proposition~\ref{lem:key} to $\W' \ss \W$, which yields
	\[
		\|u_{\W'} - u_{\W}\|_{L^1} \leq C\big[\tor(\W') - \tor(\W) + \ei(\W') - \ei(\W)\big],
	\]
	where the constant depends only on $R, \vmax, v$. (Note that the constant a priori also depends on an upper bound for  $\ei(\W')$; a basic cutoff function argument like the one in step 3 below and the assumption on $\En(\Omega)$ show that $\ei(\W')\leq C(R, \vmax , v).$)   If $M/\Isom$ is compact, we may take the constants here independent of $R$ by using Lemma \ref{l:globalsimpleeval} in place of Lemma~
	\ref{l:simpleeval}. This leads to
	\[
		|\nl(\W) - \nl(\W')| \leq |\W \triangle \W'| + \int |u_{\W'} - u_{\W}| \leq |\W \sm \W'| + C \left[\tor(\W') - \tor(\W) + \ei(\W') - \ei(\W)\right]
	\]
	using property \ref{a:nllip} of $\nl$. From the inward minimality of $\W$, we have $\Ep(\W) \leq \Ep(\W')$ and so
	\begin{align*}
		0 &\leq \Ep(\W') - \Ep(\W) \\
		& \leq  \ei(\W') - \ei(\W) + \tpar [\tor(\W') - \tor(\W)] + \fv(\W') - \fv(\W) \\
		&\qquad \qquad + \err \Big(\left|\W \sm \W'\right| + C \big[\tor(\W') - \tor(\W) + \ei(\W') - \ei(\W)\big]\Big) \\
		& \leq 2 \big[\tpar (\tor(\W') - \tor(\W)) + \ei(\W') - \ei(\W)\big] + (\err - \vpar)  |\W \sm \W'|,
	\end{align*}
	where the final inequality holds provided  $\err_m$ is chosen small enough in terms of $C$ and $\tpar$.  So long as $\err < \vpar/2$, we may rewrite this as
	\begin{equation} \label{e:davidtoroiteration}
		|\W \cap B_{3r/4}(x)| = |\W \sm \W'| \leq C \big[\tpar (\tor(\W') - \tor(\W)) + \ei(\W') - \ei(\W)\big].
	\end{equation}

{\it Step 2: Estimates for the eigenfunction and torsion function. }
	First, note that $\Lap u_\W, \Lap w_\W \geq - C$ on $M$ for a constant $C$ depending only on $R, v, \vmax, $ and $\eta$. Indeed,  the lower bound for $\Lap w_\W $ comes directly from Lemma \ref{l:basic}. For $\Lap u_\W$, the assumed upper bound on the energy of $\W$ paired with the torsional rigidity lower bound of Corollary~\ref{c:tornotminusinf} and the fact that $f_{v,\eta}(|\W|) \geq -v\,\eta$ implies that $\lambda_1(\W) \leq \lambda_0(R, v, \vmax , \eta)$. Then applying   Lemma \ref{l:basic} and the bounds from Lemmas \ref{l:efbdd}, we deduce the statement for $\Lap u_\W$. 

Fix $\phi$ a cutoff function which is $1$ on $B_{\frac{7 r}{8}}(x)$, compactly supported on $B_1$, and has $|\grad \phi|\leq C/r$. Applying the Caccioppoli inequality to $w_\W$ (i.e. use $\phi^2 w_\W$ as a test function for $w_\W$ and rearrange terms) and the assumption \eqref{eqn: dt hyp},
	\[
		\tpar \int \phi^2 |\grad w_\W|^2 \leq  \tpar C \int\left(  |\grad \phi|^2 w_\W^2 + \phi^2 w_\W \right) \leq C [r^{n - 2} c_m^2 r^2 + r^n \sqrt{\tpar} c_m r  ] \leq C r^n c_m^2.
	\]
	In the final inequality we have used that $r \leq \sqrt{r} \leq c_m$. We may estimate $u_\W$ in a similar manner:
	\begin{equation}
		\label{rev 3}
		\int \phi^2 |\grad u_\W|^2  \leq C\int \left( |\grad \phi|^2 u_\W^2 + \phi^2 u_\W \right) \leq C r^n c_m^2.
		\end{equation}
	All together, this implies
	\[
		\int_{B_{7r/8}(x)} |\grad u_\W|^2 + \tpar |\grad w_\W|^2 \leq C r^n c_m^2, \qquad \int_{B_{7r/8}(x)} u_\W^2 + \tpar w_\W^2 \leq C r^{n + 2} c_m^2.
	\]

{\it Step 3: Estimates for the eigenvalue and torsional rigidity of $\W'$.}
	Now take another cutoff function $\phi_1$ which is $1$ on $M \sm B_{7r/8}(x)$, vanishes on $B_{3r/4}(x)$, and has $|\grad \phi_1|\leq C/r$. We use $\phi_1 u_\W$ as a competitor for $\ei(\W')$ to find
	$
		\ei(\W') \leq {\int |\grad (\phi_1 u_\W)|^2}/{\int (\phi_1 u_\W)^2}.
	$
	Estimating the denominator,
	\[
		\int \phi_1^2 u_\W^2 \geq \int u_\W^2 - \int_{\W \sm B_{7 r/8}(x)} u_\W^2 \geq 1 - C r^n c_m^2.
	\]
	For the numerator, 
	we use the Caccioppoli inequality \eqref{rev 3} with $\phi_1$ in place of $\phi$ and the Cauchy Schwarz inequality to find
	\begin{align*}
			\int |\grad (\phi_1 u_\W)|^2  & =  \int |\grad u_\W|^2 \phi_1^2 + \int |\grad \phi_1|^2 u_\W^2  + 2 \int \langle \grad u_\W, \grad \phi_1 \rangle\,  u_\W \phi_1\\
			&\leq \ei(\W)  +2  \int |\grad u_\W|^2 \phi_1^2 + |\grad \phi_1|^2 u_\W^2 \leq \ei(\W) +  C r^n c_m^2,	
	\end{align*}
	so
	\[
		\ei(\W') \leq \ei(\W) + C r^n c_m^2.
	\]
	An identical estimate on $w_\W$ gives
	\[
		\tpar \tor(\W') \leq  \tpar \tor(\W) + C r^n c_m^2.
	\]

{\it Step 4: Conclusion. }	Substituting the estimates of Step 3 into \eqref{e:davidtoroiteration}, we arrive at
	\begin{equation} \label{e:davidtoroiteration2}
		|\W \cap B_{3r/4}(x)| \leq C r^n c_m^2.
	\end{equation}
	We now apply the local maximum principle for subsolutions, \cite[Theorem 8.17]{GT}, to $u_\W + \sqrt{\tpar}w_\W$:
	\begin{align*}
		\sup_{B_{r/2}(x)} u_\W + \sqrt{\tpar}w_\W &\leq C \left[r^{-n/2} \|u_\W + \sqrt{\tpar}w_\W\|_{L^2(B_{3r/4}(x))} + r^2 \right] \\
		&\leq C\left[ c_m r r^{-n/2} |\W \cap B_{3r/4}(x)|^{1/2} + r c_m^2  \right]\leq Cr c_m^2,
	\end{align*}
	where we used \eqref{e:davidtoroiteration2}, \eqref{eqn: dt hyp}, and that $r \leq c_m^2$ by assumption. Finally, we choose $c_m$ small enough so that $C c_m \leq \frac{1}{4}$ to conclude the proof.
\end{proof}
Theorem~\ref{t:lb} now follows by iteratively applying Lemma~\ref{l:davidtoroiteration}:
\begin{proof}[Proof of Theorem~\ref{t:lb}]
	Set all constants as in Lemma \ref{l:davidtoroiteration}. Take any $y \in \bar{\W}$ and any $r < r_0 := c_m^2/16$. We claim that
	\[
		\sup_{B_r(y)} u_\W + \sqrt{\tpar} w_\W \geq c_m r.
	\]
	Indeed, if this is not the case, we apply Lemma \ref{l:davidtoroiteration} to obtain
	\[
		\sup_{B_{r/2}(y)} u_\W + \sqrt{\tpar} w_\W \leq \frac{c_m}{2} \frac{r}{2}.
	\]
	Then take any $x \in B_{r/4}(y)$:
	\[
		\sup_{B_{r/4}(x)} u_\W + \sqrt{\tpar} w_\W \leq c_m \frac{r}{4},
	\]
	so repeatedly applying Lemma \ref{l:davidtoroiteration} gives
	\[
		\sup_{B_{r/4 \cdot 2^{-k}}(x)} u_\W + \sqrt{\tpar} w_\W \leq \frac{c_m}{2}  \frac{r}{4} 2^{-k}.
	\]
	In particular, $x$ is a Lebesgue point of $u_\W$ and of $w_\W$, $w_\W$ and $u_\W(x) = w_\W(x) = 0$. This is true of all points in $B_{r/4}(y)$, so $u_\W \equiv w_\W \equiv 0$ there. Using $\W \sm B_{r/4}(y)$ as a competitor for $\W$ implies that $|\W \cap B_{r/4}(y)| = 0$, contradicting that $y \in \bar \W$ and $\W$ open.
	
	For $r \in [r_0, 1]$, we simply use
	$
		\sup_{B_r(y)} u_\W + \sqrt{\tpar} w_\W \geq c_m r_0 \geq c_m r_0 r.
	$
	This gives $\DO(\W) \geq c_m r_0$.
\end{proof}

Note that while for an absolute minimizer (should one exist) $\En(\W) \leq \minE + \d$ follows from $\err$ being small (see Section \ref{s:exist}), there is no reason for that to be the case for an inward minimizer in general.

\begin{remark}{\rm
	The careful reader may observe that it is possible to modify the argument in this section to give
	$
	\sup_{B_r(y)} w_\W \geq c_m r
	$
	as long as $\err < \err_m(R, v, \vmax)$ independent of $\tpar$ (by using Lemma \ref{l:torsionbound} to estimate $u_\W \leq C w_\W$). Estimates in this spirit were used by Bucur \cite{Bucur12} for other problems. We do not pursue this point further here because the opposite estimate on $\UP(\W)$, discussed in the next section, cannot be altered in this way without a version of Proposition~\ref{lem:key} which omits the torsional rigidity terms on the right-hand side.
	}
\end{remark}

The following corollary will help us bound the diameter of minimizers when taking $R \rightarrow \infty$. We will discuss much sharper estimates in Section \ref{ss:density}, using the upper bound $\UP(\W)$ as well.

\begin{corollary}\label{c:badldensitybd}
	Let $\W$ be as in Theorem \ref{t:lb}. Then for any $x \in \W$ and $r < c_m^4$,
	\[
		|\W \cap B_r(x)| \geq c(r, v, \vmax, \vpar, R).
	\]
	If $M/\Isom$ is compact, $c$ may be taken independent of $R$.
\end{corollary}

\begin{proof}
	As in the proof of Lemma \ref{l:davidtoroiteration}, apply the local maximum principle on $B_r(x)$ to obtain
	\begin{align*}
		\sup_{B_{r/2}(x)} u_\W^2 + \tpar w_\W^2  \leq Cr^{-n}\int_{B_r(x) \cap \W} (u_\W^2 + \tpar w_\W^2) + Cr^4 
		& \leq C|\W \cap B_r(x)| r^{-n} \sup_{\W} (u_\W^2 + \tpar w_\W^2) + Cr^4\\
		& \leq C|\W \cap B_r(x)| r^{-n} + Cr^4,
	\end{align*}
	with the last step using Lemmas \ref{l:efbdd} and \ref{l:torsionfunctionbdd}. The quantities $|\tor(\W)|$ and $\ei(\W)$ are controlled by $\En(\W)$, which is bounded from the assumption $\minE \leq \En(\W) \leq \minE + \d$ and Lemma \ref{l:emin} or \ref{l:globalElb}. The constant $C$ can be taken to depend only on $M, v, \vmax, \vpar$ if $M/\Isom $ compact, and otherwise also depends on $R$.
	Then applying Theorem \ref{t:lb} to the left-hand side,
	$
	c_m^2 r^2/2 \leq C|\W \cap B_r(x)| r^{-n} + Cr^4,
	$
	so reabsorbing the $r^4$ implies the conclusion.
\end{proof}

\section{The upper bound}\label{s:ub}
In this section we consider outward minimizers of the main functional $\Ep$ defined in \eqref{eqn: main functional}. The main result of the section is Theorem~\ref{l:upper} below, which shows that the function $u_\W + \sqrt{\tpar} w_\W$ grows at most linearly away from the boundary of an outer minimizer $\W$. The result is phrased in terms of the quantity $\UP(\W)$ defined in \eqref{def: UP}. 
	Throughout this section, we fix $R>0$, $0 < v< \vmax$, and $0<\eta\leq \eta_0$ (recall Remark~\ref{rmk: eta parameter}). 
\begin{theorem}\label{l:upper}
	 There are constants $\d_M, \tpar_M, C_M > 0$ depending only on $R, v, \vmax,$ and $\vpar$ and a constant $\err_M(\tpar, R, v, \vmax, \vpar) > 0$ such that if we fix $\tpar < \tpar_M$ and then $\err < \err_M$, the following holds. Let $\W$ be an outward minimizer on $\Qb_R$ and suppose $\En(\W) \leq \minE + \d_M$,  Then we have
	\[
	\UP(\W) \leq C_M
	\]
	where $\UP(\W)$ is defined in \eqref{def: UP}.
\end{theorem}

The proof of Theorem~\ref{l:upper} will require a few initial lemmas. We start by checking that, for an outward minimizer $\W$, $u_\W$ and $w_\W$ are close to their harmonic replacements in small balls. By the harmonic replacement of $u_\W$ on $B_r(x) \cap \Qb_R$, we mean the unique function $u$ with $u - u_\W \in H^1_0(B_r(x) \cap \Qb_R)$ and $\Lap u = 0$ on $B_r(x) \cap \Qb_R$ in the weak sense, extended so that $u = u_\W$ on $\Qb_R\sm B_r(x)$.

\begin{lemma}\label{l:harmonicapprox}
 There are constants  $\d_M, \tpar_M, C_M, r_0 > 0$ depending only on $R, v, \vmax,$ and $\vpar$ and a constant $\err_M(\tpar, R, v, \vmax) > 0$ such that if we fix $\tpar < \tpar_M$ and then  $\err < \err_M$,  the following holds.  	Let $\W$ be an outward minimizer on $\Qb_R$ satisfying $\En(\W) \leq \minE + \d_M$. Then for any $x \in \Qb_R$ and $r < r_0$,
	\[
	\int |\grad (u_\W - u)|^2 + \tpar |\grad (w_\W - w)|^2 \leq C_M \Big[|B_r(x) \sm \W| + r^n \sup_{B_r(x)} (u_\W^2 + \tpar w_\W^2) \Big] \leq C_M r^n,
	\]
	where $u, w$ are the harmonic replacements of $u_\W$, $w_\W$ on $B_r(x) \cap \Qb_R$.
\end{lemma}

\begin{proof} The basic idea of the proof is make use of the outward minimality property of $\W$ using an energy  competitor $\W'$ obtained by adding a small ball to $\W$. 

{\it Step 1: The competitor $\W'$. } 
	Set $
	\W' = \W \cup (B_r(x) \cap \Qb_R)$ and $A = \sup_{B_r(x)} u_\W^2 + \tpar w_\W^2$. First, we may use that $\ei(\W') \leq \ei(\W)$ and $\tor(\W') \leq \tor(\W)$ to deduce that
	\[
		\En(\W') \leq \En(\W) + \frac{C}{\vpar} r^n \leq \d_M + \frac{C}{\vpar}r_0^n \leq 2 \d_M
	\]
	if $r_0$ is small enough. Choosing $\d_M$ and $\tpar_M$ sufficiently small and applying Lemmas  \ref{l:volumenottoobig} and \ref{l:simpleeval} shows that $\W' \in \compset$ and $\l_2(\W') > \ei(\W') + c(v, \vmax, R)$. The latter fact allows us to apply  the key Proposition~\ref{lem:key} to $\W, \W'$ to find  that
	\[
	\int |u_\W - u_{\W'}| \leq C[\ei(\W) - \ei(\W') + \tor(\W) - \tor(\W')],
	\]
	where $C$ depends only on $R$ and $v, \vmax$. So, from assumption \ref{a:nllip} on $\nl$, we know that
	\[
		|\nl(\W) - \nl(\W')| \leq |\W \triangle \W'| + \int |u_\W - u_{\W'}| \leq |B_r(x) \sm \W| + C[\ei(\W) - \ei(\W') + \tor(\W) - \tor(\W')].
	\]
Now, we use $\W'$ as an energy competitor for $\W'$; the outward minimizing property gives that
	\begin{align*}
		0 \leq \Ep(\W') - \Ep(\W) 
		&\leq \ei(\W') - \ei(\W) + \tpar (\tor(\W') - \tor(\W)) + \frac{C}{\vpar} |\W' \sm \W| + \err |\nl(\W) - \nl(\W')| \\
		&\leq [\ei(\W') - \ei(\W) + \tpar ( \tor(\W') - \tor(\W) )] \left(1 - C \frac{\err}{\tpar}\right) + \frac{C}{\vpar} |B_r(x) \sm \W|.
	\end{align*}
	Select $\err_M$ small enough that $1 - C \frac{\err_M}{\tpar} < \frac{1}{2}$; then
	\begin{equation}\label{e:harmonicapprox}
		\ei(\W) - \ei(\W') + \tpar ( \tor(\W) - \tor(\W') ) \leq C |B_r(x) \sm \W|.
	\end{equation}

{\it Step 2: Estimates for the eigenvalue and torsional rigidity of $\W'$.}
	Next, we use $u$ as a competitor for $\ei(\W')$ to find  $\ei(\W') \leq {\int |\grad u|^2}/{\int u^2}$.
	To estimate this quantity further, from the maximum principle we know that $u^2 + \tpar w^2 \leq C T$ on $B_r(x)$ as well. This can be used to bound the denominator:
	\[
		\int u^2 = \int u_\W^2 + \int_{B_r(x)\cap \Qb_R} u^2 - u_\W^2 \geq 1 - C A |B_r(x)| \geq 1 - C A r^n.
	\]
	For the energy term,
	\begin{align*}
		\int |\grad u|^2 &= \ei(\W) + \int_{B_r(x) \cap \Qb_R} |\grad u|^2 - |\grad u_\W|^2 \\
		 &= \ei(\W) + \int_{B_r(x) \cap \Qb_R} g(\grad (u - u_\W), \grad (u + u_\W)) \\
		 &= \ei(\W) - \int_{B_r(x) \cap \Qb_R} g(\grad (u - u_\W), \grad (u - u_\W)) = \ei(\W) - \int_{B_r(x) \cap \Qb_R} |\grad (u - u_\W)|^2,
	\end{align*}
	where the second-to-last step used that $\int g(\grad u, \grad (u - u_\W)) = 0$ since $\Lap u = 0$ and $u - u_\W \in H^1_0(B_r(x)\cap \Qb_R)$. Putting these together leads to
	\[
		\ei(\W') \leq \ei(\W) - \int |\grad (u - u_\W)|^2 + C Ar^n.
	\]
	A similar computation (using Lemma \ref{l:torsionfunctionbdd}) gives
	\[
		\tpar \tor(\W') \leq \tpar \tor(\W) - \frac{\tpar}{2} \int |\grad (w - w_\W)|^2 + C A r^n.
	\]
	
	{\it Step 3: Conclusion.}
	Plugging these into \eqref{e:harmonicapprox},
	\[
		\int |\grad (u - u_\W)|^2 + \tpar \int |\grad (w - w_\W)|^2 \leq C [A r^n + C |B_r(x) \sm \W|].
	\]
	Recalling from Lemmas \ref{l:efbdd} and \ref{l:torsionfunctionbdd} that $A\leq C(R,  \vmax)$ and $|B_r(x) \sm \W| \leq |B_r(x)|\leq C r^n$ concludes the proof.
\end{proof}

Lemma~\ref{l:harmonicapprox} immediately gives that $u_\W$, $w_\W$ are H\"{o}lder continuous functions from \cite{Camp65}, and also satisfy a Morrey-type estimate:

\begin{corollary}\label{c:badreg}
	Let $\W$ be as in Lemma \ref{l:harmonicapprox}. Then $u_\W, \tpar w_\W \in C^{0, \a}(\Qb_R)$ for any $\a < 1$, with
	\begin{equation}\label{e:calpha}
		[u_\W]_{C^{0, \a}} + \sqrt{\tpar} [w_\W]_{C^{0, \a}} \leq C \sup_{x \in \Qb_R, r < r_0} r^{-n/2 + 1 - \a} \left[\|\grad u_\W\|_{L^2(B_r(x))} + \tpar \|\grad w_\W\|_{L^2(B_r(x))} \right] \leq C
	\end{equation}
	for $C=C(R, v, \vmax, \vpar, \a)$.
\end{corollary}

Below we show how to instead obtain the more precise estimate $u_\W, w_\W \in C^{0,1}$, which will follow easily from the growth estimate of Theorem~\ref{l:upper}. Before proving Theorem~\ref{l:upper}, let us recall some basic facts about Green's functions and prove a  mean value-type inequality.
Consider the (positive) Green's function $G(x, \cdot)$ for $B_r(x)$ (with pole at the center of the ball). For any continuous function $\phi$ which is smooth on a neighborhood of $x$ and has $\Lap \phi$ represented by a finite Borel measure, we have that
	\[
	\phi(x) + \int_{B_r(x)} G(x, y)  d \Lap \phi(y) = \int_{\p B_r(x)} g(\grad G(x, y),  \nu_y) \phi(y) d\cH^{n-1}(y)
	\]
	where $\nu_y$ is the outward unit normal to $\p B_r(x)$. From, say, \cite[Theorem 1.2.8]{Kenig}, we have the following standard bounds on $G$:
	\begin{equation}\label{e: gradient greens}
	-C r^{1 - n} \leq  g(\grad G(x, y),  \nu_y) \leq -c r^{1 - n} \qquad  y \in \p B_r(x)
	\end{equation}
	and
	\begin{equation}\label{e:upper5}
	c \left( |y - x|^{2 - n} - r^{2 - n} \right) \leq G(x, y) \leq C\left( |y - x|^{2 - n} - r^{2 - n} \right) \qquad y \in B_r(x)\sm \{x\}
	\end{equation}
	if $n \geq 3$, while
	\[
	c ( \log |y - x| - \log r ) \leq G(x, y) \leq C( \log |y - x| - \log r )
	\]
	instead when $n = 2$, with constants depending only on $R$ (i.e. they are uniform in $x$ and $r$).
	 
\begin{lemma}[Mean value-type inequality]\label{l: MVP}
 Let $\phi $ be a nonnegative continuous function that is either (1) smooth in a neighborhood of $x$ or (2) satisfies the Morrey-type estimate \eqref{e:calpha}, has $\phi(x)=0$,  and whose distributional Laplacian $\Delta \phi$ is a Radon measure with locally finite total variation. Then
		\begin{equation}\label{e:upper2}
		\frac{1}{r^{n-1}}\int_{\p B_r(x)} \phi \leq C \left[\phi(x) + \int_0^r s^{1 - n} |\Lap \phi|(B_s(x)) ds\right].
	\end{equation}
\end{lemma}
\begin{proof}
Let us first assume that $\phi \geq 0$ is smooth in a neighborhood of $x$.	When $n\geq 3$, making use of the upper bound in \eqref{e:upper5}, we have 
	\begin{align*}
	  \int_{B_r(x)} G(x, y) d \Lap \phi(y) 
	&\leq \int_{B_r(x)} G(x, y) d |\Lap \phi|(y) \\
	&\leq C\int_0^r \int_{\p B_s(x)} \left[s^{2 - n} - r^{2 - n}\right] d |\Lap \phi| ds \\
	& =  C\int_0^r \left[s^{2 - n} - r^{2 - n}\right] \int_{\p B_s(x)}  d |\Lap \phi|ds  =  C (n - 2) \int_0^r s^{1 - n} \int_{B_s(x)} d |\Lap \phi| ds.
	\end{align*}
	The last step was an integration by parts in one variable, using that $ s^{2 - n} \int_{B_s(x)} |\Lap \phi| \rightarrow 0$ as $s \rightarrow 0$ (recall that $\Lap \phi$ is smooth near $0$). 
	From \eqref{e: gradient greens}, we have 
	 \[
	 \frac{1}{r^{n-1}}\int_{\p B_r(x)} \phi - C \phi(x) \leq C \left[- \phi(x) - \int_{\p B_r(x)} g_y(\grad G(x, y),  \nu_y) \phi(y)d\cH^{n-1}(y)\right] =C \int_{B_r(x)} G(x, y) d \Lap \phi(y).
	 \]
Together these two estimates give \eqref{e:upper2} for $n\geq 3$. A similar computation gives an identical estimate when $n = 2$.

Now, suppose that instead of $\phi$ being smooth in a neighborhood of $x$, we simply know that $\phi$ satisfies \eqref{e:calpha} and $\phi(x)=0$.
	For fixed $t \in (0,r/2)$, let $\eta_t$ be a cutoff function which vanishes on $B_t(x)$, is $1$ outside of $B_{2t}(x)$, and has $|\grad \eta_t| \leq C t^{-1}$, $|D^2 \eta_t|\leq C t^{-2}$. In this way, $\phi \eta_t $ is smooth in a neighborhood of $x$ and we may apply \eqref{e:upper2} to $\phi \eta_t $.
	\[
		\Lap (\phi\eta_t) = \eta_t \Lap \phi +  2 g(\grad \phi, \grad \eta_t) + \phi \Lap \eta_t
	\]
	as distributions (with the second two terms absolutely continuous), so
	\begin{align*}
		|\Lap (\phi \eta_t)|(B_s(x)) &\leq |\Lap \phi|(B_s(x) \sm B_t(x)) + \int_{B_{2t}(x)} 2 g(\grad \phi, \grad \eta_t) + \phi \Lap \eta_t\\
		 &\leq |\Lap \phi|(B_s(x)) + C t^{n/2 + \a - 1} t^{n/2 - 1} + C t^{\a} t^{n - 2}\\
		 &\leq |\Lap \phi|(B_s(x)) + C t^{n - 2 + \a}
	\end{align*}
	for $ s \geq t$, and $0$ otherwise.  Then \eqref{e:upper2} applies to $\eta_t \phi$, to give
	\begin{align*}
		\frac{1}{r^{n-1}}\int_{\p B_{r}(x)} \phi &\leq C \left[\eta_t \phi (x) + \int_t^{r}s^{1 - n}  |\Lap \phi|({B_s(x)})  + C s^{1 - n} t^{n - 2 + \a}  ds\right] \\
		&\leq C \left[0 + \int_t^{r}s^{1 - n}  |\Lap \phi|({B_s(x)})  ds + Ct^{\a} \right]\,.
	\end{align*}
	Sending $t \rightarrow 0$ completes the proof.	
\end{proof}

We are now ready to prove Theorem~\ref{l:upper}.
\begin{proof}[Proof of Theorem~\ref{l:upper}]
{\it Step 1: Setup. }
	Fix a point $x \in \partial \W$ and an $r < r_0$, with $r_0\leq \text{inj}_M/4$ to be chosen below. We will show that
	\[
		S = S(x, r) := \sup\{ u(y) :  y \in \p B_r(x) , \ d(y, \p \W) = r \}  \leq C_M r,
	\]
	where $u$ is either $u_\W$ or $\sqrt{\tpar} w_\W$.	As we know that $u_\W, w_\W$ are bounded, the same inequality is automatic for $r \in [r_0, 1]$ with constant $C/r_0$. This will imply the conclusion, as for every $y \in \W$ with $d(y, \p \W) \leq 1$, there is an $x \in \p \W$ and $r$ with $r = d(y, \p \W) = d(y, x)$, so $u \leq S(x, r)$.

If $B_{2r}(x)$ is not fully contained in $\Qb_R$, we have that
$
		u_\W \leq C w_\W \leq C w_{\Qb_R}
$
	on $\W$ (using Lemma \ref{l:torsionbound}), and from elliptic regularity and the smoothness of $\Qb_R$, $w_{\Qb_R}(y) \leq C(R) d(y, \partial \Qb_R)$. This gives
	\[
		S \leq C \sup_{\p B_r(x)} w_{\Qb_R} \leq C r,
	\]
	and we are done. Assume, then, that $B_{2r}(x) \ss \Qb_R$.
s

{\it Step 2: Initial $S$ bound.} In this step, we bound $S$ from above in terms of the integral of $u$ over spheres $\p B_s(x)$ for $s \in (0,1)$ using a multiple of the Green's function as a barrier.  
More specifically, let $z \in \partial B_r(x) \cap \{y: d(y, \p \W) = r  \}$ be a point such that $u(z) = S$. By definition $B_r(z) \ss \W$, and so from Lemma \ref{l:basic} we have $|\Lap u|\leq C$
	on this ball. Applying the Harnack inequality \cite[Theorems 8.17, 8.18]{GT}, 
	\[
		S \leq \sup_{B_{r/2}(z)} u \leq C\inf_{B_{r/2}(z)} u + C r^2.
	\]
	If $S \leq r$, we are done; if not, choose $r_0$ small enough in terms of $C$ so  that $Cr^2 \leq \frac{1}{2} r \leq \frac{1}{2} S$, and thus
	\begin{equation}\label{e:upper4}
		\inf_{B_{r/2}(z)} u \geq \frac{S}{C}.
	\end{equation}
	In order to propagate this bound to balls centered at $x$ (at least in an integral sense), consider the barrier $v(y) = \frac{S}{C_*}r^{n - 2} G(z, y)$, where $G(z, y)$ is the Green's function for $B_r(z)$. On $\p B_r(z)$, we have $v = 0 \leq u$, while $v \leq u$ on $\p B_{r/2}(z)$  so long as $C_*$ is chosen large relative to the constants in \eqref{e:upper4} and \eqref{e:upper5}. On the annular region between them, $\Lap v = 0$ and $\Lap u \leq 0$. It follows from the comparison principle that $v \leq u$ on $B_r(z) \sm B_{r/2}(z)$. In particular, in conjunction with \eqref{e:upper4}, this means that for any $t \in (\frac{r}{2}, r)$ we have
	\[
		\inf_{B_{t}(z)} u \geq c S \left(\frac{r^{n - 2}}{t^{n - 2}} - 1\right) \geq c S \frac{r - t}{r}.
	\]
	For any $s \in (0, r)$, choose $t\in (r/2,r)$ so that $r - t = s/2$. One may check that the set $B_t(z) \cap \p B_s(x)$ has measure of this set is bounded from below by $c s^{n-1}$. So, integrating over this set and applying the previous estimate, we have
	\begin{equation} \label{e:upper1}
		\int_{\p B_s(x)} u \geq \int_{\p B_s(x) \cap B_{t}(z)} u \geq c s^{n-1} \frac{s}{r} S
	\end{equation}
	for all $s \in (0, r)$.	This is the first of two estimates which will be combined to bound $S$.\\
	
Ultimately, in Step 4, we will use the mean value-type inequality of Lemma~\ref{l: MVP} in order to bound the left-hand side of \eqref{e:upper1}. So, it is essential to estimate the $|\Delta|$ measure of balls $B_s(x)$.

{\it Step 3: Estimate for $|\Delta|$.}	
	The second estimate involves computing the total variation of $\Lap u$ when viewed as a measure, using Lemma~\ref{l:harmonicapprox}. From Lemma \ref{l:basic}, we know that $\Lap u = \mu - f$ in the sense of distributions, where $\mu$ is a nonnegative Borel measure supported on $\p \W$ and $f$ is either $\ei(\W)u_\W$ or $\tpar 1_\W$, and in each case is uniformly bounded. Fix $s \leq r$ and let $h$ be the harmonic replacement of $u$ on $B_s(x)$: from Lemma \ref{l:harmonicapprox},
	\[
		\int_{B_s(x)} |\grad (u - h)|^2 \leq C s^n.
	\]
	Using that $\Lap h = 0$  on $B_s(x)$ and $u = h$ on $\p B_s$,
	\[
		\int_{B_s(x)} (h - u) d \Lap u = - \int_{B_s(x)} g(\grad u, \grad(h - u)) = \int |\grad(u - h)|^2 \leq C s^n.
	\]
	The leftmost identity used that $\Lap u$ represents the Laplacian of $u$ in the sense of distributions. Decomposing this further,
	\[
		\int_{B_s(x)} (h - u) d \Lap u = \int_{B_s(x) \cap \p \W} (h - u) d \mu - \int_{B_s(x) \cap \W} (h - u) f.
	\]
	The second term is controlled by $C s^n$, as both $h, u$ are bounded. In the first term, $u = 0$,  so
	\begin{equation}\label{e:upper6}
		\int_{B_s(x) \cap \p \W} h d \mu \leq C s^n.
	\end{equation}
We thus need to bound $h$ from below on this set. 	Applying the Harnack inequality followed by  \eqref{e:upper2} to $h$ gives that
	\[
	C \inf_{B_{s/2}(x)} h  \geq 	h(x) \geq c s^{1-n} \int_{\p B_s(x)} h. 
	\]
	Since $h=u$ on $\p B_s(x)$, we find from  \eqref{e:upper1} that 
	$
		\inf_{B_{s/2}(x)} h  \geq c \frac{s}{r} S.
	$
	Combining this with  \eqref{e:upper6}, we get
	\[
		c \frac{s}{r} S\mu(B_{s/2}) \leq \mu(B_{s/2}) \inf_{B_{s/2}(x)} h \leq \int_{B_s(x) \cap \p \W} h d \mu \leq C s^n.
	\]
	So, we arrive at our second main estimate:
	\begin{equation}\label{e:upper3}
		|\Lap u|(B_{s/2}) \leq \frac{C r}{S} s^{n-1} + C s^n \leq \frac{C r}{S} s^{n-1},
	\end{equation}
	 where the very last step used that $s \leq r$ and $S \leq C$ is bounded.

{\it Step 4: Conclusion.} By Corollary~\ref{c:badreg}, and the fact that $u(x) = 0$, we may apply Lemma~\ref{l: MVP} to $u$. So, using \eqref{e:upper3}, Lemma~\ref{l: MVP} tells us that 
\[
\frac{1}{r^{n-1}}\int_{\p B_{r/4}(x)} u \leq C\int_0^{r/4} \frac{ r}{S}  ds =C \frac{r^2}{S}.
\]
Applying \eqref{e:upper1} to the left-hand side
	gives $S^2 \leq C r^2$. This bounds $S$, completing the proof.
\end{proof}

By applying standard elliptic estimates on $B_d(x)$ for any $x \in \W$, $d = \min\{r_0, d(x, \p \W)/2 \}$ and using that $u_\W + \sqrt{\tpar} w_\W \leq C d \UP(\W)$ on this ball, we obtain the following gradient estimate as well:

\begin{corollary} \label{cor:lip}
	Under the assumptions of Theorem~\ref{l:upper},
	\[
		|\n u_\W | + \sqrt{\tpar} |\n w_\W|\leq C(\UP(\W), R, v, \vmax).
	\]
\end{corollary}

\section{Existence of minimizers}\label{s:exist}

We are now in a position to show that minimizers to the main functional $\Ep$ defined in \eqref{eqn: main functional} exist. Theorem~\ref{t:min} below establishes the existence of minimizers of $\Ep$ among sets in $\compset_{R, \vmax}$ for any Riemannian manifold. In the case when $M/G_0$ is compact for a subgroup $G_0$ of the isometry group for which the functional is invariant, Theorem~\ref{t:globalexist} shows that the parameter $R$ may be taken to $\infty$ and global minimizers exist. Throughout the section, we fix $R>0$ and $0<v<\vmax$, and assume that $\vpar < \vpar_0(R, v, \vmax)$ and $\tpar < \tpar_0(R, v, \vmax, \vpar)$ are fixed to be small enough that all results in Sections \ref{s:basics}, \ref{s:lb}, and \ref{s:ub} apply.

In the following lemma, we show that any set $\W$ with low base energy is contained in an outward minimizer whose energy does not exceed the energy of $\W$.
\begin{lemma}\label{l:outmin} There are $\d = \d(R, v, \vmax, \vpar) > 0$ and $\err_0 = \err_0(R, v, \vmax, \tpar, \vpar)$ such that if $\err < \err_0$ and $\W \in \compset$ with $\En(\W) \leq \minE + \d$, then there exists a $U \in \compset$ with $\W \ss U$, $\Ep(U) \leq \Ep(\W)$, and $U$ an outward minimizer.
\end{lemma}

\begin{proof}
	Given any bounded open set $E$, let
	\[
		m_E = \inf\{\Ep(V) : V \in \compset, E \ss V\}.
	\]
	Clearly $m_\W \leq \Ep(\W)$, and as $m_\W \geq \minE > -\infty$ from Lemma \ref{l:emin} this is bounded from below. If $m_\W = \Ep(\W)$, then $\W$ is itself an outward minimizer and we are done by letting $U=\W$.  If not, we construct the set $U$ in the following way.  Let $U_1 \in \compset$ be a set with $\W \ss U_1$ which has $\Ep(U_1) < ({m_\W + \Ep(\W)})/{2}$. Now repeat this construction with $U_1$ in place of $\W$, producing a nested sequence (finite or infinite) $U_k \in \compset$ with $\W \ss U_k \ss U_{k + 1}$, and $\Ep(U_k) <({m_{U_{k-1}} + \Ep(U_{k-1})})/{2}$. 
	
	The sequence is finite if $m_{U_k} = \Ep(U_k)$ for the final $k$, meaning it is also an outward minimizer. In this case, $U_k$ satisfies the conclusions of this lemma and we are done by taking $U= U_k$. Now  assume the sequence is infinite. Observe that $m_{U_k} \geq m_{U_{k-1}}$ from the fact that $U_{k-1} \ss U_{k}$. Set $m = \lim_{k} m_{U_k}$. On the other hand, $\Ep(U_k)$ is decreasing and
	\[
		\Ep(U_k) - m_{U_{k}} < \frac{m_{U_{k - 1}} + \Ep(U_{k-1})}{2} - m_{U_{k-1}} = \frac{\Ep(U_{k-1}) - m_{U_{k-1}}}{2},
	\]
	meaning $m = \lim_k \Ep(U_k)$.
	
	Let $U = \cup_k U_k$; this is an open set. Note that $\W \ss U \ss \Qb_R$ and $|U| = \lim |U_k| \leq \vmax$ from the monotone convergence theorem. We also have $m_U \geq m_{U_k}$ for each $k$, so $m_U \geq m$. Our goal is to show that $\Ep(U) = m$; if we do so, then clearly $U$ is an outward minimizer and satisfies the assumptions.
	
	First, $\En(U) \leq \liminf_k \En(U_k)$. Indeed, this follows by extracting a subsequence of first eigenfunctions $u_{U_k}$ and torsion functions $w_{U_k}$ for $U_k$ which converge weakly in $H^1_0(U)$ to some $u, w$ which may be used as competitors for the eigenvalue and torsional rigidity problems on $U$ (together with the previously observed fact that $|U \sm U_k|\rightarrow 0$):
	\[
		\ei(U) \leq \frac{\int |\grad u|^2}{\int u^2} \leq \liminf_k \frac{\int |\grad u_{U_k}|^2}{\int u_{U_k}^2} = \liminf_k \ei(U_k),
	\]
	and similarly $\tor(U) \leq \liminf_k \tor(U_k)$ and $\fv(U) = \lim_k \fv(U_k)$. This implies that 
	\begin{equation}\label{e:outmin1}
		\En(U) \leq \lim \Ep(U_k) \leq \Ep(\W) \leq \En(\W) + \err_0 \leq \minE + \d + \err_0.
	\end{equation}
	Set $\a = \liminf_k \En(U_k) - \En(U) \geq 0$,
	noting that
	\[
		\liminf_k \1\l_1(U_k) + \tor(U_k) - \l_1(U) - \tor(U)\2 \leq \frac{\a}{\tpar}.
	\]
	
	Consider now the remaining term $\nl(U)$. From Proposition~\ref{lem:key} applied to $U$ and $U_k$ (using Lemma \ref{l:simpleeval} and \eqref{e:outmin1}, choosing $\d$ and $\err_0$ small enough), we have that
	\[
		\int |u_{U_k} - u_U | \leq C_4 (\tor(U_k) - \tor(U) + \l_1(U_k) - \l_1(U)).
	\]
	Therefore, from property \ref{a:nllip} of $\nl$, we have that
	\[
		|\nl(U) - \nl(U_k)| \leq |U\sm U_k| + C_4 (\tor(U_k) - \tor(U) + \l_1(U_k) - \l_1(U)) \leq \frac{C_4 \a}{\tpar} + o_k(1).
	\]
	In particular, if $\err_0 C_4 < \tpar/2$, this leads to
	\begin{align*}
		m \leq m_U \leq \Ep(U)
		&= \En(U) + \err \nl(U_k) + \err |\nl(U) - \nl(U_k)|\\
		& \leq \En(U) + \err \nl(U_k) + \err \frac{C_4 \a}{\tpar} + o_k(1)\\
		& = \En(U_k) - \a + \err \nl(U_k) + \frac{\a}{2} + o_k(1) \\
		& = \Ep(U_k) - \a + \frac{\a}{2} + o_k(1) = m - \frac{\a}{2} + o_k(1).
	\end{align*}
	Taking the limit, we see that $\a = 0$, $\Ep(U) = m$, and the conclusion follows.
\end{proof}

We now show that a minimizer of $\Ep$ exists among sets in $\compset = \compset_{R,\vmax}$. The proof makes use of Lemma~\ref{l:outmin} to replace any minimizing sequence with a minimizing sequence of {\it outer minimizers}, whose eigenfunctions and torsions functions we know to be uniformly Lipschitz from the results of Section~\ref{s:ub}.
\begin{theorem}\label{t:min} There is a $\err_0 = \err_0(R, v, \vmax, \tpar, \vpar)$ such that if $\err < \err_0$, there exists an $\W \ss \compset$ that is a minimizer of $\Ep$ among sets in $\compset = \compset_{R,\vmax}$.
\end{theorem}

\begin{proof} Let $\a = \inf \{\Ep(E) : E \in \compset \}$, and $\W_j$ a minimizing sequence for $\Ep$. Note that $\En(\W) \leq \Ep(\W) \leq \En(\W) + \err$ for any set $\W$, so $\a \leq \minE + \err$ and $\En(\W_j) \leq \minE + 2\err$ for all $j$ large enough.  Apply Lemma \ref{l:outmin} to replace each $\W_j$ with an outward minimizer with smaller $\Ep(\W_j)$, choosing $2\err_0 < \d_0$ there. Let $u_{\W_j}$, $w_{\W_j}$ be the first eigenfunctions and torsion functions respectively, and pass to subsequences with $u_{\W_j} \rightarrow u$ and $w_{\W_j} \rightarrow w$ weakly in $H^1_0(\Qb_R)$ and strongly in $L^2(\Qb_R)$. Our approach here is different from the previous lemma: we show that $u$ is the unique  first eigenfunction on its positivity set: $u = u_{\{u > 0\}}$.

Apply Corollary \ref{cor:lip} to see that the $u_{\W_j}$ and $w_{\W_j}$ are equicontinuous, and so converge uniformly to $u$ and $w$ (which are Lipschitz functions) on $\Qb_R$. This also implies that $\W = \{ u > 0 \}$ is, by definition, an open set. Let us show that $\W$ is minimizer.

Let $\l = \lim \l_1(\W_j)$ (pass to a subsequence if needed). From the uniform convergence of $u_{\W_j} \rightarrow u$, we have that for every $x \in \W$, there is a ball $B_r(x)$ such that $B_r(x)\ss \W_j$ for every $j$ large enough. On this ball, the $u_{\W_j}$ converge in smooth topology to $u$ from standard elliptic estimates, and so $-\Lap u(x) = \l u(x)$ passes to the limit. In particular, this implies that $u_\W$ is an eigenfunction for $\W$ with eigenvalue $\l$.

From the weak convergence of $u_{\W_j} \rightarrow u$ and $w_{\W_j} \rightarrow w$ in $H^1_0$, we have that
\[
	\ei(\W) + \tpar \tor (\W) \leq \frac{\int |\grad u|^2}{\int u^2} + \tpar \int \frac{|\grad w|^2}{2} - w \leq \liminf_j \ei(\W_j) + \tpar \tor(\W_j).
\]
From Fatou's lemma, $|\W| \leq \liminf_j |\W_j|$, so $\En(\W) \leq \liminf_j \En(\W_j)$. We also have that $\W \ss \Qb_R$ and $|\W| \leq \vmax$, so $\W \in \compset$ is an admissible competitor.

Using that $\En(\W) \leq \liminf \En(\W_j) \leq \minE + 2\err$, we may apply Lemma \ref{l:simpleeval} to deduce that $\l_2(\W) > \ei(\W) + c(R, v, \vmax)$. As $\minE \leq \En(\W)$ as well, we must have that 
\[
	0 \leq \lim \ei(\W_j) - \ei(\W) \leq 2 \err_0,
\]
and so $|\ei(\W) - \l| \leq 2\err_0$. Ensuring that $\err_0$ is small enough in terms of $R, v, \vmax$ guarantees that $\l < \l_2(\W)$. In particular, this means there is a unique eigenfunction on $\W$ with eigenvalue at most $\l$, and so $u_\W = u$. In particular, this means that $u_{\W_j} \rightarrow u_\W$ uniformly.

Set
$
	\b = \liminf_j \En(\W_j) - \En(\W) \geq 0,
$
which has
$
	\b \geq \vpar [\liminf_j |\W_j| - |\W|].
$
From Fatou's lemma, $|\W |\leq \liminf_j |\W_j \cap \W|$, so $\liminf_j |\W\sm \W_j| = 0$. This gives that
\[
	\liminf_j |\W \triangle \W_j| = \liminf_j |\W_j \sm \W| = \liminf_j |\W_j| - |\W| \leq \frac{\b}{\vpar}.
\]

It follows from assumption \ref{a:nllip} of Definition~\ref{def: nl} on $\nl$ that $|\nl(\W) - \nl(\W_j)|\leq o_j(1) + |\W \triangle \W_j| \leq o_j(1) + \frac{\b}{\vpar}$. We can now estimate the energy of $\W$ in the following way:
\begin{align*}
	\a \leq \Ep(\W)
	& \leq \En (\W) + \err \nl(\W_j) + o_j(1) + \err \frac{\b}{\vpar}\\
	& \leq \En(\W_j) - \b + \err \nl(\W_j) + o_j(1)  + \err \frac{\b}{\vpar}\\
	& = \Ep(\W_j) - \b + o_j(1) + \err \frac{\b}{\vpar} = \a - \b + o_j(1) + \err \frac{\b}{\vpar}.
\end{align*}
If $\err < \vpar$, this implies that $\b = 0$, $\Ep(\W) = \a$, and $\W$ is a minimizer.
\end{proof}

Next, we prove that when $M/\Isom_0$ is compact, a minimizer of $\Ep$ among all open bounded sets with $|\W|\leq \vmax$ exists. The main idea is to show that the minimizers $\W_R$ from the previous theorem have uniformly bounded diameter, and thus $\W_R$ is a global minimizer for $R$ sufficiently large. 
\begin{theorem} \label{t:globalexist} Assume $M /\Isom_0$ is compact. Then exist constants $S = S( v, \vmax, \tpar, \vpar) > 0$ and  $\err_0 = \err_0(v, \vmax, \tpar, \vpar)$ such that if $\err < \err_0$, there exists an open, bounded $\W$ with $|\W|\leq \vmax$ which minimizes $\Ep$ over all such sets. For every such minimizer $\W$ of $\Ep$, there is an $e \in \Isom_0$ and an  such that $e(\W) \ss \Qb_S$.
\end{theorem}

Recall that $\Isom_0$ is a (possibly empty) subgroup of isometries under which $\nl$, and hence $\Ep$, is invariant. The assumption here implies that $M / \Isom$ is compact.

\begin{proof}
	For each $R\gg 1$, consider the minimizer  $\W_R$ obtained from Theorem \ref{t:min} of $\Ep$ over $\compset_R$. Set 
	\[
	\a_R = \Ep(\W_R) = \inf\{\Ep(\W) : \W \in \compset_R \};
	\]
	note that $\a_R$ is nonincreasing in $R$ and from Lemma \ref{l:globalElb}, we have that $\a_R \geq \En(\W_R) \geq \minE(\infty) > - \infty$. We claim that so long as $\err_0$ and $\tpar_0$ are chosen small enough, $\W_R$ has bounded diameter uniformly in $R$. Indeed, fix $r$ small and apply the lower density estimate of Corollary \ref{c:badldensitybd} to $\W_R$: for every $x\in \W_R$,
	\[
		|\W_R \cap B_r(x)| \geq c.
	\]
	As $|\W_R| \leq \vmax$, we this implies that $\W_R$ may be covered by a bounded number $K$ of balls $B_r(x_k)$, with $x_k \in \W_R$. Assume that the union of these balls is disconnected: this implies that $\W$ itself cannot have a connected component of measure greater than $|\W| - c$. Apply Lemma \ref{l:globalsimpleeval} with $\e = c/2$: this gives a contradiction to the third conclusion there. We infer that all the balls $B_r(x_k)$ have connected union of diameter at most $2Kr$, as promised.
	
	It follows that there is an $S$ and $e_R \in \Isom_0$ such that $e_R(\W_R) \ss \Qb_{S}$ for all $R$ (this uses only the diameter bound and that $\cup_{e \in \Isom_0} e(\Qb_S) = M$ for some $S$, as in the proof of Lemma \ref{l:globalElb}). Since $\Ep$ is invariant under isometries in $\Isom_0$, $\Ep(\W_R) = \Ep(e_R(\W_R))$ and $e_R(\W_R)$ is also a minimizer. This implies that for $R \geq S$,
	\[
		\a_S \geq \a_R = \Ep(\W_R) = \Ep(e_R(\W_R)) \geq \a_S.
	\]
	In particular, $\a_R$ is independent of $R$ for $R \geq S$, and $\W_S$ minimizes $\Ep$ over all open bounded sets with $|\W|\leq \vmax$.	The second conclusion follows from the fact that any open, bounded set $\W$ with $\Ep(\W) = \a_S$ lies in $\Qb_R$ for a sufficiently large $R$, and so the above argument applies to $\W$.
\end{proof}

A consequence of Theorem \ref{t:globalexist} is that all constants pertaining to minimizers of $\Ep$ may be taken independently of $R$ if $M /\Isom_0$ is compact. This allows us to avoid tracking the dependence on $R$ below, instead fixing a sufficiently large $R$ and looking at minimizers in $\Qb_R$.

\section{Measure-theoretic estimates}\label{s:measure}

Having now established the existence of minimizers of the main energy, we move toward understanding some initial measure-theoretic properties of these minimizers. Ultimately, the results of this section will be used in Section~\ref{s:el} to derive the Euler-Lagrange equation satisfied by $\W$. 

In Section~\ref{ss:density}, we first prove that $\Omega$ is an NTA domain.  This allows us to utilize an inhomogeneous boundary Harnack principle for NTA domains recently shown in \cite{AKS20} in Section~\ref{ss:greens} in order to prove some fine estimates for the Green's function. This allows us to show that the first eigenfunction of a minimizer grows at least linearly from the boundary in Proposition~\ref{l:bdryharnackef}. Proposition~\ref{l:green} contains some finer Green's function estimates that will be crucial in deriving a useful form of the Euler-Lagrange equation in the following section.
With these results in hand, in Section~\ref{ss:reducedbdry} we can recover some basic measure theoretic properties of $\W$ and understand the nontangential limits of $|\nabla u_{\Omega}|$ and $|\nabla w_{\Omega}|$ on the reduced boundary 
$\p^* \Omega$, similarly to other recent approaches in vectorial free boundary problems \cite{CSY18, MTV17}.

We assume throughout the section that $R$ and $v<\vmax$ are fixed, and that $\vpar < \vpar_0(R, v, \vmax)$, $\tpar < \tpar_0(R, v, \vmax, \vpar)$, and $\err < \err_0(R, v, \vmax, \tpar, \vpar)$ are small enough that all results in earlier sections apply.  We recall  that the growth quantities $\DO(\W)$ and $\UP(\W)$ are defined in \eqref{def: DO} and \eqref{def: UP} respectively.

\subsection{Density estimates and the NTA property}\label{ss:density}
Let us summarize some consequences of the upper and lower bounds of Sections \ref{s:lb} and \ref{s:ub}. The proofs are mostly standard and may be carried out in local coordinates, so we provide only brief sketches and references. First, we have uniform upper and lower bounds on the volume density of $\W$ and perimeter density of $\p \W$.

\begin{lemma} \label{l:densitybd}
	Let $\W \in \compset$ be a minimizer of $\Ep$. Then there are $r_0, C > 0$ depending only on $R, v, \vmax, \vpar$ such that for any $x\in \p \W$ and $r < r_0$,
	\begin{equation}\label{e:voldensitybd}
		\frac{1}{C} < \frac{|B_r(x)\cap \W|}{|B_r(x)|} < 1 - \frac{1}{C}
	\end{equation}
	and
	\begin{equation}\label{e:perdensitybd}
		\frac{1}{C} < \frac{\cH^{n-1}(\p \W \cap B_r(x))}{r^{n-1}} < C.
	\end{equation}
\end{lemma}

\begin{proof}[Sketch of proof.]
	First, from Theorem \ref{t:lb} and Theorem~\ref{l:upper}, we have that $c \leq \DO(\W) \leq \UP(\W) \leq C$. To prove the lower bound in \eqref{e:voldensitybd}, let $y \in B_{r/2}(x) \cap \W$ be a point with $u_\W(y) + \sqrt{\tpar} w_\W(y) \geq \DO(\W) r/2 $. By Corollary \ref{cor:lip} we have $|\grad (u_\W(y) + \sqrt{\tpar} w_\W)| \leq C(\UP(\W))$, and so we must have that $d(y, \p \W) \geq \frac{\DO(\W) r}{2 C(\UP(\W))} \geq c r$. So long as $r_0$ is small enough, $|B_{cr}(y)| \geq c' r^n$, $|B_r(x)|\leq C r^n$, and so
	\[
		\frac{|B_r(x)\cap \W|}{|B_r(x)|} \geq \frac{|B_{cr}(y)|}{|B_r(x)|} \geq \frac{1}{C}.
	\]

	For the upper bound in \eqref{e:voldensitybd}, first note that if $B_{r/2}(x) \cap (M \sm \Qb_R)$ is nonempty, then $|B_r(x) \sm \Qb_R| \geq c(R) r^n$ using the smoothness of $\p \Qb_R$, and this implies the estimate. If this is not the case, then $B_{r/2}(x) \ss \Qb_R$. We apply Lemma \ref{l:harmonicapprox} to learn that if $u$, $w$ are harmonic replacements for $u_\W, w_\W$ respectively on $B_{r/2}(x)$, then
	\[
		\int |\grad (u_\W - u)|^2 + \tpar |\grad (w_\W - w)|^2 \leq C\Big[|B_{r/2}(x) \cap \W| +  r^n \sup_{B_{r/2}(x)} \tpar (w_\W^2 + u_\W^2) \Big].
	\]
	Now, as $x \in \p \W$, we have that $\sup_{B_{r/2}(x)} \tpar w_\W^2 + u_\W^2 \leq C(\UP(\W)) r^2$. On the other hand, there must be a point $y \in B_{r/4}(x)$ with $\sqrt{\tpar} w_\W(y) + u_\W(y) \geq \DO(\W) r/4$. Recalling that $\Lap (\sqrt{\tpar} w_\W(y) + u_\W(y)) \geq - C(R, \vmax)$, applying \cite[Theorem 8.16]{GT} to $u_\W - u, w_\W - w$ gives
	\[
		\sup_{B_{r/2}(x)} \sqrt{\tpar} [w_\W - w] + [u_\W(y) - u] \leq C r^2.
	\]
	At $y$ then, $u(y) + \sqrt{\tpar} w(y) \geq c r - C r^2 \geq c r$ as long as $r_0$ is taken small enough. Using the Harnack inequality on $B_{r/2}(x)$, this implies $u + \sqrt{\tpar} w \geq c r$ on $B_{r/4}(x)$. By contrast, we know that on $B_{\e r}(x)$, $u_\W + \sqrt{\tpar} w_\W \leq \e r \UP(\W) \leq \frac{cr}{2}$ as long as we choose $\e$ small enough. Integrating and using the Poincar\'e inequality,
	\begin{align*}
		c r^{n + 2} &\leq  \int_{B_{\e r}(x)} |u_\W - u|^2 + \tpar |w_\W - w|^2\\
		& \leq \int_{B_{ r/2}(x)} |u_\W - u|^2 + \tpar |w_\W - w|^2 \\
		& \leq C r^2 \int_{B_{ r/2}(x)} |\grad (u_\W - u)|^2 + \tpar |\grad (w_\W - w)|^2   \leq C r^2 \left[|B_{r/2}(x) \cap \W| +  r^{n + 2}\right].
	\end{align*}
	As long as $r_0$ is small enough, the last term may be reabsorbed on the left, giving $|B_{r}(x) \cap \W| \geq c r^{n} \geq c |B_{r}(x)|$.
	
	For \eqref{e:perdensitybd}, the upper bound may be obtained as in \cite[Lemma 2.4]{KL18} or \cite{MTV17}. The lower bound follows from applying the relative isoperimetric inequality \cite[4.5.2(2)]{Federer}.
\end{proof}

The volume density bounds may be automatically improved slightly to give clean balls, or corkscrew points, instead.

\begin{definition}
	Let $\W \ss M$ be an open set. We say that $\W$ satisfies the \emph{inner (resp. outer) clean ball condition} with constant $K$ at $x \in \p \W$ if for any $r < 1$, there is a point $y$ such that $B_{r/K}(y) \ss B_r(x) \cap \W$ (resp. $B_{r/K}(y) \ss B_r(x)\sm \W$). We say that $\W$ satisfies the inner (resp. outer) clean ball condition with constant $K$ if it satisfies it at every $x \in \p \W$.
\end{definition}

\begin{corollary} \label{c:cleanball}
	Let $\W$ be as in Lemma \ref{l:densitybd}. Then there exists an $\a = \a(R, v, \vmax, \vpar) > 0$ such that for any $x\in \p \W$ and $r < r_0$ there are two balls $B_{\a r}(y) \ss (B_r(x) \cap \W)$ and $B_{\a r}(z) \ss (B_r(x) \sm \W)$. In particular, $\W$ satisfies the inner and outer clean ball condition with constant depending only on $r_0, \a$.
\end{corollary}

\begin{proof}
	For $\a$ small, take the collection $\{B_{\a r}(y) : y \in B_{r/2}(x) \sm \W \}$ and from it pick a finite-overlapping subcover $\{B_{\a r}(y_k) \}_{k = 1}^K$ of $B_{r/2}(x) \sm \W $. Then if every one of $B_{\a r / 2}(y_k)$ intersects $\p \W$ at a point $x_k$, we have from Lemma \ref{l:densitybd} that
	\[
		\cH^{n-1}(B_{\a r}(y_k) \cap \p \W) \geq \cH^{n-1}(B_{\a r/2}(x_k) \cap \p \W) \geq c (\a r)^{n-1}.
	\]
	As $|B_{r/2}(x)\sm \W|\geq c |B_{r/2}(x)| \geq c r^n$ and the $B_{\a r}(y_k)$ cover this set, we must have that $K (\a r)^n \geq c r^n$, or $K \geq c \a^{-n}$. Summing over $k$ and using the finite overlapping property and the upper bound in \eqref{e:perdensitybd},
	\begin{align*}
	 C r^{n-1} &\geq	\cH^{n-1}(\p \W \cap B_{r}(x))\\
	 & \geq c \sum_{k = 1}^K  \cH^{n-1}(B_{\a r}(y_k) \cap \p \W)  \geq c K (\a r)^{n-1}  \geq \frac{c}{\a} r^{n-1}.
	\end{align*}
	If $\a$ is taken small enough, this is a contradiction, so at least one of $B_{\a r/2}(y_k) \ss B_{r}(x) \sm \W$. The inner clean ball may be found similarly.
\end{proof}

\begin{definition}
	Let $\W \ss M$ be an open set. We say that $\W$ satisfies the Harnack chain condition with constant $K$ at $x, y \in \W$ if there is a curve $\g : [0, 1] \rightarrow \W$ with $\g(0) = x$, $\g(y) = y$, $l(\g([0, 1])) \leq K d(x, y)$, and
	\[
		d(\g(t), \p \W) \geq \frac{1}{K} \min\big\{ l(\g([0, t])),\,  l(\g([t, 1])) \big \}.
	\]
	Here $l(\g([a,b])) = \int_a^b |\dot \g| \,dt $ denotes  length. We say $\W$ satisfies the Harnack chain condition  with constant $K$ if it satisfies it at every $x, y \in \W$. We say $\W$ is nontangentially accessible (NTA) with constant $K$ if it satisfies the inner and outer clean ball conditions and the Harnack chain condition with constant $K$.
\end{definition}

The next lemma follows from known results on Bernoulli-type free boundary problems \cite{ACS87}, but we present a proof in Appendix \ref{s:appendixnta} as we are unaware of a version with coefficients and nonzero right-hand side treated in the literature.

\begin{lemma} \label{l:localharnackchain} Let $\W \in \compset$ be a minimizer of $\Ep$. Then there are $r_0, K > 0$ depending only on $R, v, \vmax$, and $\vpar$ such that for any $z\in \p \W$ and any $x, y \in \W \cap B_{r_0}(z)$, $\W$ satisfies the Harnack chain condition at $x$ and $y$ with constant $K$.
\end{lemma}

From Lemma~\ref{l:localharnackchain}, we find that a minimizer $\W$ is an NTA domain.
\begin{corollary}\label{c:nta}
	Let $\W \in \compset$ be a minimizer of $\Ep$. Then $\W$ is an NTA domain, with constant depending only on $R, v, \vmax$, and $\vpar$.
\end{corollary}

\begin{proof}
	First, observe that $\W$ is connected so long as $\vpar_0, \tpar_0, \err_0, r$ are taken small enough. Indeed, we have from Lemma \ref{l:simpleeval} that each connected component $U$ of $\W$ besides one must have $|U| < \e$. On the other hand, Lemma \ref{l:localharnackchain} implies that $B_{r_0}(z) \cap \W$ lies within a single connected component of $\W$ for every $z \in \partial \W$, which when combined with Lemma \ref{l:densitybd} gives that $|U| \geq |\W \cap B_{r_0}(z)| \geq c r_0^n$; if $\e$ is small enough then $\W$ has only one connected component.

	Consider the set $U_r = \{x \in \W : d(x, \p \W) > r \}$ for $r$ small and fixed. We claim that any two points $x, y\in U_r$ may be connected by a path $\g$ which has length bounded by $C(R, v, \vmax, \vpar, r)$ and stays a distance $1/C$ away from $\p \W$. To see this, cover $\W$ by a finitely-overlapping collection $\{B_{r}(z_k) \}_{k = 1}^K$ with $z_k \in \W$. As $|\W|\leq C$, we have that $K \leq C$ from the finite-overlapping property and Lemma \ref{l:densitybd}. For any two balls $B_r(z_k), B_r(z_j)$ with nontrivial intersection, we have two possibilities: either both $B_r(z_k), B_r(z_j) \ss U_r$, or at least one of them (say $B_r(z_k)$) intersects $\p U_r$. In the first case, any pair of points $x \in B_r(z_k), y\in B_r(z_j)$ may be connected by a curve of length $2r$ which stays $r$ away from $\p \W$ as it stays inside $B_r(z_k) \cup B_r(z_j$). In the second case, we have that as long as $r$ is small enough, $B_r(z_k) \cup B_r(z_j) \ss B_{r_0}(x)$ for some $x \in \p \W$. Applying Lemma \ref{l:localharnackchain} gives that for any $x \in B_r(z_k) \cap U_r$ and any $y \in B_r(z_j) \cap U_r$, $x$ and $y$ may be connected by a curve of length $Cr$ staying $r/C$ away from $\p \W$.
	

Take a graph with vertices $\{z_k\}$ and with an edge between $z_k$ and $z_j$ if and only if $B_r(z_k), B_r(z_j)$ have nontrivial intersection.  As these balls cover $\W$ and $\W$ is connected, the graph must also be connected, and so any two balls may be connected by a chain of distinct, pairwise overlapping balls. For any $x \in B_r(z_k)$ and $y\in B_r(z_j)$ with $x, y \in U_r$, find such a path of pairwise intersecting balls $B_r(z_{i_m})$, $m = 1, \ldots, J$, with $k = i_1$ and $j = i_J$. We have shown that $z_{i_m}$ and $z_{i_{m+1}}$, as well $x$ and $z_{i_1}$, $z_{i_J}$ and $y$, may be connected by curves of length $Cr$ and staying a distance $C/r$ from $\p \W$. By concatenating these curves (and using that $J \leq K$ is bounded), we see that $x$ and $y$ may be connected by a curve of uniformly bounded length remaining a distance $r/C$ from the boundary, as promised.

	This, together with Lemma \ref{l:localharnackchain}, shows that $\W$ satisfies the Harnack chain condition. Combining with Corollary \ref{c:cleanball} implies $\W$ is an NTA domain.
\end{proof}

\subsection{Estimates on the Green's function}\label{ss:greens}

In Lemma \ref{l:torsionbound}, we saw from a basic maximum principle argument that $u_\W \leq C w_\W$ for essentially arbitrary domains. The opposite inequality is far more subtle, and will in general fail even on polygonal domains in $\R^n$ (in particular, it does not follow from the NTA property of Corollary \ref{c:nta}). Nonetheless, we show in Proposition~\ref{l:bdryharnackef} below that it is, in fact, valid for \emph{minimizing} $\W$.

Let $G_\W(x, y)$ be the positive Green's function for $\W$ with pole at $x$. From \cite[Theorem 1.2.8]{Kenig}, we have that
\begin{equation}\label{e:green}
	\begin{split}
	G_\W(x, y) \leq C d^{2 - n}(x, y) &\qquad \forall x, y \in \Qb_R \\
	G_\W(x, y) \geq c d^{2 - n}(x, y) & \qquad \forall x, y \in \W,\quad d(x, y) \leq \frac{1}{2}d(x, \p \W),
	\end{split}
\end{equation}
where the constants depends only on $R$. If $n = 2$ the same is valid with $-\log d(x, y)$ in place of $d^{2-n}(x, y)$; we will only consider the case of $n > 2$ below, but all estimates remain valid if $n = 2$ after similar modification. We remark that similar estimates for the Green's function were obtained in \cite{CMMR}.

\begin{lemma} \label{l:bdryharnackgreen}
	Let $\W \in \compset$ be a minimizer of $\Ep$. Then for every $c_0 > 0$ there is a $C_0$ depending only on $R, v, \vmax, \vpar, c_0$ such that $w_\W(y) \leq C_0 G_\W(x, y)$ for any $y \in \W$ and any $x \in \W$ with $d(x, \p \W) > c_0$.
\end{lemma}

\begin{proof}
	Choose $3r_0 \leq \min\{ c_0, \text{inj}_M\}$. Let $K$ be the NTA constant of $\W$ from Corollary \ref{c:nta}. Let us first consider the case that $y$ is such that  $d(y, \p \W) \geq r_0/K$. From the Harnack chain property of  Corollary \ref{c:nta}  we may find a sequence of finitely many balls (the number depending only on $r_0$ and $K$) $B_{cr_0}(y_k) \ss \W$ so that $B_{cr_0/2}(y_k) \cap B_{cr_0/2}(y_{k+1})$ is nonempty and $y_1 = y$ and  $y_J = x$. Applying the Harnack inequality to $G_\W(x, \cdot)$ on each ball, we get
	\[
	\sup_{z \in B_{c r_0/2}(y_k)} G_\W(x, z) \leq C \inf_{z \in B_{c r_0/2}(y_k)} G_\W(x, z)
	\]
	for $k < J$, while the assumption that $d(x \W) >c_0$ guarantees that $G_\W(x, z) \geq c r_0^{2 - n}$ for $z \in B_{c r_0 /2}(x)$ from the Green's function lower bound in \eqref{e:green}. Together these guarantee  that for any such point $y$ we have
	\begin{equation}\label{e:bdryharnackgreen2}
		G_\W(x, y) \geq c(r_0).
	\end{equation}

	On the other hand, from Lemma \ref{l:torsionfunctionbdd} $w_\W(y) \leq C$ at $y$, while for any $z \in \W$ applying Theorem \ref{t:lb} and Lemma \ref{l:torsionbound} gives
	\begin{equation} \label{e:bdryharnackaux}
	w_\W(z) \geq c [u_\W(z) + \sqrt{\tpar} w_\W(z)] \geq c \DO(\W) d(z, \p \W).
	\end{equation}
	
	Now apply \cite[Theorem 2.2]{AKS20} on $B_{r_0}(z)$ for any $z \in \p \W$, with $u_1 = w_\W/w_\W(y)$ and $u_2 = G_\W(x, \cdot)/G_\W(x, y)$, where $y \in B_{r_0}(z)$ is a point with $d(y, \p \W) \geq r_0/K$ (such a point $y$ exists from the inner clean ball property). In the theorem, we set $U = \W$, $Q$ the collection of NTA domains, $\b = 1$, $\z = 2$, $x^0 = y$, $V = \{ u \in C(\bar\W) \cap C^2(\W) : u \geq 0, |\Lap u| \leq C\}$ and $H$ the solution to the Dirichlet problem for $\Delta$. The assumptions (P1-P7) follow from standard elliptic estimates and (P8) from \cite[Lemma 1.3.7]{Kenig}. Using \cite[Remark 2.4]{AKS20}, we only need to verify the growth bound (2.1) for $w_\W$, and it follows from \eqref{e:bdryharnackaux}. We conclude that
	\begin{equation}\label{e:bdryharnackgreen1}
	w_\W(a) \leq C G_\W(x, a)
	\end{equation}
	for all $a \in B_{r_0}(z)$. Applying at any $z$, \eqref{e:bdryharnackgreen1} remains valid for each $a \in \W$ with $d(a, \p \W) < r_0$. For $a$ with $d(y, \p \W) \geq r_0$, \eqref{e:bdryharnackgreen1} follows from \eqref{e:bdryharnackgreen2} and that $|w_\W|\leq C$ instead.
\end{proof}

We use the Green's function as a barrier for $u$ from below and apply Lemma~\ref{l:bdryharnackgreen} to show that the torsion function bounds the eigenfunction from below:
\begin{proposition}\label{l:bdryharnackef}
	Let $\W \in \compset$ be a minimizer of $\Ep$. Then $w_\W \leq C(R, v, \vmax, \vpar) u_\W$ on $\W$, and in particular
	\[
	\sup_{B_r(x)} u_\W \geq \frac{1}{C} \DO(\W) r
	\]
	for all $x \in \p \W$ and $r \in (0, 1)$.
\end{proposition}

\begin{proof}

	We have that $|\W| \leq \vmax$ and $\int u_\W^2 = 1$. This means there must be a point $x\in \W$ with $u_\W(x) \geq c$, and by Theorem~\ref{l:upper}, we have $d(x, \p \W) \geq c/\UP(\W) \geq c$. From the Lipschitz estimate of Corollary \ref{cor:lip}, we still have $u \geq c/2$ on a ball $B_{r}(x) \ss \W$, where  $r >0$ depends on $c$ and $\UP(\W)$, and thus on $R, v, \vmax,$ and $ \vpar$. Using the Green's function upper bound in \eqref{e:green}, we have $G_\W(x, y) \leq c r^{2 - n} \leq c$ on $\p B_{r}(x)$. As $\Lap u_\W \leq 0$ on $\W$, we may use $c G_\W(x, \cdot)$ as a barrier from below for $u_\W$ on the set $\W \sm \bar{B}_{r}(x)$;  the comparison principle implies  that
	\[
	u_\W(y) \geq c G_\W(x, y) \geq c w_\W(y),
	\]
	with the last inequality from Lemma \ref{l:bdryharnackgreen}. On $B_r(x)$, on the other hand, $u_\W \geq c/2$ while $w_\W \leq C$, so the same inequality follows. This completes the proof.
\end{proof}

Next, we prove sharp bounds on $G_\W$ near $\p \W$. To state these in a more useful fashion, recall that the \emph{harmonic measure} on $\W$, which we denote by $\w_x$, is defined as follows: for any $f \in C(\p \W)$, let $\bar{f} \in C(\bar{\W})$ be the unique classical solution to the Dirichlet problem
\[
\begin{cases}
\Lap \bar{f} = 0 & \text{ on } \W\\
\bar{f} = f & \text{ on } \p\W.
\end{cases}
\]
Such a solution $\bar{f}$ may be obtained using Perron's method or by approximation schemes, noting that Lemma \ref{l:densitybd} guarantees that $\W$ satisfies the Wiener criterion (see \cite[Theorem 8.31]{GT}). Then $f \mapsto \bar{f}(x)$ is a linear functional from $C(\p \W)$ to $\R$, and from the maximum principle it has norm bounded by $1$ and positive (i.e. $f \geq 0 \implies \bar{f}(x) \geq 0$). From the Riesz representation theorem, there is a positive Borel measure $\w_x$ on $\p \W$ such that 
\[
\int_{\p \W} f d\w_x = \bar{f}(x).
\]
This is the harmonic measure; we clearly also have $\w_x(\p \W) = 1$.

\begin{proposition}\label{l:green} Let $\W \in \compset$ be a minimizer of $\Ep$. Then the Green's function $G_\W$ satisfies
	\begin{equation}\label{e:lgreen2}
	c \frac{d(x, \p \W) d(y, \p \W)}{d^n(x, y)} \leq	G_\W(x, y) \leq C \frac{d(x, \p \W) d(y, \p \W)}{d^n(x, y)}.
	\end{equation}
	Moreover, for any $x\in \p \W$, and any $y \in \W$ with $d(y, x) > 4 r$, we have
	\begin{equation} \label{e:lgreen}
	c \frac{d(y, \p \W)}{d^{n}(x, y)} \leq \frac{\w_y(B_r(x))}{r^{n-1}} \leq C \frac{d(y, \p \W)}{d^{n}(x, y)}
	\end{equation}
	All constants depend only on $v, \vmax, R,$ and $\vpar$.
\end{proposition}

\begin{proof}
	For any point $a\in W$, set $d_a = d(a, \p \W)$. Take two points $x, y \in \W$. If $d(x, y) \leq \frac{1}{2} d_y$, then \eqref{e:lgreen2} follows directly from \eqref{e:green} (using the diameter bound for $\W$ in the lower bound) and there is nothing more to show. We now consider the case of $d(x, y) \geq \frac{1}{2} \max\{d_x, d_y\}$, breaking up the estimate into several cases depending on the locations of $x$ and $y$.  The notation $s \approx t$ below stands for $c s \leq t \leq C s$ with constants depending only on $v, \vmax, R, \vpar$.\\
	
{\it Case 1: $2d(x,y) \geq d_y \geq r_0$.} 	First, assume that $d_y \geq r_0$ for some fixed $r_0$ to be chosen below. We have from Lemma \ref{l:bdryharnackgreen} that
	\[
	c d_x \leq  w_\W(x) \leq C(r_0) G_\W(y, x)
	\]
	for any $x\in \W$. On the other hand, from \eqref{e:green} we have $G_\W(y, x) \leq C(r_0)$ for $x \in \p B_{r_0/2}(y)$, while $u_\W(x) \geq c d_x$ there from Proposition~\ref{l:bdryharnackef}. So, we may use a multiple of $G_\W(y, \cdot)$ as a lower barrier for $u_\W$ on the set $\W \sm B_{r_0/2}(y)$: as $\Lap u_\W \leq 0$, from the comparison principle this leads to
	\[
	G_\W(y, x) \leq C u_\W(x) \leq C \UP(\W) d_x
	\]
	for $x \notin B_{r_0/2}(y)$. Together these two estimates show that in the case when $d_y \geq r_0$ and $d(x, y) \geq \frac{1}{2} d_y \geq \frac{r_0}{2}$, we have
	\begin{equation} \label{eq:1}
	G_\W(y, x) \approx d_x \approx \frac{d_x}{r_0^{n-1}} \approx \frac{d_x d_y}{d^n(x, y)}.
	\end{equation}
\\

{\it Case 2: $d_y \geq d(x,y)/4K.$ }	Now fix $r_0$ small enough that $B_{2 r_0}(y_0) \ss \W$ for some reference point $y_0$, using the interior clean ball condition from Lemma \ref{c:nta}. The next two cases we treat are when $d_y \geq \frac{1}{4K} d(x, y)$, with $K$ the NTA constant of $\W$. Take a point $a \in \p B_{d_y/4}(y)$. The basic estimates \eqref{e:green} apply to this point to give $G_\W(y, a) \approx d^{2-n}(y, a) \approx d_y^{2-n}$. If $d_x \geq \frac{1}{4 K} d_y$, then $d_x \approx d_y \approx d(x, y)$ are all comparable and we may use the Harnack chain condition to find a uniformly bounded chain of balls $B_r(x_i)\ss \W$ with $r \approx d_y$ connecting $a$ and $x$; applying the Harnack inequality finitely many times gives
	\begin{equation} \label{eq:2}
		G_\W(y, x) \approx G_\W(y, a) \approx d_y^{2-n} \approx \frac{d_y d_x}{d^n(x, y)}.
	\end{equation}
	We may therefore assume that $d_x \leq \frac{1}{4 K} d_y$. In this case, let $z \in \p \W$ be a point with $d_x = d(x, z)$, and find an $a'$ with $B_{{d_y}/{4K}}(a')  \ss B_{d_y/4}(z) \cap \W$ using the interior clean ball property. Since $d_{a'} \geq d_y/4K$, we have $G_\W(y, a') \approx d_y^{2-n}$ from \eqref{eq:2}. On the other hand, from \eqref{eq:1} we have $G_\W(y_0, a') \approx d_y$. Applying the boundary Harnack principle \cite[Lemma 1.3.7]{Kenig} to $G_\W(y, \cdot)$ and $G_\W(y_0, \cdot)$ on $B_{d_y/2}(z)$ (note that this excludes the pole at $y$), we have that
	\[
		\frac{G_\W(y, x')}{G_\W(y_0, x')} \approx \frac{G_\W(y, a')}{G_\W(y_0, a')} \approx \frac{d_y^{2-n}}{d_y} \approx \frac{d_y}{d^{n}(x, y)}
	\]
	for any $x' \in B_{d_y/4}(z)$ (with the last step using $d_y \approx d(x, y)$). In particular this is valid at $x' = x$, leading to
	\begin{equation}\label{eq:3}
		G_\W(y, x) \approx G_\W(y_0, x) \frac{d_y}{d^{n}(x, y)}  \approx \frac{d_x d_y}{d^{n}(x, y)}
	\end{equation}
	using \eqref{eq:1} again.
\\

{\it Case 3: $\max \{d_y, d_x\} \leq  d(x, y)/{4K}$.}	
	The only case remaining is when $\max \{d_y, d_x\} \leq \frac{1}{4K} d(x, y)$. Assume without loss of generality that $d_y \geq d_x$ and choose $z_x \in \p \W$ with $d(z_x, x)  = d_x$. Now use the clean ball property to find a point $a_x$ with $B_{r/K}(a_x) \ss B_{r}(z_x)\cap \W$ for $r = \min \{{d(x, y)}/{4}, r_0 \}$. Similarly construct $z_y$ and $a_y$. At $a_y$, we have from \eqref{eq:3} that 
	\[
		G_\W(y, a_y) \approx \frac{d_y d_{a_y}}{d^{n}(a_y, y)} \approx \frac{d_y}{r^{n-1}}.
	\]
	Applying the Harnack inequality along a Harnack chain connecting $a_x$ and $a_y$, this also gives $G_\W(y, a_x) \approx {d_y}/{r^{n-1}}$. We now use the boundary Harnack principle on $G_\W(y, \cdot)$ and $G_\W(y_0, \cdot)$ on the region $B_{2r}(z_x) \cap \W$ (which contains $x$) to give
	\[
		\frac{G_\W(y, x)}{G_\W(y_0, x)} \approx \frac{G_\W(y, a_x)}{G_\W(y_0, a_x)} \approx \frac{d_y}{r^{n-1}} \frac{1}{r}.
	\]
	In other words, 
	\[
		G_\W(y, x) \approx G_\W(y_0, x) \frac{d_y}{r^n} \approx \frac{d_x d_y}{r^n}.
	\]
	Noting that $r \approx d(x, y)$ leads to the conclusion \eqref{e:lgreen2}, which we have now established in all cases.
	
	For \eqref{e:lgreen}, \cite[Corollary 1.3.6]{Kenig} gives that at any $x \in \p \W$, $r < r_0$, and $y \in \W \sm B_{4r}(x)$, if $z$ is a point with $B_{r/K}(z) \ss B_r(x) \cap \W$,
	\[
	\w_y(B_{r}(x)) \approx G_\W(z, y) r^{n - 2}.
	\]
	Applying \eqref{e:lgreen2} leads to
	\[
		G_\W(z, y) \approx \frac{d_z d_y}{d^n(z, y)} \approx r \frac{d_y}{d^n(x, y)},
	\]
	from which \eqref{e:lgreen} follows.
\end{proof}

Note that we used the upper and lower bounds $\UP(\W), \DO(\W)$ in a crucial way to get \eqref{eq:1} and then \eqref{eq:2} and the remaining estimates. For general NTA domains, the conclusions of this lemma need not hold.

A direct consequence of these estimates and Lemma \ref{l:densitybd} is that $\w_y$ is absolutely continuous with respect to Hausdorff measure restricted to $\p \W$, and the Radon-Nykodym derivative
\begin{equation}\label{e:poissonkerneldef}
	K(x, y) = \frac{d \w_y}{d\cH^{n-1}\mres \p \W}(x) \approx \frac{d(y, \p \W)}{d^n(x, y)}
\end{equation}
the \emph{Poisson kernel}, is a bounded function of $x$ satisfying similar estimates.

\subsection{Blow-up analysis and the reduced boundary}\label{ss:reducedbdry}
Given a point $x \in \p \W$ and vector $\nu \in T_x M$ with $|\nu| = 1$, let 
\[
	B_{r, \nu}(x) = \exp_x \{ \mathscr{v} \in T_x M : |\mathscr{v}| < r, g(\mathscr{v}, \nu) < 0  \} \ss B_r(x)
\]
be a half-ball. For $r \leq \text{inj}_M$, define the function $l_{x, \nu} : B_r(x) \rightarrow \R$ be given by
\begin{equation}\label{eqn: truncated linear}
		l_{x, \nu}(\exp_x (\mathscr{v})) = g(\mathscr{v}, -\nu)_+,
\end{equation}
i.e. $l_{x, \nu}$ is a truncated linear function in normal coordinates.

\begin{lemma}\label{l:tangentplane} Let $\W$ be a minimizer of $\Ep$. For each $s \in (0, s_0)$, there are constants $\e(R, s, v, \vmax, \vpar) > 0$ and $c = c(R, v, \vmax, \vpar)$ such that if $x \in \p \W$ has  $|(B_r(x) \cap \W) \triangle B_{r, \nu}(x)| \leq \e r^n$ for some $\nu$ and $r < \e$, then:
	\begin{enumerate}
		\item $B_{c r}(x) \cap \p \W \ss \exp_x \{ \mathscr{v} \in T_x M : |\mathscr{v}| < r,  |g(\mathscr{v}, \nu)| < s r \}$,
		\item $|u_\W(y) - \a l_{x, \nu}(y) | < s r$ for some $\a \in [c, 1/c]$ and for all $y \in B_{c\sqrt{s} r}(x)$,
		\item $|w_\W(y) - \b l_{x, \nu}(y) | < s r$ for some $\b\in [c, 1/c]$ and for all $y \in B_{c\sqrt{s} r}(x)$.
	\end{enumerate}
\end{lemma}

\begin{proof}
	Conclusion (1) follows from standard geometric measure theory arguments using only Lemma \ref{l:densitybd}: if there is a $y \in \p \W\cap B_{c r}(x)$ with $y = \exp_x \mathscr{v}$, $|g(\mathscr{v}, \nu)| \geq s r$, then we may find clean balls $B_{c s r}(y_1) \in \W \cap B_{s r/2}(y)$ and $B_{c s r}(y_2) \in B_{s r/2}(y)\sm \W$, with either both inside or both outside $B_{r, \nu}(x)$. This implies that $|(B_r(x) \cap \W) \triangle B_{r, \nu}(x)| \geq c s^n r^n$ (one of these two balls must be in this symmetric difference), and this is a contradiction if $\e \ll s^n$.
	
	The other two conclusions follow from a compactness argument: assume, say, (2) is false for a given $s < s_0$, and take a sequence $\W_k$ of minimizers, points $x_k$  in $\p \W_k$, $r_k \rightarrow 0$, $\nu_k \in T_{x_k} M$ with norm $1$, and $B_{r_k}(x_k)$ with 
	\[
		B_{r_k}(x_k) \cap \p \W_k \ss \exp_x \{ \mathscr{v} \in T_{x_k} M : |v| < r_k,  |g(v, \nu_k)| < s_k r_k \}
	\]
	with $s_k \rightarrow 0$ (after applying (1) with $s_k$ rather than with $s$, using that $\e_k \rightarrow 0$ to do so). Identify $T_{x_k}M$ with $\R^n$ using an orthonormal basis with $e_n = \nu_k$, and define $\tilde{u}_k : B_1(0) \ss \R^n \rightarrow \R$ by
	\[
		\tilde{u}_k(e) = \frac{u_k(r_k \exp_{x_k} e )}{r_k}.
	\]
	Using Corollary \ref{cor:lip}, the functions $\tilde{u}_k$ are  uniformly Lipschitz functions defined on $B_1(0)$ and  $\{\tilde{u}_k > 0\}$ converges to $B_{1, e_n}(0) := \{y \in B_1: y \cdot e_n < 0 \}$ in the Hausdorff topology. After passing to a subsequence, $\tilde{u}_k \rightarrow \tilde{u}_\infty$  uniformly on $B_1(0)$. Moreover, letting $g_k$ denote the pullback of the metric $g$ on $B_{r_k}(x_k)$ under the rescaled exponential map $ e \in B_1(0) \mapsto \exp_{x_k}(r_k e)$, then  $g_k \rightarrow g_\infty$ in $C^2$ topology where $g_\infty(e_i, e_j) \equiv \d_{ij}$ the Euclidean metric.
	
	Let $\Lap_k$ denote the Laplacian with respect to the metric $g_k$ on $B_1(0)$. We have $\Lap_k \tilde{u}_k =  - r_k \ei(\W_k) \tilde{u}_k$ on $\{\tilde{u}_k > 0\}$; from elliptic estimates this means $\tilde{u}_k \rightarrow \tilde{u}_\infty$ in $C^2$ locally on $B_{1, e_n}(0)$. Rassing to the limit at any point in $B_{1, e_n}(0)$ gives $\Lap_{\infty} \tilde{u}_\infty = 0$. Therefore,  $\tilde{u}_\infty$ is harmonic on $B_{1, e_n}(0)$ and vanishes on the rest of $B_1(0)$ (using the uniform convergence again). In particular, using elliptic estimates means $\tilde{u}_\infty$ is a piecewise smooth function on $B_1(0)$ with $|\grad \tilde{u}_\infty | + |D^2 \tilde{u}_\infty|\leq C$ on $B_{1/2, e_n}(0)$.
	
	We also have that $c \DO(\W_k) r \leq  \sup_{B_r(0)} \tilde{u}_k \leq C \UP(\W_k) r$ for $r \in (0, 1)$ from Proposition~\ref{l:bdryharnackef}, and the constants here are uniform in $k$. Passing to the limit, $c r \leq  \sup_{B_r(0)} \tilde{u}_\infty \leq C r$, which implies that $|\n \tilde{u}_\infty(0)| \geq c$ (from the $-e_n$ direction) and so $\tilde{u}_\infty(x) = \a (x\cdot - e_n)_+ + O(|x|^2)$ for some $c \leq \a \leq C$. On $B_{\k \sqrt{s}}(0)$, this implies that
	\begin{align*}
		|\tilde{u}_k(x) - \a (x\cdot - e_n)_+ | \leq |\tilde{u}_\infty(x) - \a (x\cdot - e)_+ | &+ |\tilde{u}_k(x) - \tilde{u}_\infty(x) | \\
		&\leq C |x|^2 + o_k(1) \leq C s \k^2 + o_k(1)< \frac{1}{2}s.
	\end{align*}
	in the last step we took $\k$ small, and then $k$ large. Changing back to the original variables,
	\[
		\sup_{B_{\k \sqrt{s} r_k}(x_k)} |u_{\W_k} - \a l_{x_k, \nu_k}| < s r_k,
	\]
	which is a contradiction to (2) failing for $\W_k$ at $x_k$. The argument for (3) is similar.
\end{proof}

We say $x \in \p^* \W$, the reduced boundary, if
\[
	\lim_{r \searrow 0} \frac{|(B_r(x) \cap \W) \triangle B_{r, \nu}(x)|}{|B_r(x)|} = 0
\]
for some $\nu \in T_x M$ with $|\nu| = 1$; we will use the notation $\nu_x$ for this $\nu$. Lemma \ref{l:densitybd} implies that $\cH^{n-1}(\p \W \sm \p^*\W) = 0$ (see \cite[4.5.11, 4.5.6]{Federer}).

\begin{corollary}\label{c:reducedgmt} Let $\W$ be a minimizer of $\Ep$, and $x\in \p^* \W$. Then we have, for constants $0 < c \leq C < \infty$ depending only on $R, v, \vmax$, and $\vpar$:
	\begin{enumerate}
		\item $\p \W \cap B_r(x)$ converges to the approximate tangent plane $\p B_{r, \nu_x}(x) \cap B_r(x)$ in the following Hausdorff sense:
		\[
			\lim_{r \searrow 0} \frac{1}{r}\sup_{y \in \p \W \cap B_r(x)} d(y, \p B_{r, \nu_x}(x) \cap B_r(x)) = 0.
		\]
		\item $\cH^{n-1}(\p \W \cap B_r(x)) = \w_{n-1}r^{n-1} + o(r^{n-1})$, where $\w_{n-1}$ is the measure of the unit ball in $\R^{n-1}$.
		\item $u_\W(y) = |\grad u_\W(x)| l_{x, \nu_x}(y) + o(d(x, y))$ for some number $|\grad u_\W (x)| \in [c, C]$.
		\item $w_\W(y) = |\grad w_\W(x)| l_{x, \nu_x}(y) + o(d(x, y))$ for some number $|\grad w_\W (x)| \in [c, C]$.
	\end{enumerate}
\end{corollary}
The second conclusion can be found in \cite[4.5.6(2)]{Federer}; the others follow directly from Lemma \ref{l:tangentplane}. 

We say that a function $u \in C(\W)$ \emph{converges nontangentially} to $\a$ at $x \in \p \W$ if
\[
	\lim_{r \searrow 0}\  \sup\left\{ |u(y) - \a| \, : \, y \in B_r(x),\, d(y, \p \W) > \e r \right\} 
\]
for all $\e > 0$. While the numbers $|\grad u_\W(x)|$ in this corollary were abstract (not directly connected to $\grad u_\W$ proper), they may be reinterpreted as nontangential limits of $|\grad u_\W|$:

\begin{corollary}\label{c:ntlimits} Let $\W$ be a minimizer of $\Ep$, $x\in \p^* \W$, and $|\grad u_\W(x)|, |\grad w_\W(x)|$ as defined in Corollary \ref{c:reducedgmt}. Then:
	\begin{enumerate}
		\item $- |\grad u_\W(x)| \nu_x $ is the nontangential limit of $\grad u_\W$ at $x$.
		\item $- |\grad w_\W(x)| \nu_x$ is the nontangential limit of $\grad w_\W$ at $x$.
	\end{enumerate}
\end{corollary}

\begin{proof}
	We only prove (1), setting $\a = |\grad u(x)|$. Using Corollary \ref{c:reducedgmt},
	\[
		\lim_{r \searrow 0}\frac{\sup_{y \in B_r(x)} |u_\W(y) - \a l_{x, \nu_x}(y) |}{r} = 0.
	\] 
	Working in normal coordinates, at any $y \in B_{r, \nu_x}(x)$ the function $l_{x, \nu_x}(y)$ has
	\[
	|\Lap l_{x, \nu_x}(y)| \leq C\|g_x - g_{euc} \|_{C^1} \leq C r \qquad |\grad l_{x, \nu_x}(y) - \grad l_{x, \nu_x}(x)| \leq C r
	\]
	where $g_{euc}$ denotes the Euclidean metric. Fix $\e > 0$ and take $y \in B_r(x)$ with $d(x, y) \geq \e r$.   Then applying elliptic estimates on $B_{r \e}(y)$ in the first inequality and the fact that $|\Lap u_\W|\leq C$  in the second inequality, we have
	\begin{align*}
		\|\grad (u_\W  - \a  l_{x, \nu_x})\|_{L^\infty (B_{\e r /2}(x))} & \leq C \left[r \|\Lap (u_\W - \a l_{x, \nu_x})\|_{L^\infty(B_{\e r}(x))} + \frac{1}{r}\|u_\W - \a l_{x, \nu_x} \|_{L^\infty(B_{\e r}(x))} \right] \\
		& \leq C r + o_r(1) = o_r(1).
	\end{align*}
 The constant here depends only on $\e$. It follows that
	\[
		|\grad u_\W(y) + \a \nu_x |\leq o_r(1) + |\a||\nu_x + \grad l_{x, \nu_x}(y)| = o_r(1) + |\a||\grad l_{x, \nu_x}(x) - \grad l_{x, \nu_x}(y)| \leq o_r(1).
	\]
\end{proof}

\section{The Euler-Lagrange equation}\label{s:el}

Our next goal is to derive the Euler-Lagrange equation satisfied by minimizers by differentiating $\Ep$ with respect to smooth families of diffeomorphisms applied to $\W$. There will essentially be three main steps to do this. First, in Section~\ref{ss: dist EL}, we derive the distributional form of the Euler Lagrange equation. Computing the derivatives of the terms appearing in $\En$ is more or less routine for smooth sets in Euclidean space, though we must take some care because minimizers are not necessarily smooth and we are working on a Riemannian manifold. For the distributional form of the Euler-Lagrange equation, the derivatives of the nonlinear term are essentially given in condition \ref{a:nlc1} in the definition of admissible nonlinearity (Definition~\ref{def: nl}).

Next, in Section~\ref{ss: pointwise EL}, we derive a pointwise form of the distributional Euler-Lagrange  equation from Section~\ref{ss: dist EL}. The basic idea is to plug into the distributional Euler-Lagrange equation a sequence of vector fields that approximate the outer unit normal to $\W$ at each point in the reduced boundary. Carrying this out for the terms coming from the base energy is not difficult, but it turns out to be fairly delicate for the term $\nl$ for reasons discussed further below.  This will require the optimal Green's function estimates from Proposition~\ref{l:green} of the previous section. Ultimately, we are able to arrive at the following pointwise form of the free boundary condition (see Corollary~\ref{c:elpt} in Section~\ref{ss: pointwise EL}). Then there is a constant $A_0$ such that for almost every $x\in \p^* \W \cap \Qb_R$, we have the identity
\begin{equation}\label{eqn: EL intro8}
	-|\n u_\W(x)|^2  - \frac{\tpar}{2}|\n w_\W(x)|^2 + \err \left[ b_\W(x) + \int_{\W}a_\W v^x_\W \right] = - A_0.
\end{equation}
	If $|\W| \neq v$, then $A_0 = \fv'(|\W|)$. Here, for each $x \in \p \W \cap Q_R$, the function $v_\W^x$ is the solution of an auxiliary PDE.

It is not clear that the Euler-Lagrange equation \eqref{eqn: EL intro8} behaves as a Bernoulli-type problem or  should yield any  regularity.  The most important part of this section, which is the content of Section~\ref{ss: EL rewritten}, is to rewrite the Euler-Lagrange equation in a way which involves only $|\grad u_\W|$, up to ``lower order'' terms. In doing so, the term $b_\W(x)$ is more or less innocuous and will be controlled by simply choosing the parameter $\err$ to be sufficiently small. The $|\grad w_\W|$ term is a priori problematic, but can be handled similarly to \cite{CSY18, MTV17}. The most challenging term to control is $\int_{\W}a_\W v^x_\W $, because the function $v^x_\W$ solves an auxiliary PDE that does not immediately lend itself to comparison with $u_\W$. We will discuss how to overcome this difficulty at the beginning of Section~\ref{ss: EL rewritten} once we have seen introduced the equation for $v^x_\W$.   The main result of this section is the following free boundary condition.

\begin{theorem}\label{thm:usefulEL}
	 There is a function $\r \geq - C \err$ with $\|\r\|_{C^{0, \a}(\p \W)} \leq C$, $C, \a$ depending only on $v, \vmax, \vpar, R$, such that
	\[
		|\grad u_\W (x)|^2(1 + \r(x)) = A_0
	\]
	for $\cH^{n-1}$-a.e. $x \in \p^* \W \cap \Qb_R$. The constant $A_0$ has $A_0 \in [1/C, C]$.
\end{theorem}

As always, we assume that $R$ is fixed, $\vpar < \vpar_0(R, v, \vmax)$, $\tpar < \tpar_0(R, v, \vmax, \vpar)$, and $\err < \err_0(R, v, \vmax, \tpar, \vpar)$ are small enough that all results in earlier sections apply. We recall that $a_\W$ and $b_\W$ are functions given in the definition of admissible nonlinearity in Definition~\ref{def: nl}.

\subsection{The first variation along a vector field}\label{ss: dist EL}
The main goal of this subsection is to derive the distributional form of the Euler-Lagrange equation in Lemma~\ref{l:elvf}. We start with some auxiliary lemmas quantifying how eigenvalues and eigenfunctions vary under change of domains. We let $\nu= \nu_{\W}$ denote the outer unit normal of $\W$. If we knew that $\W$ had smooth boundary, then the following lemma would follow easily from the standard Hadamard variational formula;  see \cite{HP}. Since thus far we only know that $\W$ satisfies some very basic measure theoretic properties, the  proof requires some more care.

\begin{lemma}\label{l:tderevalue}
	Let $\W$ be a minimizer of $\Ep$. Let  $\phi_t : M \rightarrow M$  be a one-parameter family of smooth diffeomorphisms with $\phi_0(x) = x$ and  $\p_t \phi_t|_{t = 0} = T$ for  a vector field $ T$ on $M$. Set $\W_t = \phi_t(\W)$. Then:
	\begin{enumerate}
		\item $\limsup_{t\rightarrow 0} \frac{1}{|t|} \1 \ei(\W_t) - \ei(\W) + t \int_{\p^* \W} |\grad u_\W|^2 g(T, \nu)\, d\cH^{n-1} \2 \leq 0$
		\item $\limsup_{t\rightarrow 0} \frac{1}{|t|} \1 \tor(\W_t) - \tor(\W) + t \frac{1}{2}\int_{\p^* \W} |\grad w_\W|^2 g(T, \nu)\, d\cH^{n-1} \2 \leq 0$
		\item $\lim_{t\rightarrow 0} \frac{1}{t} \1|\W_t| - |\W| - t \int_{\p^* \W} g(T, \nu) \,d\cH^{n-1}\2 = 0$
		\item $\l_2(\W_t) \geq \ei(\W_t) + c(R, v, \vmax)$ for all $|t|$ small enough.
		\item $\lim_{t\rightarrow s} \frac{1}{|t - s|} |\ei(\W_t) - \ei(\W_s) + (t - s) f_s^\phi |  = 0$, where $f_s^\phi : (-c, c) \rightarrow (0, \infty)$ is a continuous function of $s$.
	\end{enumerate}
	In particular, $\lambda_1(\W_t)$ is differentiable and $\p_t \lambda_1(\W_t)|_{t=0} = -\int_{\p^* \W} |\grad u_\W|^2 g(T, \nu)\, d\cH^{n-1}$.
\end{lemma}

\begin{proof} {\it Step 1: Some general expressions.}
	Take any $v \in L^2(\W)$, and let $v_t = v \circ \phi_t^{-1} \in L^2(\W_t)$. Then we may compute
	\begin{equation}\label{e:tderevalue2}
	\int_{\W_t} v_t^2 = \int_\W v^2 |\det d \phi_t | = \int_\W v^2 \left(1 + t \dvg T + O(t^2)\right),
	\end{equation}
	where the $O$ depends only on $\|\phi_t\|_{C^1}$. If instead we consider a $v \in H^1(\W)$,
	\[
	\int_{\W_t} |\grad v_t|^2 = \int_\W \left|d \phi_t^{-1} (\grad v)\right|^2 |\det d \phi_t | = \int_\W \left[|\grad v|^2(1 + t \dvg T) - 2 t g(\grad v, \grad_{\grad v} T)\right] + O(t^2)\int_\W |\grad v|^2.
	\]
	If $v$ is smooth on $\W$, bounded, and $-\Lap v = f \in L^2(\W)$, this may be further rewritten using the identity
	\[
	\dvg\left(|\grad v|^2 T - 2g(\grad v, T)\grad v\right) = |\grad v|^2 \dvg T + 2g(\grad v, T) f - 2g(\grad v, \grad_{\grad v} T) \in L^2(\W)
	\]
	to give
	\[
	\int_{\W_t} |\grad v_t|^2 = \int_\W |\grad v|^2 + t\int_{\W} - 2 g(\grad v, T) f + \dvg\left(|\grad v|^2 T - 2g(\grad v, T)\grad v\right) + O\left(t^2 \right)\int_\W |\grad v|^2.
	\]
	Applying the divergence theorem of \cite{CTZ09} to the second term and assuming $\grad v$ has nontangential limits at $\cH^{n-1}$-a.e. point of $\p \W$ lets us rewrite
	\[
	\int_{\W_t} |\grad v_t|^2 = \int_\W |\grad v|^2 - t\int_{\W} 2 g(\grad v, T) f + t\int_{\p^* \W} |\grad v|^2 g(T, \nu) - 2 g(\grad v, T) g(\grad v, \nu) d\cH^{n-1} + O\left(t^2\right) \int_\W |\grad v|^2.
	\]
		
	{\it Step 2: Proofs of (1)-(4).} We now verify the conclusions of the lemma. For (1), take $v = u_\W$, and use $v_t$ as a competitor for the definition of $\ei(\W_t)$. Then $v_t \in H^1_0(\W_t)$, and
	\[
	\int_{\W_t} v_t^2 = \int_\W u_\W^2 \left(1 + t \dvg T + O(t^2)\right) = 1 - 2 t \int_{\W} u_\W g(\grad u_\W, T) +O(t^2),
	\]
	integrating by parts as $u_\W \in H^1_0(\W)$. On the other hand, $\grad u_\W$ has nontangential limits a.e. (from Corollary \ref{c:ntlimits}), and so
	\[
	\int_{\W_t} |\grad v_t|^2 = \ei(\W) - t\int_{\W} 2 g(\grad u_\W, T) \ei(\W) u_\W + t\int_{\p^* \W} |\grad u_\W|^2 g(T, \nu) - 2 g(\grad u_\W, T) g(\grad u_\W, \nu) d\cH^{n-1} + O(t^2)
	\]
	Using the formula from Corollary \ref{c:ntlimits} on the boundary gives $g(\grad u_\W, T) g(\grad u_\W, \nu) = |\grad u_\W|^2 g(T, \nu)$, so this leads to
	\[
	\ei(\W_t) \leq \frac{\int |\grad v_t|^2}{\int u_t^2} \leq \ei(\W) - t\int_{\p^* \W} |\grad u_\W|^2 g(T, \nu) d\cH^{n-1} + O(t^2),
	\]
	as the terms of the form $t\int_\W u_\W g(\grad u_\W, T)$  cancel to order $t$. This proves (1).
	
	The proof of (2) is analogous:  proceeding in the same way using $v = w_\W$,  the terms $t\int_{\W} g(\grad u_\W, T)$ in  the expansion of $\tor(\W_t)$ again cancel and we find
	\[
	\tor(\W_t) \leq \int \frac{1}{2}|\grad v_t|^2 - v_t \leq \tor(\W) - \frac{t}{2}\int_{\p^* \W} |\grad w_\W|^2 g(T, \nu) d\cH^{n-1} + O(t^2).
	\]
	
	To prove (3), we use  the constant function $v = 1$  in \eqref{e:tderevalue2} and the divergence theorem to obtain
	\[
	|\W_t| = |\W| + t \int_{\p^* \W} g(T, \nu) d\cH^{n-1} + O(t^2).
	\]
	
	Using (1), (2), and  (3), we see that $\En(\W_t) \leq \En(\W) + C|t| \leq \minE + \d$ if $|t|$ is small enough. From Lemma \ref{l:simpleeval}, this implies that $\l_2(\W_t) > \ei(\W) + c$, giving (4).
	\\
	
	{\it Step 3: Proof of (5).}
	Finally, for (5) we proceed as for (1), except with $\W_s$ in place of $\W$, using $u_{s, t} = u_{\W_s}\circ \phi_s \circ \phi_t^{-1} : \W_t \rightarrow \R$ as a competitor. These are nonnegative and have
	\begin{equation}\label{e:tderevalue3}
		\int_{\W_t} u_{s, t}^2 = 1 + O(|t - s|)
	\end{equation}
	from \eqref{e:tderevalue2}. We do not try to justify the application of the divergence theorem in this case, instead only using
	\begin{align*}
	\ei(\W_t) &\leq  \frac{\int_{\W_t} |\grad u_{s, t}|^2} {\int_{\W_t} u_{s, t}^2} \\
	& = \ei(\W_s) + (t - s)  \int_\W |\grad u_{\W_s}|^2 \dvg T_s - 2 g(\grad u_{\W_s}, \grad_{\grad u_{\W_s}} T_s) + 2 u_{\W_s} g(\grad u_{\W_s}, T_s) + O((t - s)^2) \\
	& := \ei(\W_s) + (t - s)f_s^\phi + O((t-s)^2),
	\end{align*}
	where $T_s = \p_t \phi_t|_{t = s}$. Clearly $f_s^\phi$ is bounded in terms of $\phi$ and $\ei(\W_t)$, and $\ei(\W_t)$ is bounded from (1). Applying with $t$ and $s$ reversed shows that
	\begin{equation} \label{e:tderevalue1}
	\ei(\W_s) \leq \ei(\W_t) + (s - t) f_t^{\phi} + O((t - s)^2),
	\end{equation}
	and so $|\ei(\W_s) - \ei(\W_t)|\leq C |s - t|$. Together with (4), Lemma \ref{l:poincarestab}, and \eqref{e:tderevalue3} this implies that
	\begin{align*}
	\|u_{s, t} - u_{\W_t}\|_{H^1_0(\W_t)}^2 &\leq \bigg\|\frac{u_{s, t}}{\sqrt{\int u_{s, t}^2}} - u_{\W_t}\bigg\|_{H^1_0(\W_t)}^2 + C |t - s| \\
	& \leq C\left[\frac{\int_{\W_t} |\grad u_{s, t}|^2} {\int_{\W_t} u_{s, t}^2} - \ei(\W_t) \right] + C |t - s| \leq C |t - s|.
	\end{align*}
	On the other hand, as $u_{s, t}$ are defined by $u_{\W_s}$ composed with smooth diffeomorphisms approaching the identity, $\lim_{t \rightarrow 0}\|u_{s, t} - u_{\W_s}\|_{H^1(M)} =  0$ from the Vitali convergence theorem. Together, these give that $u_{\W_t} \rightarrow u_{\W_s}$ strongly in $H^1$ as $t \rightarrow s$, and it follows that $f_t^\phi$ is a continuous function of $t$. Combining with \eqref{e:tderevalue1} leads to
	\[
	\ei(\W_s) \leq \ei(\W_t) + (s - t) f_s^\phi + o(t-s).
	\]
	This completes the proof of (5). The final statement of the lemma is immediate from (1) and (5).
\end{proof}

We now turn to a finer analysis of how the eigenfunctions $u_{\W_t}$ vary under domain variation. We already saw in the proof of (5) that they vary continuously in $t$, in $H^1$ norm; however, we will need a much more careful estimate to handle the term $\nl$. Our first goal is a kind of $C^1_t L^2_x$ estimate, which we break into two parts.

\begin{lemma} \label{lem:differentiation}
	Let $\W$ be a minimizer of $\Ep$. Let $g_t$ be a smooth (in both $x$ and $t$, uniformly on $\W$), one-parameter family of metrics with $g_0  = g$, the original metric on $M$, and let $\Lap_{t}$ and $m_t$ respectively be the Laplace-Beltrami operator and volume measure for the metric $g_t$. Let $f_t$ be a $C^1$ function of $t$ that is constant in $x$  with $f_0 = \ei(\W)$. Assume that $u_t \in H^1_0(\W)$ is the unique nonnegative function with $\int u_t^2 dm_t = 1$ and
	\[
		- \Lap_{t} u_t = f_t u_t,
	\]
	and assume that no solutions $v\in H^1_0(\W)$ to
	\[
		-\Lap_{t} v = \l f_t v
	\]
	exist for any $\l \in (0, 1 + c)$ (for some $c > 0$) apart from scalar multiples of $u_t$. Then $u_t$ is continuously differentiable in $t$ for $|t|< c$ and any $x \in \W$, the derivative $v$ at $t = 0$ lies in $C(\bar{\W}) \cap H^1_0(\W)$, and
	\[
		\lim_{t \rightarrow 0} \frac{1}{t} \|u_t - u_0 - t v\|_{H^1_0(\W)} = 0.
	\]
\end{lemma}

Note that the existence of $u_t$ is part of the lemma's hypotheses; we do not claim such a $u_t$ exists.

\begin{proof}
	Observe that  $\|u_t\|_{H^1_0(\W)}$ is uniformly bounded in $t$, using $u_t$ as a test function for itself:
	\begin{equation} \label{e:diff1}
		\int |\grad u_t|_t^2 dm_t = f_t \int  u_t^2 dm_t \leq f_t.
	\end{equation}

	For a $q \in H^{-1}(\W)$ consider the problem of finding a $p \in H^1_0(\W)$ which solves
	$
		- \Lap_t p = f_t p + q,
	$
	with $\int p u_t = 0$. Such a $p$ exists if and only if  $q[u_t] = 0$; if it exists it is unique and
	\begin{equation}\label{e:fredholm}
		\|p\|_{H^1(\W)} \leq C \|q\|_{H^{-1}(\W)}.
	\end{equation} 
Indeed, the operator $-\Lap_t - f_t I : H^1_0(\W) \rightarrow H^{-1}(\W)$ is bounded and, by our assumptions, has a one-dimensional kernel spanned by $u_t$. From the Fredholm alternative and open mapping theorem (see \cite{Brezis}) that the range consists of all $q \in H^{-1}(\W)$  with $q[u_t]= 0$, and
	\[
		\int \Big|\grad \Big[p - u_t\int g_t(\grad u_t, \grad p)\Big] \Big|_{g_t}^2 dm_t \leq C \|q\|_{H^{-1}}^2.
	\]
	Then \eqref{e:fredholm} follows, as $\int  g_t(\grad u_t, \grad p) = \int f_t u_t p = 0$  from the definition of $u_t$.

	Insert $u_t$ into the equation for $u_s$, to get (in weak form with $\phi \in H^1_0(\W)$ a test function)
	\begin{align*}
		\int g_s(\grad u_t, \grad \phi)dm_s & = \int g_s(\grad u_t, \grad \phi)dm_t  + \int g_s(\grad u_t, \grad \phi)d (m_s - m_t) \\
		& = f_t \int u_t \phi dm_t + \int g_s(\grad u_t, \grad \phi) - g_t(\grad u_t, \grad \phi) dm_t + \int g_s(\grad u_t, \grad \phi)d (m_s - m_t)\\
		& = f_s \int u_t \phi dm_s + \int (f_t - f_s) u_t \phi + g_s(\grad u_t, \grad \phi) - g_t(\grad u_t, \grad \phi) dm_t \\
		&\qquad + \int - f_s u_t \phi + g_s(\grad u_t, \grad \phi)d (m_s - m_t)\\
		& = f_s \int u_t \phi dm_s + q_s^t[\phi],
	\end{align*}
	where $q_s^t \in H^{-1}(\W)$ has
	\[
		|q_s^t[\phi]| \leq \Big[\|f_t - f_s\|_{L^\infty} + C \sup_\W |g_t - g_s|  \Big] \|u_t\|_{L^2}\|\phi\|_{L^2} + \sup_\W |g_t - g_s| \|u_t\|_{H^1_0(\W)} \|\phi\|_{H^1_0(\W)},
	\]
	and in particular $\|q\|_{H^{-1}(\W)} \leq C |t - s|$. Here $|g_t - g_s|= |g_t-g_s|_{g}$ is with respect to the original metric. Subtracting $u_s$, we have (in weak form)
	\begin{equation}\label{e:tdiffpde}
		- \Lap_s (u_t - u_s) = f_s (u_t - u_s) + q_s^t.
	\end{equation}

	By applying \eqref{e:fredholm} to $u_t - u_s - u_s \int u_s(u_t - u_s) dm_s = u_t - a u_s$, we see that
	\[
		\|u_t - a u_s\|_{H^1_0} \leq C \|q_s^t\|_{H^{-1}} \leq C |t - s|.
	\]
	The constant $a$ need not be $1$, but we do know that as $u_t, u_s \geq 0$ by assumption, $a = \int u_t u_s dm_s \geq 0$. Using the Poincar\'e and triangle inequalities and the normalization on $u_s, u_t$,
	\begin{align*}
		|1 - a|  = \big|\|u_t\|_{L^2(dm_t)} -  \|a u_s\|_{L^2(dm_s)}\big| 
		&\leq \big|\|u_t\|_{L^2( dm_s)} - \|a u_s\|_{L^2(dm_s)}\big| + \big|\|u_t\|_{L^2(dm_s)} - \|u_t\|_{L^2(dm_t)} \big|\\
		&\leq \|u_t - a u_s\|_{L^2( dm_s)} + C |t - s| \leq C |t - s|.
	\end{align*}
	Therefore, we see that
	\begin{equation}\label{e:diff3}
		\|u_t - u_s\|_{H^1_0(\W)} \leq \|u_t - a u_s\|_{H^1_0(\W)} + |1 - a|\|u_s\|_{H^1_0(\W)} \leq C |t - s|.
	\end{equation}

	Set
	\[
q'_t[\phi] :=  - \int (\p_t f_t) u_t \phi - (\p_t g_t)(\grad u_t, \grad \phi) + m'_t\big[f_t u_t \phi - g_t(\grad u_t, \grad \phi)\big] dm_t,
	\]
	where $m'_t$ stands for the derivative of the volume form (i.e. $m_t(E) - m_s(E) = \int_s^t \int_E m'_r dm_r dr $).	Then $|q'_t[\phi]|\leq C\|\phi\|_{H^1_0} \|u_t\|_{H^1_0}$, so $\|q'_t\|_{H^{-1}} \leq C$. On the other hand, directly estimating each term using the assumptions on $f_t$ and $g_t$ leads to
	\begin{align*}
	 |q_s^t[w] - (s - t) q'_t[w]| &\leq C \Big[\sup_\W |f_t - f_s + (s - t) \p_t f_t| + \sup_\W |g_t - g_s + (s - t)g'_t| + \sup_\W |f_t - f_s|^2 + |g_t - g_s|^2\Big]\\
		&\qquad \cdot \|u_t\|_{H^1_0(\W)} \|\phi\|_{H^1_0(\W)}\\
		& \leq o(|t - s|) \|u_t\|_{H^1_0(\W)} \|\phi\|_{H^1_0(\W)},
	\end{align*}
	so $\|q_s^t - (s - t) q'_t\|_{H^{-1}(\W)} = o(|t - s|)$. Using \eqref{e:diff3}, $\|q'_t - q'_s\|_{H^{-1}(\W)} \rightarrow 0$ as $ t\rightarrow s$, so this may be rewritten as
	\[
		\|q_s^t - (s - t) q'_s\|_{H^{-1}(\W)} = o(|t - s|).
	\]
	
	We also claim that $q'_s[u_s] = 0$. Indeed, as $- \Lap_s u_t = f_s u_t + q_s^t$ is a nontrivial solution, from the Fredholm alternative we must have $q_s^t[u_s] = 0$. But then
	\[
		|t-s| |q'_t[u_s]| = o(|t - s|),
	\]
	so taking $t \rightarrow s$ implies $q'_s[u_s] = 0$.  Let $v_s^0 \in H^1_0(\W)$ be the unique solution to $-\Lap v_s^0 = f_s v_s^0 + q'_s$ with $\int u_s v_s^0 = 0$, using the Fredholm alternative from \eqref{e:fredholm}. Then together with \eqref{e:tdiffpde}, we have that
	\[
		\|u_t - a u_s - (s - t) v_s^0\|_{H^1_0} \leq C \|q_s^t - (s - t) q'_s\|_{H^{-1}(\W)}  = o(|t -s|),
	\]
	where $a = \int u_t u_s dm_s$ as before. Now set $v_s^1 = u_s \frac{1}{2}\int u_s^2 m'_s dm_s$, and estimate $a$ again:
	\begin{align*}
		\Big|\int u_t^2 dm_t - \int u_s^2 dm_s - (t - s) \int u_s^2 m'_s dm_s\Big| & \leq \Big|\int u_t^2 dm_t - \int u_s^2 dm_s  - (t - s) \int u_t^2 m'_s dm_s\Big| + C |t - s|^2\\
		& \leq C |t - s|^2,
	\end{align*}
	so taking square roots and recalling that $\int u_t^2 dm_t = 1$,
	\begin{align*}
	\Big|\|u_t\|_{L^2(dm_t)}& - \|u_t\|_{L^2(dm_s)} - (t - s)\frac{1}{2} \int u_s^2 m'_s dm_s\Big| \\
	& \leq  C |t - s|^2 + \frac{|2 - 1 - \|u_t\|_{L^2(dm_s)}|}{1 + \|u_t\|_{L^2(dm_s)}} |t - s|\int u_s^2 |m'_s| dm_s   \leq C |t - s|^2.
	\end{align*}
	This may be used to estimate $1 - a$ up to a correction:
	\begin{align*}
		\Big|1 - a - (t - s)\frac{1}{2} \int u_s^2 m'_s dm_s\Big| &\leq \big| \|u_t\|_{L^2(dm_s)} - a \|u_s\|_{L^2(dm_s)}\big| + C |t - s|^2 \\
		& \leq \big| \|u_t\|_{L^2(dm_s)} - \|a u_s - (s - t) v_s^0 \|_{L^2(dm_s)}\big| + C |t - s|^2 \\
		& \leq \|u_t - a u_s - (s - t) v_s^0\|_{H^1_0} + C |t -s|^2 = o(|t -s|).
	\end{align*}
	The second step used that $\int u_s v_s^0 dm_s = 0$. Thus,
	\begin{align*}
		\big\|u_t - u_s - (s - t) [v_s^0 + v_s^1]\big\|_{H^1_0} & = \|u_t - a u_s - (s - t) v_s^0\|_{H^1_0} + \big|1 - a - (t - s)\frac{1}{2} \int u_s^2 m'_s dm_s\big|\\
		& = o(|t-s|).
	\end{align*}
	Set $v_s = v_s^0 + v_s^1$. This establishes the main conclusion of the lemma. It remains to verify that $v_0 \in C(\bar{\W})$ and that $u_t$ is continuously differentiable on $\W$.

	These can be seen from the PDE for $v_s^0$. As $u_t$ is smooth on the interior of $\W$, as long as $\phi \in H^1_0(U)$ for $U\cc \W$, we have (integrating by parts as needed)  that $|q'_s[\phi]| \leq C(U) \|\phi \|_{L^1} $. In other words, $q'_s$ (and similarly $q_s^t$) may be represented by a bounded function. Applying elliptic regularity estimates to \eqref{e:tdiffpde} gives that for $U' \cc U$,
	\[
		\|u_t - u_s - (s - t)v_s\|_{C^{1, \a}(U')} \leq C(U, U') \big[\|u_t - u_s - (s - t)v_s\|_{H^1(U)} + \|q_s^t - (s - t) q'_s\|_{L^\infty(U)} \big] = o(t - s).
	\]
	In particular, this implies that $u_t(x), \grad u_t$ are continuously differentiable in $t$, locally uniformly on $\W$.
	
	As for checking that $v_0 \in C(\bar{\W})$, note that $q'_0$ only depends on $u_0, \grad u_0$, which are bounded uniformly on $\W$ (as $u_0 = u_\W$, and using Corollary \ref{cor:lip}). Thus from elliptic estimates on NTA domains (see e.g. \cite[Theorem 1.2.8]{Kenig}), we have that for some $\a > 0$,
	\[
		\|v_0^0\|_{C^{0, \a}(\bar{\W})} \leq C \left[ \|v_0^0\|_{H^1_0(\W)} + \|q'_0\|_{L^\infty(\W)}\right] \leq C.
	\]
	On the other hand, $v_0^1$ is a bounded multiple of $u_0$, and therefore Lipschitz continuous.
\end{proof}

\begin{lemma}\label{l:tderpde}
	Let $\W$ be a minimizer of $\Ep$, $\phi_t : M \rightarrow M$ a one-parameter family of smooth diffeomorphisms with $\phi_0(x) = x$, $\p_t \phi_t|_{t = 0} = T$ a vector field on $M$, and $\W_t = \phi_t(\W)$. Then $u_{\W_t}: (-c, c)\rightarrow H^1(\W)$ is differentiable in $t$ at $t = 0$ with derivative $\dot{u}_\W = \dot{u}^T_\W \in L^\infty(M)$ vanishing outside $\W$, in the following sense:
	\[
		\lim_{t \rightarrow 0} \frac{1}{t} \left\|u_{\W_t} - u_\W - t\, \dot{u}_\W\right\|_{L^2(M)} = 0.
	\]
	The function $\dot{u}_\W$ satisfies
	\begin{equation}\label{eqn: udot eqn}
		\begin{cases}
			- \Lap \dot{u}_\W = \ei(\W) \dot{u}_\W - u_\W \int_{\p^* \W } |\grad u_\W|^2 g(T, \nu) d\cH^{n-1}  & \text{ on } \W \\
			\dot{u}_\W = - g(\grad u_\W, T) & \text{ on } \p^* \W\\
			\int_\W \dot{u}_\W u_\W = 0  .& \\
		\end{cases}
	\end{equation}
	The boundary condition holds in the sense of nontangential limits at a.e. point on $\p^* \W$.
\end{lemma}

We will write $\dot{u}_\W^T$ to denote the derivative found in Lemma~\ref{l:tderpde} when we wish to emphasize its dependence on the vector field $T$, and will simply write $\dot{u}_\W$ otherwise in order to alleviate notation. We will discuss the specifics of the PDE satisfied by $\dot{u}_\W$ below, but note for now that from the absolute continuity of harmonic and Hausdorff measure from Proposition~\ref{l:green}, the boundary condition is given on a sufficiently large portion of $\p \W$ (i.e. there is a unique solution to this problem; see \cite[Theorem 1.4.4]{Kenig}).

\begin{proof}
	Set $u_t = u_{\W_t}\circ \phi_t : \W \rightarrow [0, \infty)$ to be the pullback of $u_{\W_t}$. Then $u_t$ satisfies a PDE on $\W$; namely,
	$
		- \Lap_t u_t = f_t u_t,
	$
	where $\Lap_t$ is the Laplace-Beltrami operator associated with the metric $g_t = \phi_t^* g$ and $f_t = \ei(\W_t)$. From Lemma \ref{l:tderevalue}, part (5), $f_t$ is a $C^1$ function. Moreover, $u_t$ is the unique nonnegative solution to this equation in $H^1_0(\W)$ with $\|u_t\|_{L^2(\W, dm_t)} = 1$, while the same equation with $f_t$ replaced by $\l f_t$, $\l \in [0, 1 + c]$ admits no nontrivial solutions (this follows from Lemma \ref{l:tderevalue}, part (4)). Apply Lemma \ref{lem:differentiation} to this to learn that $u_t$ is continuously differentiable in $t$ and to obtain a function $v \in H^1_0(\W) \cap C(\bar{\W})$ with
	\[
		\lim_{t \rightarrow 0} \frac{1}{t}\|u_t - u_\W - t v \|_{H^1_0(\W)} = 0.
	\]
	In particular, we have that
	\begin{equation}\label{e:tdirpde1}
		\|u_{\W_t} - u_\W \circ \phi_t^{-1}\|_{H^1_0(\W_t)} \leq C \|u_t - u_\W\|_{H^1_0(\W)} \leq C|t|.
	\end{equation}

	Our goal now is to estimate $u_{\W_t} - u_t$. Consider, with $t$ \emph{fixed}, the function $\psi \circ \phi_s(x) : M \times [0, t] \rightarrow [0, \infty)$ for any $\psi \in H^1(M)$. This function lies in $H^1(M \times [0, t])$, with distributional derivative $d_{(x, s)} (\psi \circ \phi_s)(v, 1) = (d_{\phi_s(x)}\psi( d_x\phi_s(x) v), d_{\phi_s(x)}\psi(\p_s \phi_s(x))$; this may be verified by approximating by smooth functions. Plugging in $u_{\W_t}$ for $\psi$ leads to the identity
	\[
		u_{\W_t}(x) - u_{t}(x) = - \int_0^t \p_s [u_{\W_t} \circ \phi_s](x) ds = - \int_0^t g(\grad u_{\W_t}(\phi_s(x)), \p_s \phi_s(x)) ds,
	\]
	which is valid for almost every $x \in M$. Now, for $|s|\leq |t|$,
	\begin{align*}
		\|\grad (u_{\W_t}\circ \phi_s) - \grad u_\W\|_{L^2(M)} & \leq \|\grad (u_{\W_t}\circ \phi_s) - \grad (u_\W \circ \phi_t^{-1} \circ \phi_s)\|_{L^2(M)} + \|\grad u_{\W} - \grad (u_\W \circ \phi_t^{-1} \circ \phi_s)\|_{L^2(M)} \\
		& \leq C|t| + o_t(1).
	\end{align*}
	We used \eqref{e:tdirpde1} to bound the first term (after changing variables), while the second goes to $0$ from the dominated convergence theorem ($\grad u_\W$ is bounded while $\phi_t^{-1}\circ \phi_s(x) \rightarrow x$ pointwise). We also have that
	\[
		|\p_s \phi_s(x) - T|\leq C |s| \leq C |t|
	\]
	as $\phi_s$ is smooth. Applying both of these,
	\begin{align*}
		\frac{1}{|t|}\| u_{\W_t} - u_t + t g(\grad u_\W, T)\|_{L^2(M)} &\leq \frac{1}{|t|} \int_0^t \left\|g(\grad u_{\W_t}(\phi_s(\cdot)), \p_s \phi_s(\cdot)) - g(\grad_\W , T) \right\|_{L^2(M)}ds\\
		& \leq C|t| + o_t(1)  = o_t(1).
	\end{align*}

	Now, set $\dot{u}_\W = v - g(\grad u_\W, T)$ (extending it by $0$ outside of $\W$). We have shown that
	\[
		\frac{1}{|t|}\|u_{\W_t} - u_\W - t \dot{u}_\W \|_{L^2(M)} \leq \frac{1}{|t|}\big[\|u_t - u_\W - t v \|_{L^2(\W)} + \|u_{\W_t} - u_t + t g(\grad u_\W, T) \|_{L^2(M)}\big] \rightarrow 0.
	\]
	As both $v$ and $\grad u_\W$ are bounded, so is $\dot{u}_\W$. It remains to verify that $\dot{u}_\W$ solves the stated PDE.

	First, recall that from Lemma \ref{lem:differentiation} we know $u_t$ (and its spatial derivatives) are continuously differentiable in $t$ on $\W$ for small $|t|$. It follows that $u_{\W_t}$ is continuously differentiable in $t$ on $\W_t$ as well, and $\p_t u_{\W_t}|_{t = 0} = \dot{u}_\W$ on $\W$. Therefore, we may differentiate the PDE
	$
		-\Lap u_{\W_t} = \ei(\W_t) u_{\W_t}
	$
	pointwise on $\W$ at $t = 0$ to give
	\[
		-\Lap \dot{u}_\W = \ei(\W) \dot{u}_\W + (\p_t \ei(\W_t) |_{t = 0}) u_\W = \ei(\W) \dot{u}_\W - u_\W \int_{\p^* \W} |\grad u_\W|^2 g(T, \nu) d\cH^{n-1}
	\]
	on $\W$, where in the second equality we used the final statement of Lemma \ref{l:tderevalue}.
	The orthogonality condition may be inferred from the normalization of $u_{\W_t}$: indeed,
	\[
		1 = \int u_{\W_t}^2 = \int (u_\W + t \dot{u}_\W)^2 + o(t) = 1 + 2 t \int u_\W\dot{u}_\W + o(t),
	\]
	Rearranging and sending $t$ to $0$ gives $\int u_\W \dot{u}_\W = 0$. 
		Finally, we have that $- g(\grad u_\W, T)$ has nontangential limits a.e. on $\p \W$ from Corollary \ref{c:ntlimits}, while $v \rightarrow 0$ on $\p \W$ uniformly; this gives the boundary condition.
\end{proof}

We are now in a position to prove the main result of this subsection, which is a ``distributional'' Euler-Lagrange equation for $\W$.

\begin{lemma}[Distributional Euler-Lagrange equation]\label{l:elvf}
	Let $\W$ be a minimizer of $\Ep$. Fix  $B_r(x_0) \ss \Qb_R$ and let $L:=|B_r(x_0)|$. Let $T$ be a smooth vector field on $M$ with $|T|\leq 1$ that is compactly supported on $B_r(x_0)$.
	
	If $T$ is volume preserving to first order in the sense that $\int_{\p^* \W} g(T, \nu_x) d\cH^{n-1} = 0$,  then
	\[
		\bigg| - \int_{\p^* \W} |\n u_\W|^2 g(T, \nu) + \frac{\tpar}{2}|\n w_\W|^2 g(T, \nu_x) d\cH^{n-1} + \err \Big[\int_\W \dot{u}_\W a_\W  + \int_{\p^*\W} b_\W g(T, \nu) d\cH^{n-1} \Big]  \bigg| \leq C L,
	\]
	where $a_\W$, $b_\W$ are the functions from property \ref{a:nlc1} of $\nl$ and $\dot{u}_\W = \dot{u}_\W^T$ is as in Lemma~\ref{l:tderpde}.
	
	If $|\W| \neq v$, $L < ||\W| - v|$, and we do not assume $\int_{\p^* \W} g(T, \nu_x) d\cH^{n-1} = 0$, we instead have
	\begin{align*}
			\bigg| - \int_{\p^* \W}& |\n u_\W|^2 g(T, \nu) + \frac{\tpar}{2}|\n w_\W|^2 g(T, \nu) d\cH^{n-1} \\
			&+ \fv'(|\W|)\int_{\p^* \W} g(T, \nu) d\cH^{n-1} + \err \Big[\int_\W \dot{u}_\W a_\W  + \int_{\p^*\W} b_\W g(T, \nu) d\cH^{n-1} \Big]  \bigg| \leq C L.
	\end{align*}
\end{lemma}

The term containing $\dot{u}_\W$ may first appear to be lower-order since it is  integrated on $\W$ rather than the boundary $\p^* \W$. This is not the case: due to how $\dot{u}_\W=\dot{u}_\W^T$ depends on $T$, it will end up being of the same order as the others.

\begin{proof}[Proof of Lemma~\ref{l:elvf}]
	Let $\phi_t(x)$ be the flow associated with $T$: $\p_t \phi_t(x) = T(\phi_t(x))$; then $\phi_t$ is a smooth family of diffeomorphisms for all $|t|\leq c$ small, with $\phi_t(x) = x$ outside of $B_r(x_0)$. Setting $\W_t = \phi_t(\W)$, apply Lemma \ref{l:tderevalue} to $\W$ and $\phi_t$: this implies that $\l_2(\W_t) > \ei(\W_t) + c$ and
	\[
		\En(\W_t) \leq \En(\W) - t \int_{\p^* \W} |\n u_\W|^2 g(T, \nu) + \frac{\tpar}{2}|\n w_\W|^2 g(T, \nu) d\cH^{n-1} + o(t).
	\]

	Let us now consider $\nl(\W_t)$: using property \ref{a:nlc1}, we have that
	\[
		\nl(\W_t) - \nl(\W) = \int (u_{\W_t} - u_\W) a_\W + \int_{\W_t} b_\W - \int_\W b_\W + o(t) + O(t r^n).
	\]
	The second term here may be estimated as in Lemma \ref{l:tderevalue} using \eqref{e:tderevalue2} to give
	\begin{align*}
		\int_{\W_t} b_\W - \int_\W b_\W = t \int_{\W} b_\W \dvg T + o(t)
		& = t \int_{\p^* \W} b_\W g(T, \nu) d\cH^{n-1} - t \int_\W g(\grad b_\W, T) + o(t)\\
		& = t \int_{\p^* \W} b_\W g(T, \nu) d\cH^{n-1} + O(t r^n) + o(t),
	\end{align*}
	using that $\|b_\W\|_{C^{0, 1}} \leq C$ by assumption. For the first term, we see that $	\|u_{\W_t} - u_\W - t\dot{u}_\W\|_{L^2} = o(t)$ from  Lemma \ref{l:tderpde}, and so
	\[
		\int (u_{\W_t} - u_\W) a_\W = t \int\dot{u}_\W\, a_\W + o(t).
	\]
	Putting everything together and using that $\Ep(\W) \leq \Ep(\W_t)$,
	\[
		0 \leq - t \int_{\p^* \W} |\n u_\W|^2 g(T, \nu) + \frac{\tpar}{2}|\n w_\W|^2 g(T, \nu) d\cH^{n-1} + \err\Big[t \int_{\p^* \W} b_\W g(T, \nu) d\cH^{n-1} + t \int \dot{u}_\W\, a_\W\Big] + o(t) + O(tr^n).
	\]
	The conclusion follows from dividing by $t$ and sending $t\rightarrow 0$.
	
	If $|\W| \neq v$ and $||\W| - |\W_t|| < ||\W| - v|$ (which always holds if $t$ is small enough) then $\fv$ is linear and we may estimate
	\[
		\fv(|\W_t|) = \fv(|\W|) + \fv'(|\W|)[|\W_t| - |\W|] = \fv(|\W|) + t \fv'(|\W|) \int_{\p^* \W} g(T, \nu) d\cH^{n-1} + o(t).
	\]
	We then proceed as previously.
\end{proof}

\subsection{The first variation at a point}\label{ss: pointwise EL}

Next, we wish to derive a pointwise form of the distributional Euler-Lagrange equation of Lemma~\ref{l:elvf} by sending $T$ to $\d_x \nu_x - \d_y \nu_y$ at points $x, y \in \p^* \W$. Here and in the sequel, we let $\nu_x = \nu_\W(x)$ be the outer unit normal of $\W$ at $x$. Passing to this limit is actually quite delicate, with the main challenge arising from  making sure that the functions $\dot{u}_\W=\dot{u}_\W^T$ behave well under such an approximation procedure. Heuristically speaking, recalling the equation \eqref{eqn: udot eqn} satisfied by $\dot{u}_\W^T$, the limit of the $\dot{u}_\W^T$ should solve an equation with a Dirac delta as boundary data. On the other hand, the boundary data in \eqref{eqn: udot eqn} is achieved only $\cH^{n-1}$-a.e. and in a nontangential limit sense. To rigorously derive this limit, we will use the PDE \eqref{eqn: udot eqn} solved by $\dot{u}_\W$ by writing $\dot{u}_\W$ in terms of the Green's function and Poisson kernel. 

It will be convenient to decouple the Poisson kernel part and the Green's function part, or in other words, the harmonic part and its remainder part, of $\dot{u}_\W$ in the following way. For any smooth vector field as in Lemma~\ref{l:elvf}, let the harmonic part $h^T$  of $\dot{u}_\W^T$ be given by
\[
	h^T(x) = \int_{\p \W} \dot{u}_\W^T\, d\w_x,
\]
where $\w_x$ is harmonic measure; i.e. it solves the PDE
\[
	\begin{cases}
		-\Lap h^T = 0 & \text{ on } \W\\
		h^T =\dot{u}_\W^T = - g(\grad u_\W, T) & \text{ on } \p \W
	\end{cases}
\]
with the boundary condition in the sense of nontangential limits at a.e. point. Then, let $q^T = \dot{u}^T_\W - h^T$ be the remainder part. Using the equation for $\dot{u}_\W^T$ from Lemma~\ref{l:tderpde}, we see that $q^{T}$ solves the PDE
	\begin{equation}\label{eqn: PDE for q}
		\begin{cases}
		- \Lap q^T = \ei(\W) (q^T + h^T) - u_\W\int_{\p^* \W} |\grad u_\W|^2 g(T, \nu) d\cH^{n-1}   & \text{ on } \W \\
		q^T = 0 & \text{ on } \p \W\\
		\int (q^T + h^T) u_\W = 0. &
		\end{cases}
	\end{equation}

We construct our approximating vector fields in the following way.
Given a point $x \in \p^* \W$ and scale $r$ sufficiently small, let $T_{x, r}$ be a vector field chosen so that $|T_{x, r}|\leq 1$, $\supp T_{x, r} \ss B_{c(n)r}(x)$, and $\int_{\p^* \W } g(T_{x, r}, \nu) d\cH^{n-1} = \w_{n-1}r^{n-1}$ where $\w_{n-1}$ is the volume of the unit ball in $\R^{n-1}$; this may always be found by starting with the vector field $\nu_x \kappa_r$ in normal coordinates near $x$, with $\kappa_r$ a smooth cutoff function approximating $1_{B_r(x)}$, and then multiplying by a constant.

Below, recall the notation
\[
	K(x, y) = \frac{d\w_y}{d\cH^{n-1}\mres \p \W}(x)
\]
for the Poisson kernel, i.e. the Radon-Nikodym derivative of harmonic and surface measures, from \eqref{e:poissonkerneldef}.

First, in the following lemma we will show that the harmonic parts $h^{T_{x,r}}$, after suitable renormalization, converge to a multiple of the Poission kernel as $r \to 0$.  The main tool in the proof is the sharp estimates for the harmonic measure established in Proposition~\ref{l:green}.
\begin{lemma}\label{l:hconv} For every $x \in \p^* \W$, if $\supp T \ss B_r(x)$  then the function $h^T$ satisfies
	\[
		|h^T(y)| \leq \frac{C}{ d^{n-1}(y, x)} \int_{\p^* \W} |T| d \cH^{n-1} 
	\]
	whenever  $d(x, y) \geq C r$.

	For $\cH^{n-1}$-almost every $x \in \p^* \W$, if $h^x(y) = |\grad u_\W(x)| K(x, y)$ is a multiple of the Poisson kernel, then $h^x$ is well-defined, satisfies
	\[
		|h^x(y)| \leq \frac{C}{ d^{n-1}(y, x)},
	\]
	and
	\begin{equation} \label{e:hconv}
		\lim_{r \searrow 0} \sup_{y \in \W \sm B_{C r}(x)} \left|h^x(y) - \frac{h^{T_{x, r}}(y)}{\w_{n-1}r^{n-1}}\right| d^{n-1}(y, x) = 0.
	\end{equation}
\end{lemma}

\begin{proof}
	The first conclusion follows from the harmonic measure estimate \eqref{e:lgreen} in Proposition~\ref{l:green}, along with the fact that $|\grad u_\W|\leq C$ from Lemma \ref{cor:lip}. For the others, select $x$ with the following properties: (1) $x\in \p^* \W$, (2) $x$ is a Lebesgue point of $|\grad u|$ with respect to $\cH^{n-1}$, (3) $x$ is a Lebesgue point of $K(x, y)$ with respect to $\cH^{n-1}$ for a fixed $y \in \W$ with $B_{c_0}(y) \ss \W$. Then the Poisson kernel estimates of \eqref{e:poissonkerneldef} following Proposition~\ref{l:green} directly give that for any $z \in \W$,
	\[
		|h^x(z)| = |\grad u(x)| K(x, z) \leq \frac{C}{ d^{n-1}(z, x)}.
	\]
	
	As a first step toward proving \eqref{e:hconv}, we have that for this one fixed $y$,
	\begin{align*}
		\bigg|h^x(y)  &  - \frac{h^{T_{x, r}}(y)}{\w_{n-1}r^{n-1}}\bigg| = \left||\n u(x)| K(x, y) - \frac{h^{T_{x, r}}(y)}{\w_{n-1}r^{n-1}}\right| \\
		&=  \frac{1}{\w_{n-1}r^{n-1}} \left|\int_{\p^* \W \cap B_{cr}(x)}|\n u (x)| K(x, y) g(T_{x, r}(z), \nu_z)\,  d\cH^{n-1}(z) - {h^{T_{x, r}}(y)}\right|\\
		&= \frac{1}{\w_{n-1}r^{n-1}} \left|\int_{\p^* \W \cap B_{cr}(x)}|\n u (x)| K(x, y) g(T_{x, r}(z), \nu_z) - \dot{u}_\W^{T_{x, r}}(z) K(z, y)  \, d\cH^{n-1}(z)\right|\\
		&\leq  \frac{C}{r^{n-1}}\int_{\p^* \W \cap B_{cr}(x)}\big| |\n u (x)| K(x, y) g(T_{x, r}(z), \nu_z) + g(\grad u_\W(z), T_{x, r}(z)) K(z, y)\big|\, d\cH^{n-1}(z)\\
		&= \frac{C}{r^{n-1}}\int_{\p^* \W \cap B_{cr}(x)}\big| |\n u (z)| K(x, y) g(T_{x, r}(z), \nu_z) + g(\grad u_\W(z), T_{x, r}(z)) K(z, y)\big|\, d\cH^{n-1}(z) + o_r(1)\\
		&= \frac{C}{r^{n-1}}\int_{\p^* \W \cap B_{cr}(x)}\left| K(z, y) g(T_{x, r}(z), \grad u_\W(z)) - g(\grad u_\W(z), T_{x, r}(z)) K(z, y)\right| d\cH^{n-1}(z)| + o_r(1)\\
		&= o_r(1),
	\end{align*}
	using the definitions, both of the Lebesgue point assumptions, and the fact that $K(z, y), |\grad u_\W(z)|, |T_{x, r}(z)|$ are all bounded.

	Now, let us estimate the same quantity replacing $y$ with  any $z \in \W \sm B_{C r}(x)$. To this end, we use the triangle inequality and the definitions of $h^x$ and $h^{T_{r,x}}$ to write
	\[
		\bigg|h^x(z) -  \frac{h^{T_{x, r}}(z)}{\w_{n-1}r^{n-1}}\bigg|  = \left|h^x(y) \frac{K(x, z)}{K(x, y)} - \frac{h^{T_{x, r}}(z)}{\w_{n-1}r^{n-1}}\right| \leq I + II,
	\]
	where we set 
	\begin{align*}
		I & =\left|h^x (y) - \frac{h^{T_{x, r}}(y)}{\w_{n-1}r^{n-1}}\right| \frac{K(x, z)}{K(x, y)} \\
		II& =  \frac{C}{r^{n-1}}\left|\int_{\p^* \W \cap B_{cr(x)}} \dot{u}_\W^{T_{x, r}}(q)\left[K(q, z) - K(q, y) \frac{K(x, z)}{K(x, y)}\right] d\cH^{n-1}(q)\right|
	\end{align*}	
Since $d(x,y)\geq c_0$, from \eqref{e:poissonkerneldef} and Proposition~\ref{l:green} we have
	\[
		\frac{K(x, z)}{K(x, y)} \leq C K(x, z) \leq \frac{C}{d^{n-1}(z, x)}
	\]
	 and therefore we see that $I  \leq {o_r(1)}/{d^{n-1}(z, x)}$ using the first step above.
	Next, from \cite[Corollary 1.3.20]{Kenig},
	\[
		\left|\frac{d \w_z}{d \w_y}(x) - \frac{d \w_z}{d \w_y}(q)\right| \leq C \left|\frac{d \w_z}{d \w_y}(x)\right| d^\a(x, q) \leq C \frac{d^\a(x, q)}{d^{n-1}(z,x)}
	\]
	for some $\alpha \in (0,1)$, giving
	\[
		\left|K(q, z) - K(q, y) \frac{K(x, z)}{K(x, y)}\right| \leq C \frac{d^\a(x, q)}{d^{n-1}(z,x)}K(q, y) \leq C \frac{d^\a(x, q)}{d^{n-1}(z,x)}
	\]
	for any $q \in \p \W$. We therefore see that $II \leq {C r^\a}/{d^{n-1}(z, x)}$. 
	This proves \eqref{e:hconv}.
\end{proof}

We now move toward estimating the remainder $q^T$ and pass to a limit for the (renormalized) functions $q^{T_{x,r}}$ corresponding to the vector fields $T_{x,r}$ defined above.
The remainder $q^T$ satisfies a simpler boundary condition than $h^T$ (recall \eqref{eqn: PDE for q}), and we will use the following auxiliary elliptic estimate to help control it.

\begin{lemma}\label{l:l1est}
	Let $p$ be a bounded function on $\W$ with $\int p u_\W = 0$. Then
	\[
		\begin{cases}
			- \Lap u = \ei(\W) u + p & \text{ on } \W \\
			\int u u_\W = 0 &\\
			u = 0 & \text{ on } \p \W
		\end{cases}
	\]
	admits a unique solution $u \in H^1_0(\W)$, which has the estimate
	\begin{equation}\label{e:l1est}
		\| u \|_{L^{P}} \leq C(P) \| p\|_{L^1}
	\end{equation}
	for any $P \in[1, \frac{n}{n-2})$.
\end{lemma}

The existence and uniqueness are immediate from the Fredholm alternative (using that $p \in L^\infty$), as in \eqref{e:fredholm}. The main point is the estimate in terms of only the $L^1$ norm of $p$.

\begin{proof}[Sketch of proof.]
	The estimate
	\[
		\|u\|_{L^\infty(\W)} \leq C [\|p\|_{L^{P'}(\W)} + \|u\|_{L^2(\W)}]
	\]
	for any $P' > n/2$ may be found in \cite[Theorem 8.15]{GT}. Combining with the Hilbert space estimate from the Fredholm alternative (as in \eqref{e:fredholm})
	\[
		\|u\|_{L^2} \leq \|u\|_{H^1_0(\W)} \leq C \|p\|_{L^2(\W)}
	\]
	gives
	\[
		\|u\|_{L^\infty(\W)} \leq C \|p\|_{L^{P'}(\W)}.
	\]
	Then \eqref{e:l1est} follows from duality.
\end{proof}

\begin{corollary}\label{c:qconv} For almost every $x \in \p^* \W$, the functions ${q^{T_{x, r}}}/{\w_{n-1}r^{n-1}}$ converge in $L^{p}$ topology (for any $p \in [1, \frac{n}{n-2})$) to a function $q^x$, which solves
	\[
	\begin{cases}
		- \Lap q^x = \ei(\W)(q^x + h^x) -  u_\W|\grad u_\W(x)|^2  & \text{ on } \W \\
		\int_\W (q^x + h^x) u_\W = 0 &\\
		q^x = 0 & \text{ on } \p \W,
	\end{cases}
	\]
	in the sense that $q^x$ is the Green potential of the right-hand side:	for a.e. $y\in \W$,
	\begin{equation} \label{e:qconv2}
		q^{x}(y) = \int_\W \left[\ei(\W) (q^{x}(z) + h^{x}(z)) - u_\W(z)|\grad u_\W(x)|^2 \right] G_\W(z, y) dz\,.
	\end{equation}
\end{corollary}

It is straightforward to check that in fact $q^x$ is continuous (up to the boundary) away from $x$, and satisfies the PDE and boundary condition classically. However, the statement here is actually stronger: it implies that at $x$, $q^x$ still is $0$ in the sense of that it does not form  measure-valued boundary data there.

\begin{proof} Recalling the equation \eqref{eqn: PDE for q} satisfied by $q^T$, we wish to further decompose $q^T$ further as $q^T = q_1^T - u_\W \int h^T u_\W$ in order to apply the estimate of Lemma~\ref{l:l1est} above to the piece $q_1^T$, which is orthogonal to $u_\W$ by construction. Let us integrate by parts twice to compute
	\begin{equation}\label{e:qconv1}
		\begin{split}
			\ei(\W) \int_\W h^T u_\W &= - \int_\W h^T \Lap u_\W \\
			&= \int_\W g(\grad h^T, \grad u_\W) - \int_{\p^* \W} h^T g(\grad u_\W, \nu) d\cH^{n-1} = - \int_{\p^*\W} |\grad u_\W|^2 g(T, \nu) d\cH^{n-1}.
		\end{split}
	\end{equation}
	We used the divergence theorem of \cite{CTZ09}, the fact that $\Lap h^T = 0$, and Corollary \ref{c:ntlimits}. In particular, if we write $h_1^T = h^T - u_\W \int h^T u_\W$ (so that $h_1^T + q_1^T = h^T + q^T$), we have that
	\[
		\begin{cases}
	- \Lap q^T_1 = \l_1(\W) \left(q^T_1 + h^T_1\right)  & \text{ on } \W \\
	\int q^T_1 u_\W = 0 &\\
	q^T_1 = 0 & \text{ on } \p \W
		\end{cases}
	\]
	and $\int h_1^T u_\W = 0$. 
	
	We claim that  ${h^{T_{x, r}}_1}/{\w_{n-1}r^{n-1}}$ is a Cauchy sequence in $L^1$ (with respect to index $r$). To see this, apply Lemma \ref{l:hconv}, we know that $\lim_{r\searrow 0} \|\frac{h^{T_{x, r}}}{\w_{n-1}r^{n-1}} - h^x\|_{L^1} = 0$, and that 
	\[
	\lim_{r\searrow 0}\int u_\W \frac{h^{T_{x, r}}}{\w_{n-1}r^{n-1}} = \int h^x u_\W.
	\]
	From \eqref{e:qconv1} we see this implies that at each Lebesgue point $x$ of $|\grad u_\W|$ on $\p^* \W$ we have 
	\[
		\int h^x u_\W = \lim_{r \searrow 0} - \frac{1}{\ei(\W)\w_{n-1}r^{n-1}}\int_{\p^*\W} |\grad u_\W|^2 g(T_{x,r}, \nu_x) d\cH^{n-1} = - \frac{|\grad u_\W(x)|^2}{\ei(\W)}.
	\]

	 So, applying Lemma \ref{l:l1est}, we see that ${q^{T_{x, r}}_1}/{\w_{n-1}r^{n-1}}$ is Cauchy in $L^{p}$, and so converges to some function $q_1^x$. Passing the orthogonality condition to the limit, $\int q^x_1 u_\W = 0$. It then also follows that ${q^{T_{x, r}}}/{\w_{n-1}r^{n-1}} \rightarrow q^x := q^x_1 - u_\W \int h^x u_\W$.
	
	We have verified the orthogonality condition $\int_\W (q^x + h^x) u_\W = 0$. The boundary condition and PDE may be checked by writing
	\[
		q^{T_{x, r}}(y) = \int_\W \left[\l_1(\W) \left(q^{T_{x, r}}(z) + h^{T_{x, r}}(z)\right) - u_\W(z)\int_{\p^* \W} |\grad u_\W(y')|^2 g(T_{x, r}, \nu_{y'}) d\cH^{n-1}(y') \right] G_\W(z, y) dz,
	\]
	and passing both sides to the limit in $L^p$, which results in
	\[
	q^{x}(y) = \int_\W \left[\l_1(\W) (q^{x}(z) + h^{x}(z)) -  u_\W(z)|\grad u_\W(x)|^2\right] G_\W(z, y) dz
	\]
	for a.e. $y \in \W$.
\end{proof}

\begin{remark} \label{rem:L1estqh}
	We may recover further estimates for $q^x$ from this argument if we wish. For example, for any $\a, \b \in \R$,
	\[
		\|\a q^x + \b q^y\|_{L^1} \leq C \|\a h^x + \b h^y\|_{L^1} + C | \a |\grad u_\W(x)|^2 + \b |\grad u_\W(y)|^2|,
	\]
	by taking the limit in the corresponding estimates for $q^T$ with $T = \a T_{x, r} + \b T_{y, r}$. If we choose $ \a = \frac{1}{|\grad u_\W(x)|^2}$ and $\b = -\frac{1}{|\grad u_\W(y)|^2}$, the second term vanishes and we get
	\[
		\left\|\frac{q^x}{|\grad u_\W(x)|^2} - \frac{q^y}{|\grad u_\W(y)|^2}\right\|_{L^1} \leq C \left\|\frac{h^x}{|\n u_\W(x)|^2} - \frac{h^y}{|\n u_\W(y)|^2}\right\|_{L^1}.
	\]
	These may also be derived directly from the potential identity \eqref{e:qconv2}.
\end{remark}

Set $v^x_\W = q^x + h^x$. We are now in a position to derive a pointwise Euler-Lagrange equation, or free boundary condition, satisfied $\cH^{n-1}$-a.e. along $\p \W$.

\begin{corollary}[Pointwise form of the Euler-Lagrange equation]\label{c:elpt}
	Let $\W$ be a minimizer. Then there is a constant $A_0$ such that for almost every $x\in \p^* \W \cap \Qb_R$, we have the identity
	\begin{equation}\label{e:ptwiseel}
	-|\n u_\W(x)|^2  - \frac{\tpar}{2}|\n w_\W(x)|^2 + \err \left[ b_\W(x) + \int_{\W}a_\W v^x_\W \right] = - A_0.
	\end{equation}
	Here $v_\W^x = q^x +h^x$ where $q^x$ is as in \eqref{e:qconv2} and $h^x = |\nabla u_\W(x)|K(x ,\cdot)$. 
	If $|\W| \neq v$, then $A_0 = \fv'(|\W|)$.
\end{corollary}

\begin{proof}[Proof of Corollary \ref{c:elpt}]
	Take $x, y$ any two points in $\p^* \W \cap \Qb_R$ to which both Lemma \ref{l:hconv} and Corollary \ref{c:qconv} apply. Using $T_r = ({T_{x, r} - T_{y, r}})/{\w_{n-1}r^{n-1}}$ for $T$ in Lemma \ref{l:elvf}, we see that
	\[
		\left|\int_{\p^*\W} \left[-|\grad u_\W(q)|^2 - \frac{\tpar}{2}|\grad w_\W(q)|^2 + \err b_\W(q)\right] g(T_r(q), \nu_q) d\cH^{n-1}(q) + \err \int_{\W} a_\W \dot{u}_\W^{T_r}\right| \leq C r \rightarrow 0.
	\]
	Restricting further to only those $x, y$ which are Lebesgue points of both $|\grad u_\W|$ and $|\grad w_\W|$ with respect to $\cH^{n-1}\mres \p \W$, we see that the first term converges to
	\[
		|\grad u_\W(y)|^2 - |\grad u_\W(x)|^2 + \frac{\tpar}{2}\left[|\grad w_\W(y)|^2 - |\grad w_\W(x)|^2\right] + \err [b_\W(x) - b_\W(y)].
	\]
	From Lemma \ref{l:hconv}, we have that $h^{T_r} \rightarrow h^y - h^x$ in $L^1(\W)$, while from Corollary \ref{c:qconv}, $q^{T_r} \rightarrow q^y - q^x$ in $L^1(\W)$. Passing to the limit then gives
	\[
		|\grad u_\W(y)|^2 - |\grad u_\W(x)|^2 + \frac{\tpar}{2}\left[|\grad w_\W(y)|^2 - |\grad w_\W(x)|^2\right] + \err [b_\W(x) - b_\W(y)] + \err \int_\W a_\W [v^x_\W - v^y_\W] = 0.
	\]
	This implies the quantity in \eqref{e:ptwiseel} is independent of $x$.
	
	If $|\W| \neq v$, we may instead just use $T_r = {T_{x, r}}/{\w_{n-1}r^{n-1}}$; proceeding similarly gives that
	\[
		-|\n u_\W(x)|^2  - \frac{\tpar}{2}|\n w_\W(x)|^2 + \fv'(|\W|) + \err \left[ b_\W(x) + \int_{\W}a_\W v^x_\W \right] = 0.
	\]
	This completes the proof.
\end{proof}

\subsection{The Euler-Lagrange equation rewritten}\label{ss: EL rewritten}
As we discussed at the beginning of the section, the pointwise form of the Euler-Lagrange equation in Corollary~\ref{c:elpt} is not particularly useful from the perspective of regularity theory. The goal of this subsection is to prove Theorem~\ref{thm:usefulEL}, that is, to rewrite \eqref{e:ptwiseel} in a way which involves only $|\grad u_\W|$, up to ``lower order'' terms.

The most difficult term in \eqref{e:ptwiseel} to control is $\int_{\W}a_\W v^x_\W $, so let us give a brief heuristic idea of how this is done. Recall that $v^x_\W = h^x + q^x$; we will consider these two terms separately. The  term $h^x = |\n u_\W|  K(x, \cdot)$ is strictly harder to control, so let use focus on the term  $\int_{\W} a_\W h^x $. In order to compare this term to $|\n u_\W|^2,$ we reinterpret the equation solved by $h^x$ using that the Poisson kernel is the (normal) derivative of the Green's function. So,
\[
\int_\W h^x(y) \,dm(y) \leq  |\n u_\W(x)| \left|  \n_x \int_\W G_\W(x,y)a_\W(y) \, dm(y)\right| = |\n u_\W| | \n p(x)|,
\]
where $p(x)$ is the potential solving $\Lap p=a_\W$ with $p=0$ on $\p \W$.  From here, this term can  be handled in a similar way to the term $|\n w_\W|^2,$ using several applications of the inhomogeneous boundary Harnack principle.

Let us now move toward proving Theorem~\ref{thm:usefulEL}. Before proving the main theorem, we will need several lemmas that will allow us to control the other terms in \eqref{e:ptwiseel}.
 The first lemma shows that the ratio between the first eigenfunction and  a function $p$ that solves an equation is H\"{o}lder continuous up to the boundary. One application of this lemma will be to $p=w_\W$ in the proof of Theorem~\ref{thm:usefulEL}. It will also be used to help control the  
bad term $\int_{\W}a_\W v^x_\W $ in the Euler-Lagrange equation \eqref{e:ptwiseel}.

\begin{lemma}\label{lem:improvedbhrhs}
	Let $p$ be a continuous function with $p = 0$ on $\p \W$ and $- \Lap p = f$, $\|f\|_{L^\infty} \leq 1$. Then
	\begin{equation}\label{e:improvedbhrhs}
		\left\| \frac{p}{u_\W} \right\|_{C^{0, \a}(\bar{\W})} \leq C
	\end{equation}
	for some $C, \a$ depending only on $v, \vmax, R, \vpar$. Moreover, $\grad p$ admits nontangential limits at almost every point $x$ of $\p^*\W$, $|g(\grad p(x), \nu_x)| = |\grad p(x)|$, and
	\[
		\frac{|\grad p(x)|}{|\grad u_\W(x)|} = \lim_{y \rightarrow x \text{ n.t.}} \left|\frac{p(y)}{u_\W(x)}\right|.
	\]
	In particular, $|\grad p(x)| = |\grad u_\W(x)| \r_p(x)$ along $\p \W$, where $\|\r_p(x)\|_{C^{0, \a}(\p \W)} \leq C$.
\end{lemma}
 When one instead has the ratio of two harmonic functions, this fact follows from a standard argument iteratively applying the classical boundary Harnack inequality. The argument here is essentially similar, but we include the details since our functions of interest solve equations with a right-hand side and require the form of the boundary Harnack estimate established in \cite{AKS20}.
\begin{proof}
	First, let $p^+$ solve
	\[
		\begin{cases}
			- \Lap p^+ = f_+ & \text{ on } \W \\
			p^+ = 0 & \text{ on } \p \W.
		\end{cases}
	\]
	We will show that $p^+$ satisfies \eqref{e:improvedbhrhs}; then by also applying to $- p^-$, we obtain the statement for any $p$. Note that $0 \leq p^+ \leq w_\W$ from the maximum principle, so in particular from Proposition~\ref{l:bdryharnackef}, $p/u_\W \leq C$. 
	
	Fix $r < r_0$ small and $z \in \p \W$, and let
	\[
		M(r) = \sup_{\{y \in B_r(z) \cap \W: d(y, z) > t r\}} \frac{p^+}{u_\W} \qquad\qquad m(r) = \inf_{\{y \in B_r(z) \cap \W: d(y, z) > t r\}} \frac{p^+}{u_\W}
	\]
	where $t$ is small and fixed.	We claim that there are $c, \e > 0$ and a $\theta < 1$ such that 
	\[
		M(c r) - m(c r) \leq \theta [M(r) - m(r)] + r^\e.
	\]
	This claim implies \eqref{e:improvedbhrhs}: indeed, a classic iteration argument then gives that $M(r) - m(r) \leq C r^{\e'}$, and then applying this along a Harnack chain connecting any two points will give the conclusion.
	
	Our claim is immediate if $M(r) \leq r^\e$. If not, then there is a point $y \in B_r(z) \cap \W$ with $d(y, z) \geq t r$ and $p^+(y) \geq r^\e u_\W(y) \geq c r^{\e + 1}$. From the Harnack inequality, $p^+ \geq c r^{\e + 1} - C r^2 \geq c r^{\e + 1}$ on $B_{t r/2}(y)$ as long as $r_0$ was taken small enough. Then we have 
	\[
		p^+(y') \geq c r^{1 + \e}r^{n - 2} G_\W(y, y') \geq c r^{n-1 + \e}G_\W(y, y')
	\]
	for $y'$ along $\p B_{t r/2}(y)$, and by comparison principle on the entire compliment of this ball; we used \eqref{e:green} here. Applying the Green's function bounds from Proposition~\ref{l:green} gives that for $x \in B_{c r}(z)$,
	\[
		p^+(x) \geq c r^{n -  1 + \e} G(y, x) \geq c r^\e d(x, \p \W) \geq c d^{1 + \e}(x, \p \W).
	\]
	
	We now apply the inhomogeneous boundary Harnack principle \cite[Theorem 1.3]{AKS20} (working in normal coordinates and scaling suitably) on this region $B_{c r}(z)$ to
	\[
		M(r) - \frac{p^+}{u_\W} = \frac{M(r)u_\W + p^+}{u_\W}
	\]
	and
	\[
		\frac{p^+}{u_\W} - m(r) = \frac{p^+ - m(r) u_\W}{u_\W}.
	\]
	Both of these are quotients of positive functions vanishing on $\p \W$ and having bounded Laplace-Beltrami operator, and they both satisfy the growth condition with exponent $\b = 1 + \e$. This implies that
	\[
		M(r) - m(cr) \leq C[M(r) - M(c r)] \qquad\qquad  M(cr) - m(r) \leq C[m(cr) - m(r)].
	\]
	These may be rewritten to give $M(cr) - m(cr) \leq \frac{C-1}{C+1}[M(r) - m(r)]$, proving our claim.
	
	Now we consider the remaining conclusions. As $|p(x)|\leq C u_\W(x) \leq C d(x, \p \W)$, it follows from applying elliptic estimates on $B_{d(x, \p \W)/2}(x)$ that $|\grad p|\leq C$. Let $\r_p = p/u_\W \in C^{0, \a}$; then near any $x\in \p^*\W$ we have from Lemma \ref{l:tangentplane} that
	\[
		p(y) = \r_p(y) u_\W(y) = \r_p(x)[|\grad u_\W(x)|l_{x, \nu_x} + o(|x - y|)].
	\]
	Here $l_{x, \nu_x}$ is a truncated linear function in normal coordinates as defined in \eqref{eqn: truncated linear}. Applying elliptic estimates as in Corollary \ref{c:ntlimits}, we obtain that $\grad p(y) \rightarrow - \r_p(x) |\grad u_\W(x)| \nu_x$ nontangentially and the remaining conclusions follow.
\end{proof}

We focus our attention toward the term $\int_{\W}a_\W v^x_\W $. Following the discussion above, a key point will be to rewrite the equation solved by $v^x_\W$ using that the Poisson kernel is the derivative of the Green's function:
\begin{lemma}\label{lem:greentopoisson}
	The derivative of the Green's function $\grad_x G_\W(x, y)$ admits nontangential limits as $x \rightarrow z \in \p \W$ for every $y$ at $\cH^{n-1}$-a.e. $z \in \p^* \W$. Moreover, $\grad_x G_\W(x, y) = - \nu_x |\grad_x G_\W(x, y)|$ and $K(x, y) = |\n_x G_\W(x, y)|$ a.e. on $\p^* \W$.
\end{lemma}

\begin{proof}[Sketch of proof.]
	The existence of nontangential limits and the expression $\grad_x G_\W(x, y) = - \nu_x |\grad_x G_\W(x, y)|$ may be obtained as in Lemma \ref{lem:improvedbhrhs}. The last point follows from the integration by parts formula
	\[
	\int_{\p^* \W } g(\grad_x G_\W(x, y), \nu_x) \phi(x) d \cH^{n-1}(x) = - \phi(y) - \int_{\W} G_\W(x, y) \Lap \phi(x) dx
	\]
	for any $\phi$ smooth (this may be justified using the divergence theorem of \cite{CTZ09}). This gives a representation formula for every harmonic function on $\W$ with continuous boundary conditions, so it follows from the definition of harmonic measure that $K(x, y) = |\n_x G_\W(x, y)|$ at $\cH^{n-1}$-a.e. $x$.
\end{proof}

The two lemmas will be used to control the term $\int_\W a_\W v_\W^x$ in \eqref{e:ptwiseel}. Recall from the previous section that $v_\W^x = h^x + q^x$. We focus on the term $h^x$ first.
\begin{lemma}\label{l:hbh}
	Let $f : \W \rightarrow \R$ be a function with $|f|\leq 1$. Then at a.e. $x \in \p \W$,
	\[
		\int_\W h^x f = |\grad u_\W(x)|^2 \r_f(x),
	\]
	where $\|\r_f\|_{C^{0, \a}} \leq C$ and $C, \a$ depend only on $v, \vmax, \vpar, R$. As a consequence,
	\[
		\left\|\frac{h^x}{|\grad u_\W(x)|^2} - \frac{h^y}{|\grad u_\W(y)|^2}\right\|_{L^1} \leq C d^\a(x, y).
	\]
\end{lemma}

\begin{proof}
	From the definition of $h^x$ and Lemma \ref{lem:greentopoisson}, at a.e. $x \in \p^* \W$ we have
	\[
		\int_{\W}h^x f = |\grad u_\W(x)| \int_\W K(x, y) f(y) dy = - |\grad u_\W(x)| \int_\W g(\grad_x G_\W(x, y), \nu_x) f(y) dy.
	\]
	Let $x_k \in \W$ be a sequence of points converging to $x$ nontangentially (i.e. $\e d(x_k, x) \leq d(x_k, \p \W)$): then $\grad_{x} G_\W(x_k, y)$ converges to $\grad_x G_\W(x, y)$ for every $y$. From \cite[Lemma 1.2.8(iv)]{Kenig},  $\sup_x \|\grad_x G_\W(x, y)\|_{L^{P}(dy)} \leq C$ for $P \in [1, \frac{n}{n-1})$. Working in normal coordinates and applying the Vitali convergence theorem, this means that
	\[
		\int_\W g(\grad_x G_\W(x, y), \nu_x) f(y) dy = \lim_k \int_\W g(\grad_x G_\W(x_k, y), \nu_x) f(y) dy.
	\]

	Set $p(z) = \int_{\W}G_\W(z, y) f(y) dy$, which solves
	\[
	\begin{cases}
	-\Lap p = f & \text{ on } \W\\
	p = 0 & \text{ on } \p \W.
	\end{cases}
	\]
	From Lemma \ref{lem:improvedbhrhs}, $\grad p$ admits nontangential limits $\cH^{n-1}$-a.e. on $\p \W$ and $|\grad p(x)| = |\grad u_\W(x)| \r_f(x)$, with $\|\r_f\|_{C^{0, \a}} \leq C$. In particular,
	\[
		g(\grad p(x), \nu_x) = \lim_k g(\grad p(x_k), \nu_x) = \lim_k g(  \grad_x \int_\W G_\W(x_k, y) f(y) dy, \nu_x).
	\]
	We may pass the derivative under the integral sign (again using that $\sup_x \|\grad_x G_\W(x, y)\|_{L^{p}(dy)} \leq C$ to justify this), to obtain
	\[
		|\grad u_\W(x)| \r_f(x) = \lim_k \left|\int_\W g(\grad_x G_\W(x_k, y), \nu_x) f(y) dy\right| = \frac{1}{|\grad u_\W(x)|} \left|\int_{\W}h^x f\right|.
	\]
	The first conclusion of the lemma follows using $\r_f \sign \int_{\W}h^x f$ as the $\r_f$. The $L^1$ estimate is now immediate from duality.
\end{proof}

Next, the term involving $q^x$ is much easier to handle.
\begin{lemma}\label{l:qbh} Let $f : \W \rightarrow \R$ be a function with $|f|\leq 1$. Then at a.e. $x \in \p \W$,
	\[
	\int_\W q^x f = |\grad u(x)|^2 \r_f(x),
	\]
	where $\|\r_f\|_{C^{0, \a}} \leq C$ and $C, \a$ depend only on $R, v, \vmax, \vpar$.
\end{lemma}

\begin{proof}
	We apply Lemma \ref{l:hbh} to Remark \ref{rem:L1estqh} to obtain
	\[
		\left\|\frac{q^x}{|\grad u(x)|^2} - \frac{q^y}{|\grad u(y)|^2}\right\|_{L^1} \leq C \left\|\frac{h^x}{|\n u(x)|^2} - \frac{h^y}{|\n u(y)|^2}\right\|_{L^1} \leq C d^\a(x, y).
	\]
\end{proof}

We are now in a position to prove Theorem~\ref{thm:usefulEL}.
\begin{proof}[Proof of Theorem~\ref{thm:usefulEL}]
	Apply Lemmas \ref{l:hbh} and \ref{l:qbh} to $h^x$, $q^x$ with $f = a_\W$ to obtain that for almost every $x \in \p^* \cH^{n-1}$,
	\[
		\int_{\W}a_\W(x) v^x_\W = |\grad u(x)|^2 \rho_1(x),
	\]
	where $\|\rho_1\|_{C^{0, \a}(\p \W)}\leq C$. Then apply Lemma \ref{lem:improvedbhrhs} to $p = w_\W$ to give
	\[
		|\grad w_\W(x)|^2 = |\grad u_\W(x)|^2 \rho_2(x)
	\]
	where $\|\rho_2\|_{C^{0, \a}(\p \W)}\leq C$.
Inserting both into Corollary \ref{c:elpt} gives
	\[
		|\grad u_\W (x)|^2\left( -1 - \frac{\tpar}{2} \r_2(x) + \err \r_1(x)\right) = - A_0 - \err b_\W(x).
	\]
	So long as $\err_0$ is small enough, $-1 - \frac{\tpar}{2} \r_2(x) + \err \r_1(x) \leq -\frac{1}{2}$. From Lemma \ref{l:tangentplane}, we have that $|\grad u_\W| \in [1/C, C]$ a.e., so $A_0$ must also be bounded above and below. Choosing $\err_0$ smaller as necessary, we may rewrite
	\[
		A_0 = |\grad u_\W (x)|^2\frac{1 + \frac{\tpar}{2} \r_2(x) - \err \r_1(x)}{1 - \frac{\err}{A_0} b_\W(x)} = |\grad u_\W (x)|^2 (1 + \rho(x)),
	\]
	where $\|\rho\|_{C^{0, \a}(\p \W)}\leq C$ and $- C \err \leq  \rho$.
\end{proof}

\section{Viscosity form of the Euler-Lagrange equation and $C^{1,\a}$ estimates}\label{s:reg}
Thus far, we have derived the Euler-Lagrange equation satisfied by minimizers of the main energy in a pointwise ($\cH^{n-1}$ a.e.) sense and we established Theorem~\ref{thm:usefulEL} to express the Euler-Lagrange equation in a more useful form. In this section, we reformulate the (useful form of the) Euler-Lagrange equation in the viscosity sense. This will allow us to apply known regularity results for one-phase free boundary problems. In particular, when minimizers of the base energy are sufficiently regular (for instance, this is the case when $M$ is a simply connected space form), we find in Theorem~\ref{thm:C1a} below that any minimizer of the main energy is a small $C^{1,\alpha}$ perturbation a minimizer of the base energy. We additionally use the viscosity form of the Euler-Lagrange equation to guarantee that minimizers satisfy the volume constraint $|\W| =v$.

The viscosity form of the Euler-Lagrange follows directly from the pointwise version and generic free boundary arguments, and we only briefly sketch the proof. Details may be found in \cite{KL18}.

\begin{lemma}\label{l:viscel}
	Theorem \ref{thm:usefulEL} holds in the viscosity sense: let $\W$ be a minimizer of $\Ep$, $x \in \partial \W \cap \Qb_R$, and $\phi$ a smooth function on $B_r(x)$. Assume that $\phi_+ \leq(\geq) u_\W$ on $B_r(x)$ and $\phi(x) = 0$. Then
	\[
		|\grad \phi(x)|^2( 1 + \r(x)) \leq (\geq) A_0.
	\]
\end{lemma}

\begin{proof}[Sketch of proof.]	
	We work in normal coordinates around $x$. Let $u_r(y) = u_\W(r y)/r$ be rescalings, which converge (along a subsequence) locally uniformly on $\R^n$ to a function $u$. Similarly rescaling $\phi_r(y) = \frac{\phi(r y)}{r}$, this converges to the linear function $\phi_\infty(y) = \grad \phi(0) \cdot y$. We have that $-\Lap u_r \rightarrow -\Lap u$ in the sense of distributions, so in particular as measures; it follows that $u$ is harmonic on $\{u > 0 \}$. As $u_r \geq(\leq) (\phi_r)_+$, this passes to the limit to give $u \geq (\leq) \phi_\infty$. In particular, $\{u > 0\}$ contains (is contained in) the half-space $\{y: \grad \phi(0) \cdot y > 0 \}$. We also recover that $c s \leq \sup_{B_s(0)} u \leq C s$ in the limit.
	
	Possibly choosing a further subsequence, it may be shown that $u$ is a half-linear function $\a (y \cdot \nu)_+$, with $\nu = \n \phi /|\n \phi|$. See \cite[Lemma 11.17]{CaffSalsa} or \cite[Lemma 3.6]{KL18} for details.
	
	As measures, we have that $- \Lap u_r = |\n u_r| d\cH^{n-1}\mres \p \W + o_r(1)$, and $|\n u_r(y)| = \sqrt{A_0/(1 + \r(r y))}$ a.e. on $\p \W$. Passing to the limit in the sense of distributions shows that
	\[
		\int_{\p^* \W/r}\psi |\grad u_r(y)| d\cH^{n-1} = - \int \psi d\Lap u_r + o_1(r) \rightarrow - \int \psi d\Lap u = \int_{y \cdot \nu = 0} \a \psi d\cH^{n-1}
	\]
	for any smooth $\psi$. The left-most integral converges to
	\[
		\sqrt{A_0/(1 + \r(0))} \lim_r \int_{\p^* \W/r} \psi d\cH^{n-1}, 
	\]
	and as $\W/r \rightarrow \{ y \cdot \nu > 0 \}$ locally as sets of finite perimeter (possibly taking a further subsequence), this limit is
	\[
		\sqrt{A_0/(1 + \r(0))} \int_{y \cdot \nu = 0} \psi d\cH^{n-1}.
	\]
	Therefore $\a = A_0/(1 + \r(0))$, and the conclusion follows from $\a \geq (\leq)|\grad \phi(0)|$.
\end{proof}

The following is an application of a regularity theorem, first shown in \cite{AC81} in simpler settings and generalized by De Silva \cite{DeSilva11} to equations and free boundary conditions which include the one satisfied by $u_\W$ here. It guarantees that at sufficiently flat points of the boundary $\p \W$, it may be represented as a $C^{1, \a}$ graph over a tangent plane.

\begin{lemma}\label{l:ds}
	There is a $\d = \d(v, \vmax, \vpar, R)$ such that the following holds. Fix any $x\in \p \W\cap \Qb_R$ and $r < \d$ with $B_r(x) \ss \Qb_R$. If $|\W \triangle B_{r, \nu}(x) | < \d |B_r(x)|$ then we may parametrize the entirety of $\p \W\cap B_{r/2}(x)$ as a $C^{1, \a}$ normal graph over the lateral boundary $\p B_{r, \nu}(x) \cap B_r(x)$, with $C^{1, \a}$ norm bounded in terms of only $v, \vmax, \vpar, R$.
\end{lemma}

\begin{proof}
	Choosing $\d$ small, Lemma \ref{l:tangentplane} implies that
	\[
		B_{c r}(x) \cap \p \W \ss \exp_x \{ \mathscr{v} \in T_x M : |\mathscr{v}| < r,  |g(\mathscr{v}, \nu)| < \d' r \}
	\]
	and
	\[
		\sup_{B_{cr}(x)} |u_\W - l_{x,\nu}| \leq \d' c r
	\]
	for a small $\d'$. Applying Lemma \ref{l:viscel}, we see that $\W$ and $u_\W$ is a viscosity solution to the free boundary problem
	\[
		\begin{cases}
			- \Lap u_\W  = \ei(\W) u_\W &\text{ on } B_{cr}(x) \cap \W\\
			u_\W > 0 &\text{ on } B_{cr}(x) \cap \W\\
			u_\W = 0 &\text{ on } B_{cr}(x) \cap \p \W\\
			|\grad u_\W | = f(x) &\text{ on } B_{cr}(x) \cap \p \W,\\
		\end{cases}
	\]
	where $0 < c \leq f \leq C < \infty$ and $\|f\|_{C^{0, \a}(\p \W)} \leq C$. Applying the main theorem of \cite{DeSilva11}, it follows that $B_{c r/2}(x) \cap \p \W$ may be parametrized, in normal coordinates, as a $C^{1, \a}$ graph over $\{e \cdot \nu = 0 \}$, with bounded $C^{1, \a'}$ norm for some $\a' > 0$. The conclusion now follows from a standard covering argument.
\end{proof}

\begin{remark}\label{r:smallgraph}\rm{
	The previous lemma actually implies that if, working in normal coordinates around $x$, $\p \W = \{ (x', f(x')) : |x'| < \frac{r}{2} \}$ on $B_{r/2}(x)$ with $\|f\|_{C^{1, \a}} \leq C$, then $f$ is small. In particular, Lemma \ref{l:tangentplane} gives that $|f| = C \d^{1/n}$, while $[\grad f]_{C^{0, \a'}} \leq C r^{\a - \a'} \leq C \d^{\a - \a'}$ for $\a' < \a$. This means that $\|\grad f\|_{C^{0, \a'}} \leq C \d^{\z}$ for some $\z > 0$.}
\end{remark}

In the next theorem, we say that a collection of open sets, in this case $\Min$, is \emph{uniformly} $C^{k, \a}$ if there is a constant $C$ such that for every $x \in \p U$ for $U \in \Min$, $B_{1/C}(x)$ admits normal coordinates, and on this ball $\p U$ may be expressed as a graph with $C^{k, \a}$ norm at most $C$ over a hyperplane in normal coordinates.

\begin{theorem}\label{thm:C1a}
	Assume that $\Min$ is uniformly $C^{2, \a}$. Then for every $r_1 > 0$ there is an $\err_1 = \err_1(v, \vmax, \vpar, R, r_1) > 0$ such that if $\W$ minimizes $\Ep$ and $\err < \err_1$, then $\p \W \cap \{ x \in \Qb_R: d(x, \p \Qb_R) > r_1 \}$ may be parametrized as a $C^{1, \a}$ normal graph (with $C^{1, \a}$ norm bounded by $r_1$) over $\p U$ for some $U \in \Min$.
\end{theorem}

If it is known that $\W \cc \Qb_R$ (uniformly), then one may take a fixed $r_1 < d(\W, M \sm \Qb_R)$ and obtain that the entire boundary $\p \W$ may be parametrized in this manner. In particular, this always applies to the case of $M / \Isom_0$ compact, using Theorem \ref{t:globalexist}. The assumption that $\p U \in C^{2, \a}$ is not needed to guarantee that $\p \W \in C^{1, \a}$ (even a flatness condition at every point would suffice), but it is necessary to express $\p \W$ as a graph over $\p U$.

\begin{proof}
	Let $S = \{ x \in \Qb_R: d(x, \p \Qb_R) > r_1 \}$, and let $\d$ be the parameter from Lemma \ref{l:ds}. First, observe that under the assumptions on $\Min$, there is an $0 < r' < \d$ such that for every $x \in \p U \cap S$ for some $U \in \Min$, $|B_{r', \nu_x}(x) \triangle U| < |B_{r'}(x)| \frac{\d}{4}$ and $B_{r', \nu_x}(x) \sm B_{r'}(x)$ may be represented as a $C^{1, \a}$ normal graph over $\p U$ with $C^{1, \a}$ norm bounded by $1/r'$. Indeed, this follows from the regularity assumption and the tubular neighborhood theorem.

	Choose $\err_1$ small enough that $\En(\W)\leq \minE + \err \leq \minE + \err_1$ and Lemma \ref{l:simpleeval} imply that $|\W \triangle U| \leq \frac{\d}{4} \inf_{x \in \Qb_R}|B_{r'}(x)|$ for some $U \in \Min$. Up to further decreasing $\err_1,$ we may apply Lemma \ref{l:densitybd} to take Hausdorff distance between $\partial U$ and $\partial \W$ smaller than $\kappa  r_1/2,$ where $\kappa$ is a fixed constant to be specified below. Then for any $x \in \p \W \cap S$, let $y \in \p U \cap S$ with $d(x, y) < \k r'$, and compute in normal coordinates around $y$:
	\begin{align*}
		|\W \triangle B_{r', \nu_y}(x)| &\leq |U \triangle \W|  + |B_{r', \nu_y}(y) \triangle U| + |B_{r', \nu_y}(x) \triangle B_{r', \nu_y}(y)|\\
		& \leq \frac{\d}{4} |B_{r'}(x)| + \frac{\d}{4} |B_{r'}(x)| + |B_{r', \nu_y}(x) \triangle B_{r', \nu_y}(y)|  \leq \left[\frac{\d}{2} + C\k \right]|B_{r'}(x)|
	\end{align*}
	where $C$ is a constant depending only on the metric. By choosing $\k$ small enough depending on $C$ and $\delta$, the right-hand side is bounded above by $\d |B_{r'}(x)|$, so the hypotheses of Lemma~\ref{l:ds} are satisfied at $x$ at scale $r'$ in direction $\nu_y.$
	
	Apply Lemma \ref{l:ds} (and Remark \ref{r:smallgraph}) on $B_r(x)$ to obtain that $\p \W \cap B_{r'/2}$ may be expressed as a $C^{1, \a}$ graph over $B_{r', \nu_x}(x) \sm B_{r'}(x)$ with$C^{1, \a}$ norm bounded by $c r_1$ for a small $c$. We may then parametrize $\p \W$ over $\p U$ instead by composing with the parametrization of $B_{r', \nu_x}(x) \sm B_{r'}(x)$ over $\p U$ to obtain the conclusion.
\end{proof}

Another application of our Euler-Lagrange equations is the following fact about the volume $|\W|$:

\begin{proposition}\label{l:volumeisright}
	There is a $\vpar_*(v, \vmax, R) > 0$, $\tpar_*(v, \vmax, R, \vpar)>0$, and $\err_*(v, \vmax, R, \vpar, \tpar) > 0$ such that for any $\vpar < \vpar_*$, $\tpar < \tpar_*$, and $\err < \err_*$, every minimizer $\W$ of $\Ep$ has $|\W| = v$. If $M / \Isom_0 $ is compact, all constants may be taken independent of $R$.
\end{proposition}

\begin{proof}
	If $M / \Isom$ is compact, we fix $\vpar$ small and select $R = R(\vpar)$ large enough that $\W \cc \Qb_R$; this is always possible from Theorem \ref{t:globalexist}. All constants will then be independent of this $R(\vpar)$, as we will verify below.
	
	If $|\W| \neq v$, we apply Corollary \ref{c:elpt} to give that $A_0 = \fv'(|\W|)$ and
	\[
	|\grad u_\W(x)|^2 + \frac{\tpar}{2} |\grad w_\W(x)|^2 - \err \left[b_\W(x) + \int_{\W}a_\W v^x_\W\right] = \fv'(|\W|).
	\]
	From Theorem \ref{thm:usefulEL}, we have that $|\grad w_\W(x)|^2 \leq C_0 |\grad u_\W(x)|^2$ and $|\int_{\W}a_\W v^x_\W| \leq C_0 |\grad u_\W(x)|^2$, where this constant $C_0$ depends on $\vpar$. Now select $\tpar_*$ small enough (in terms of $\vpar$) so that $\frac{\tpar}{2} |\grad w_\W(x)|^2 \leq \frac{1}{10}|\grad u_\W(x)|^2$. Then select $\err_*$ small enough in terms of $\vpar$ so that $\err |\int_{\W}a_\W v^x_\W| \leq \frac{1}{10}|\grad u_\W(x)|^2$, and also so that $\err |b_\W| \leq \vpar/2$. This gives
	\[
	\frac{1}{3}\fv'(|\W|) \leq |\grad u_\W|^2 \leq 3 \fv'(|\W|)
	\]
	along $\p^* \W \cap \Qb_R$.
	
	Consider the integration by parts formula
	\[
		\frac{1}{3}\sqrt{\fv'(|\W|)}\cH^{n-1}(\p \W \cap \Qb_R) \leq - \int_{\p^* \W} g(\grad u_\W(x), \nu_x) d\cH^{n-1} = \int_\W - \Lap u_\W \leq C(R, M)|\W|.
	\]
	We have that
	\begin{equation}\label{e:volumeisright}
		\inf\left\{ \frac{\cH^{n-1}(\p \W \cap \Qb_R)}{|\W|} : \W \in \compset \right\} > c = c(R, \vmax) > 0
	\end{equation}
	from the relative isoperimetric inequality on $\Qb_R$, \cite[4.5.2(2)]{Federer}, so $\fv'(|\W|) \leq C(R)$. Choosing $\vpar_*$ so that $\vpar_* < \frac{1}{C(R)}$, we see that $\fv(|\W|) < \frac{1}{\vpar_*}$ and $|\W| \leq v$. In the case of $M /\Isom_0$ compact, $\p \W \cap \Qb_R = \p \W$, we instead use the isoperimetric inequality: for example, from \cite[Theorem 3.2]{HebeyBook} we have $|\W|^{\frac{n-1}{n}} \leq C (\cH^{n-1}(\p \W) + |\W|)$ with a constant depending only on $M$, and this bounds \eqref{e:volumeisright} from below by $c(M)\vmax^{-1/n}$.
	
	It follows that if $|\W| \neq v$, then along $\p^*\W \cap \Qb_R$ the derivative $|\grad u_\W|\leq 9 \vpar_*$ is small. On $\p \W \cap \p \Qb_R$, an elementary comparison argument with $w_{\Qb_R}$ shows that $|\grad u_\W|\leq C(R)$. From Bochner's identity we have that $\Lap (|\grad u_\W|^2 + S u_\W^2) \geq - 2 \ei(\W) S u_\W^2$ for $S$ taken large enough in terms of the Ricci curvature of $(M,g)$. Using Lemma \ref{l:efbdd} and the maximum principle, we see that $|\grad u_\W| \leq C(R)$ on $\W$ (the point being that this bound is independent of $\eta$); the constant may be taken independent of $R$ if $M /\Isom_0$ is compact.
	
	Now, $|\W| < v$ and $\int u_\W^2 = 1$, so there must be a point $x\in \W$ with $u_\W(x) \geq \frac{1}{2}\sqrt{v}$. Using the gradient estimate above, there is a ball $B_{r}(x) \ss \W$, with $r$ depending only on $R$ and $v$. We construct a sequence of balls $B_r(x_k) \ss \W$ as follows: if there is an $x \in B_{r/2}(x_{j})$ for some $j < k$ with $B_r(x) \ss \W$ and $z \in \p B_r(x) \cap \p \W \cap \Qb_R$, set $x_k = x$ and stop. If not, choose any ball $B_r(x)$ with $d(x, x_j) \geq r/2$ and equality for one $j$. If this is impossible, then $\p \W = \p \Qb_R$, which contradicts our standing assumption that $\vmax < |\Qb_R|$. This process must terminate after $C(r)\vmax$ steps, and so we have a chain of balls culminating in $B_r(x_J)$ which has $z \in \p B_r(x_J) \cap \p \W \cap \Qb_R$. Applying the Harnack inequality along this chain, we have $u_\W \geq c(r, J)$ on $B_{r/2}(x_J)$. Letting $G'$ be the Green's function for $B_r(x_J)$, we have that $C(r, J) u_\W(y) \geq  G'(x_J, y)$ on $B_r(x_J) \sm B_{r/2}(x_J)$ from the comparison principle. From standard estimates on the Green's function (see \cite[Theorem 1.2.8]{Kenig}),
	\[
		C(r, J) u_\W(y) \geq G'(x_J, y) \geq c(r, R) d(z, y).
	\]
	From the viscosity form of the Euler-Lagrange equation, Lemma \ref{l:viscel}, we then must have that 
	\[
		c(r, R) \leq |\grad G'(x_J, z)| \leq \sqrt{3 \fv'(|\W|)} \leq 3 \vpar_*.
	\]
	If $\vpar_*$ is chosen small in terms of $r, R$, this gives a contradiction. If $M /\Isom_0$ is compact, it is easy to see that all constants here may be taken uniformly in $R$.
\end{proof}

\section{Higher Regularity}\label{s:highreg}

For the purpose of deriving stability estimates, it will be helpful to have $C^{2, \a}$ regularity for $\p \W$. Higher regularity depends on the regularity of the functions $a_\W, b_\W$ which appear in property \ref{a:nlc1} of $\nl$. In particular, in the case of $a_\W$ being a multiple of $u_U$ for a smooth domain $U$, $a_\W$ is at best Lipschitz continuous on $\W$ (unless $\W = U$). Under this assumption we will have $\p \W$ locally $C^{2, \a}$ at flat points for any $\a < 1$ and not better.
Our approach is based on a boundary Harnack  estimate of De Silva and Savin.

\begin{proposition}[\cite{DSS15}, Theorem 2.4]\label{prop:higherbh}
	Let $f \in {C^{0,\alpha}(\W)}$ satisfy $ \|- \Lap f\|_{C^{0,\a}(\W)} \leq 1$ {(with $-\Delta$ taken in the classical sense)}, $f = 0$ on $\p \W$, and $\W$ be a minimizer of $\Ep$. Assume that for some $x_0\in \p \W$, $r < r_0$, and $\d > 0$ small enough we have $B_r(x_0)\ss \Qb_R$ and $\p \W \cap B_{r/2}(x_0)$ a $C^{1, \a}$ graph over $\{e: \nu \cdot e = 0 \}$ (in normal coordinates around $x_0$) with $C^{1, \a}$ norm bounded by $C_\a$. Then
	\[
		\left\|\frac{f}{u_\W}\right\|_{C^{1,\a}(\bar{\W}\cap B_{r/4}(x_0))} \leq C
	\]
	and
	\[
		\left\|\frac{g(\grad f, \nu_x)}{|\n u_\W|}\right\|_{C^{1,\a}(\p \W \cap B_{r/4(x_0)})} \leq C.
	\]
	Here $\a \in (0, 1)$ and the constants depend only on $R, v, \vmax, \vpar, \a, C_\a$.
\end{proposition}

\begin{proof}
	{We consider the operator $-\Delta -\lambda(\Omega)$, and note that $-\Delta u_{\Omega} -\lambda(\Omega)u=0$ and $\|-\Delta f - \lambda(\Omega) f\|_{C^{0,\alpha}} \leq C$}. From Theorem 2.4 in \cite{DSS15},
	\[
		\left\|\frac{f}{u_\W}\right\|_{C^{1,\a}(\bar{\W}\cap B_{r/4}({x_0}))} \leq C\left[\|f\|_{{L^\infty}} + \|\Lap f\|_{C^{0, \a}}\right] \leq C,
	\]
	using also the lower bound on $u_\W$. The second conclusion follows as $\frac{f}{u_\W} \rightarrow \frac{g(\grad f, \nu_x)}{g(\grad u_\W, \nu_x)}$ nontangentially at every point of $\p \W \cap B_{r/2}(x_0)$, and then passing to the limit.
\end{proof}

\begin{lemma}\label{l:hhighreg} Let $\W$ and $x_0$ be as in Proposition \ref{prop:higherbh}. Then
	\[
		\left\|\frac{1}{|\grad u_\W|^2}\int_\W h^x f\right\|_{C^{1, \a}(\p \W \cap B_{r/4}(x_0))} \leq C \|f\|_{C^{0, \a}(\W)}.
	\]
\end{lemma}

\begin{proof}
	As in the proof of Lemma \ref{l:hbh}, we have that
	\[
		\int_{\W} h^x f = |\n u_\W(x)| \int_{\W} K(x, y) f(y) dy = - |\n u_\W(x)|\, g\Big(\grad_x \int_{\W} G_\W(x, y) f(y) dy, \nu_x\Big).
	\]
	If we set $p(x) = \int_{\W} G_\W(x, y) f(y) dy$, this satisfies $-\Lap p = f$ on $\W$ and $p = 0$ on $\p \W$; applying Proposition \ref{prop:higherbh} gives
	\[
		\left\|\frac{g(\grad p, \nu)}{|\grad u_\W|} \right\|_{C^{1, \a}(\p \W \cap B_{r/4}(x_0))} \leq C \|f\|_{C^{0, \a}}.
	\]
	The conclusion follows from
	\[
		\frac{1}{|\n u_\W(x)|^2}\int_{\W} h^x f = \frac{- g(\grad p(x) \cdot \nu)}{|\n u_\W(x)|}.
	\]
\end{proof}

\begin{lemma}\label{l:qhighreg} Let $\W$ and $x_0$ be as in Proposition \ref{prop:higherbh}. Then
	\[
	\left\|\frac{1}{|\n u_\W|^2}\int_\W q^x f\right\|_{C^{1, \a}(\p \W \cap B_{r/4}(x_0))} \leq C \|f\|_{C^{0, \a}}.
	\]
\end{lemma}

The function $q^x$ is actually somewhat better behaved than this, but we will not require an optimal estimate below.

\begin{proof}
	Set $\bar{h}^x = \frac{h^x}{|\n u(x)|^2}$ and $\bar{q}^x = \frac{q^x}{|\n u(x)|^2}$. Rewriting the equation for $q$ in terms of these, we have that
	\[
		\begin{cases}
			- \Lap \bar{q}^x = \ei(\W) [ \bar{q}^x + \bar{h}^x] - u_\W & \text{ on } \W\\
			\int \bar{q}^x u_x = - \int \bar{h}^x u_x = - \frac{1}{\ei(\W)} & \\
			\bar{q}^x = 0 & \text{ on } \p \W.
		\end{cases}
	\]
	
	Let us decompose $f = f_1 + f_2$, where $f_2 = u_\W \int u_\W f$, and then solve for the potential function
	\[
		\begin{cases}
			- \Lap p = \l_1(\W) p + f_1 & \text{ on } \W\\
			\int p u_\W = 0 & \\
			p = 0 & \text{ on } \p \W.
		\end{cases}
	\]
	This admits a unique solution via the Fredholm alternative (using $\int f_1 u_\W = 0$) which satisfies $\|p\|_{L^2} \leq C \|f_1\|_{L^2} \leq C \|f\|_{C^{0, \a}}$. From elliptic regularity, we have that 
	\[
		\|p\|_{C^{0, \a}(\W)} \leq C [\|f_1\|_{L^\infty} + \|p\|_{L^2}] \leq C \|f\|_{C^{0, \a}}.
	\]
	
	Now, multiply the equation for $\bar{q}^x$ by $p$ and integrate:
	\[
		\int \ei(\W) \bar{h}^xp = \int p (- \Lap - \ei(\W)) \bar{q}^x = \int f_1 \bar{q}^x.
	\]
	The other term on the left vanishes as $\int p u_\W = 0$, while the simplification on the right is from Green's identity and the equation for $p$.
	
	On the other hand,
	\[
		\int f_2 \bar{q}^x = \left(\int u_\W f\right)\left(\int u_\W \bar{q}^x\right) = - \frac{1}{\l_1(\W)} \int u_\W f.
	\]
	Combining, we have that
	\[
		\int f \bar{q}^x = \ei(\W) \int \bar{h}^xp - \frac{1}{\ei(\W)} \int u_\W f,
	\]
	which is bounded in $C^{1, \a}$ by Lemma \ref{l:hhighreg} applied to $p$.
\end{proof}

\begin{corollary}\label{cor: main reg}
	Let $\W$ be a minimizer of $\Ep$. Assume that for some $x_0\in \p \W$, $r < r_0$, and $\d > 0$ small enough we have $B_r(x_0)\ss \Qb_R$ and $|\W \triangle B_{r, \nu}(x_0) | < \d |B_r(x)|$. Then for each $\a \in (0, 1)$,  $\p \W \cap B_{r/64}(x_0)$ is given by a $C^{2, \a}$ normal graph over $\p B_{r, \nu}(x_0) \cap B_r(x_0)$, with the $C^{2, \a}$ norm bounded by a constant depending only on $R, v, \vmax, \vpar$, and $\a$.
\end{corollary}

\begin{proof}
	From Lemma \ref{l:ds}, $\p \W\cap B_{r/2}(x_0)$ is a $C^{1, \a_0}$ graph over $\{e: e \cdot \nu = 0 \}$ in normal coordinates, where $\a_0 > 0$ is the fixed exponent from that lemma. Applying Lemmas \ref{l:hhighreg} and \ref{l:qhighreg} with $f = a_\W \in C^{0, 1}$, we have that
	\[
		\left\|\frac{v^x_\W}{|\grad u(y)|^2}\right\|_{C^{1, \a_0}(\p \W \cap B_{r/4}(x_0))} \leq C.
	\] 
	Applying Proposition \ref{prop:higherbh} to $w_\W$ gives that
	\[
		\left\|\frac{|\grad w_\W|^2}{|\grad u_\W|^2}\right\|_{C^{1, \a_0}(\p \W \cap B_{r/4}(x_0))} \leq C.
	\]
	We also have $\|b_\W\|_{C^{1, 1}} \leq 1$ by assumption \ref{a:nlc1}. Proceeding as in Theorem \ref{thm:usefulEL}, we may write the Euler-Lagrange equation for $u$ as
	\[
		|\n u(x)|^2 = \r(x),
	\]
	where $c \leq \r \leq C$ and $ \r \in C^{1, \a_0}(\p \W \cap B_{r/4}(x_0))$. Now apply {\cite[Appendix]{KL18}} or \cite{DSFS19} to obtain that $\p \W$ may be {parametrized} as a $C^{2, \a_0}$ graph on $\p \W \cap B_{r/8}(x_0)$.
	
	Now repeat this argument on $B_{r/8}(x)$, except using that $\p \W$ is $C^{2, \a_0} \ss C^{1, \a}$ rather than $C^{1, \a_0}$ to recover the conclusion as stated.
\end{proof}

Combining this with Theorem \ref{thm:C1a} gives that all of $\p \W \cc \Qb_R$ is $C^{2, \a}$, so long as all sets in $\Min$ are smooth:

\begin{corollary}\label{c:higherreg}
	Let $\W$ be as in Theorem \ref{thm:C1a} and assume that $\Min$ is uniformly $C^{4}$. Then $\p \W \cap \{ x \in \Qb_R: d(x, \p \Qb_R) > r_1 \}$ may be parametrized as a $C^{2, \a}$ normal graph over $\p U$ for any $\a < 1$, with $C^{2, \a}$ norm bounded by $r_1$.
\end{corollary}

\appendix

\section{Admissible nonlinearities}\label{app A}
In Section~\ref{ss:mainprob}, we gave several examples of  nonlinearities $\nl$ used in the definition of the main functional $\Ep$. In this appendix, we verify that these examples are admissible nonlinearities in the sense of Definition~\ref{def: nl}. We also define the notion of a set center, which  generalizes the notion of the barycenter of a set in Euclidean space.

To begin, it will be useful to show that a constant multiple of  $d_*(\W, U)^2$ defined in \eqref{eqn: dstar two sets} is itself an admissible nonlinearity. Given any bounded open set $U$ with $C^2$ boundary, define the functional
\begin{equation}\label{e:nlexample}
\nl_U(\W) = \frac{d_*(\W, U)^2}{C_1}  = \frac{1}{C_1}\left[ \int_\W \psi_{U} - \int_{U} \psi_U + \int |u_\W - u_U|^2\right]
\end{equation}
on bounded open sets $\W$ with $|\W|\leq \vmax$ and with $\En(\W) \leq \minE + \err_0$; recalling Remark~\ref{rmk: unique eigenfunction}, $\nl_U$ is well-defined on this class of sets. 
\begin{lemma}\label{l:A1} For $C_1$ chosen sufficiently large depending on $U$ and $\vmax$, 
	the functional $\nl_U(\W)$ is an admissible nonlinearity with respect to the trivial subgroup of the isometry group. 
\end{lemma}
\begin{proof}
	 Choose $C_1 \geq \vmax |\max f| + 4$ to guarantee that  $\nl_U(\W) \in [0,1]$ for any $\W$ so \ref{a:nlbdd} is satisfied. To see the upper bound, recall that $\int u_\W^2 = \int u_U^2 = 1$; the lower bound follows because $d_*(U, \W)\geq 0$. Since $G_0$ is trivial, \ref{a:nlinv} holds automatically.
Checking \ref{a:nllip} is also straightforward:
\begin{equation}\label{e:nlexamplep3}
|\nl_U(\W) - \nl_U(\W')|\leq \frac{1}{C_1}\left[ |\W \triangle \W'|\max |f|  + \int 2 |u_{\W'} - 2 u_\W||u_U| \right],
\end{equation}
so as long as $C_1 \geq 2\max u_U$, which is bounded by Lemma \ref{l:efbdd}, and so \ref{a:nllip} follows. 	 For \ref{a:nlc1}, we set $\W_t = \phi_t(\W)$ and perform the same computation more carefully to give
\begin{equation}\label{e:nlexamplep4}
\nl_U(\W_t) - \nl_U(\W) = \frac{1}{C_1}\left[ \int_{\W_t} \psi_U - \int_\W \psi_U + 2 \int (u_{\W} - u_{\W_t}) u_U\right].
\end{equation}
Set $a_\W = - \frac{2}{C_1} u_U$, $b_\W = \frac{1}{C_1} \psi_U$. After possibly choosing $C_1$ larger, we see that  \ref{a:nlc1} is  satisfied, using elliptic regularity and the smoothness of $U$ to bound $\grad u_U$. 
\end{proof}
Notice that  in this example, while $b_\W$ here is smooth, $a_\W$ is not better than Lipschitz; is the limiting factor for the higher regularity discussed in Section \ref{s:highreg}.

Example~\ref{ex: one min} considered a nonlinearity $\nl$ in the case when there is a unique (regular) minimizer $U$ of the base energy.  Up to multiplication by a constant $C=C(U)$, it follows that the nonlinearity in Example~\ref{ex: one min} is admissible with respect to the trivial subgroup of the isometry group for any choice of $c \in (0,1]$ in \eqref{eqn: SP nl}. Indeed, the functional $\nl = \nl_c$ of Example~\ref{ex: one min} (normalized by a constant $C=C(U)$ independent of $c$) is given by $\phi_c \circ \nl_U$ where  $\phi_c : [0, 1] \rightarrow [0, 1]$ is a smooth function with $|\phi_c'|\leq 1$ for all $c \in (0,1]$, and is thus admissible by Remark~\ref{rmk: admissible compositions}.

We now move toward verifying that the nonlinearity of Example~\ref{ex: general example} is admissible. To this end, let us introduce the notion of a set center. 

\begin{definition}[Set centers]\label{def: set center}
Let $\Isom_0 \leq \Isom$ be a closed subgroup of isometries of $M$, $\e > 0$, and $U$ a fixed bounded open set. We say that a mapping from the class of bounded, open, nonempty subsets $E$ of $M$ with
\begin{equation}\label{e:closetou}
	\inf_{e \in \Isom_0} |E \triangle e(U)| \leq \e
\end{equation}
to points $x_E$ in a complete Riemannian manifold $(N, g_N)$ is a \emph{set center adapted to} $U$ if it satisfies the following properties:
\begin{enumerate}[({C}1)]
	\item \label{a:clip} If $E, E' \ss \Qb_R$, then $d_N(x_E, x_{E'}) \leq C(R) |E \triangle E'|$.
	\item \label{a:conto} For every $E$, there is an $e \in \Isom_0$ such that $x_E = x_{e(U)}$.
	\item \label{a:cproj} For any $e, e' \in \Isom_0$, $|e(U)\triangle e'(U)| \leq C(E) d_N(x_{e(E)}, x_{e'(E)})$.
	\item \label{a:cc1} For $x \in M$ and $r \leq r_0$, let $\phi_t$ be a one-parameter family of diffeomorphisms with $\phi_0(x) = x$ and $|\p_t \phi_t|\leq 1$ such that $\{\phi_t(x) \neq x \}\ss B_r(x_0)$. If the sets $\phi_t(\W)$ satisfy \eqref{e:closetou} and  $\cH^{n-1}(\p \W) < \infty$, then $t \mapsto x_{\phi_t(\W)}$ defines  a $C^1$ curve in $N$ and for any tangent vector $v\in T_{x_\W}N$,
	\begin{equation}\label{e:c1setcenter}
		\limsup \frac{1}{|t|}\left| t g_N\left(x'_{\phi_t(\W)}|_{t = 0}, v\right) - \int_{\phi_t(\W)} a_\W^v + \int_{\W}a_\W^v\right| \leq C r^n
	\end{equation}
	 for a function $a_\W^v$ independent of $\phi_t$ and depending linearly on $v$ with $\|a_\W^v\|_{C^2(\Qb_R)} \leq C(R)|v|$. 
\end{enumerate}
\end{definition}
Given a bounded open set $E$ we let $U_E =e(U)$ denote image under the unique isometry $e\in G_0$ of $U$ such that $E$ and $U_E$ have the same set center. 

Property \ref{a:clip} implies $x_E$ is a kind of Lipschitz mapping from sets to the space of centers, while \ref{a:cproj} suggests it acts like a projection onto the images of $U$ under $\Isom_0$ action. Property \ref{a:cc1} implies that it is a $C^1$-regular projection in an appropriate sense: to understand \eqref{e:c1setcenter}, note that
\[
	\int_{\phi_t(\W)} a_\W^v - \int_{\W}a_\W^v = t \int_{\p^* \W} a_\W^v g(\phi'_0, \nu_x) d\cH^{n-1} - t \int_\W g(\phi'_0, \grad a_\W^v) + o(t).
\]
The second term is bounded by $C r^n t$, while the first captures the ``leading-order'' change in $\W$ under $t \mapsto \phi_t(\W)$. Thus \eqref{e:c1setcenter} ensures that the $v$ component of the derivative of $x_{\phi_t(\W)}$ exists and is proportional to $\int_{\p^* \W} a_\W^v g(\phi'_0, \nu_x) d\cH^{n-1}$ for some sufficiently smooth  function $a_\W^v$, up to error of size $r^n$. This type of error is lower-order when considering localized deformations near a point on $\p \W$.

Note that a set center may map to points in a Riemannian manifold $(N,g_N)$ that is not $(M,g)$. An example to keep in mind here is if $(M,g)$ is a cylinder $S^{n-1}\times \R$ equipped with the standard product metric, a set center can be defined by choosing the Euclidean barycenter of a set when $(M,g)$ is embedded into $\R^{n+1}$ in the standard way.

Let us give some examples of set centers.

\begin{example}{\rm
	If $\Isom_0$ is trivial, the constant mapping $x_E = x_0$ is a set center adapted to any $U$.
	}
\end{example}
 
 \begin{example}{\rm Suppose $(M,g)$ is simply connected and has nonpositive sectional curvature. The barycenter $x \mapsto \text{argmin}_x \int_E d^2(x, y)dm(y)$ exists and is unique. This is a set center adapted to any $U$ so long as (1) $\Isom_0$ acts transitively on $M$ and (2) if $e \in \Isom_0$ fixes $x_U$, it fixes $U$. In particular, this applies to $\R^n$ or hyperbolic space, $\Isom_0 = \Isom$, and $U = B_r(x)$ a ball. It also applies to $\R^n$ with $U$ an arbitrary (nice) open set and $\Isom_0$ the group of translations.
 	}
 \end{example}
 \begin{example}\label{proof: set center on sphere}{\rm 
 	Let $(M,g)$ be the round sphere, embedded in the standard way in $\R^{n+1}$. Let $G_0=G$ be the full isometry group and $U= B_r(x) $ be a geodesic ball with $r< \pi$. For any open $E \ss S^{n}$, let $
		y_E = \fint_{E} y d \cH^{n}(y),$
	where $y, y_E \in \R^{n+1}$ and this is a (vector-valued) surface integral. Note that $y_U \neq 0$ 
	 and that for any $E, E'$, we have
	\[
		|y_E - y_{E'}| \leq \frac{1}{|E|}\int_{E \triangle E'}|y| + |S^{n}|\left|\frac{1}{|E|} - \frac{1}{|E'|}\right| \leq C |E \triangle E'|.
	\]
	In particular, if $|E \triangle U| \leq \e$, we may ensure that $|y_E|\geq r > 0$. For any such $E$, the mapping $E\mapsto x_E = \frac{y_E}{|y_E|}$ is a set center. Indeed, then $|x_E - x_{E'}|\leq \frac{1}{r}|y_E - y_{E'}|$, so this satisfies \ref{a:clip}. It also satisfies \ref{a:conto}, as $\Isom$ acts transitively on $S^{n}$. Property \ref{a:cproj} can be checked using that $U=B_r(x)$ is invariant under rotation about its set center, and \ref{a:cc1} may be verified similarly to the proof of Proposition~\ref{l: set center euclidean} below.
	}
 \end{example}

Let us give a proof that the Euclidean barycenter is a set center in the sense of Definition~\ref{def: set center};  the other examples described above can be checked similarly.
First of all, the barycenter on Euclidean space is an example of a set center.
\begin{proposition}\label{l: set center euclidean}
	Let $(M,g)$ be Euclidean space, and let  $\Isom_0 = \Isom$ and $U = B_1$. Then the barycenter $x_E = \fint_E x \in \R^n$
	is a set center satisfying \ref{a:clip}-\ref{a:cc1}. Here $\e$ in \eqref{e:closetou} may be taken arbitrary.
\end{proposition}

\begin{proof}
	For \ref{a:clip},
	\[
		|x_E - x_{E'}| = \left|\frac{1}{|E|}\int_E x - \frac{1}{|E'|} \int_{E'}x\right|\leq \frac{1}{|E|}\int_{E \triangle E'}|x| + \left|\frac{|E'|}{|E|} - 1\right| R \leq C |E' \triangle E|.
	\]
	Then \ref{a:conto}, \ref{a:cproj} are immediate, and for \ref{a:cc1} we may compute
	$
		|\W_t| - |\W| = \cH^{n-1}(\p^*\W)t + o(t) = O(t)
	$ and
	\[
		\int_{\phi_t(\W)}x - \int_{\W} x = t \int_{\p^* \W} x \cdot \nu_x - t \int \phi'_t(0) + o(t) = O(t),
	\]
	from which we see that
	\[
		x_{\W_t} = \frac{1}{|\phi_t(\W)|}\int_{\phi_t(\W)} x = x_\W + \int_{\phi_t(\W)} \left[\frac{x}{|\W|} - \int_\W x\right] - \int_{\W} \left[\frac{x}{|\W|} - \int_{\W} x\right] + O(t^2).
	\]
	This shows $x_{\W_t}$ is a $C^1$ curve, and $x'_{\W_t}|_{t=0}$ satisfies \eqref{e:c1setcenter} with $a^v_\W = \left[\frac{x}{|\W|} - \int_\W x\right] \cdot v$.
\end{proof}

More general constructions can be carried out here, for example considering set projections onto $\Min$ in the case when $\Min$ admits ``smooth'' parameterizations by finite-dimensional manifolds, in which case the key point will always be verifying properties \ref{a:nllip} and \ref{a:nlc1} using analogues of \ref{a:clip} and \ref{a:cc1}.

Now, given a bounded open set $U$ with $C^2$ boundary, a subgroup of isometries $G_0 \leq G$, and  a set center $\W\mapsto x_\W$, define the functional
\begin{equation}\label{e:nlexample2}
	\nl(\W) := \nl_{U_\W}(\W) \qquad x_{U_\W} = x_\W
\end{equation}
where $\nl_U$ was defined in \eqref{e:nlexample}.
By properties \ref{a:conto} and \ref{a:cproj} of set centers, $\nl$ is well-defined for all $\W$ with $|\W \triangle e(U)|\leq \e$ for some $e\in \Isom_0$. From Lemma \ref{l:emin}, this holds for all $\W$ with $\En(\W) \leq \minE + \err_0$. 
 
\begin{proposition}\label{p:setcentertonl} The functional $\nl(\W)$ in \eqref{e:nlexample2} is an admissible nonlinearity with respect to $\Isom_0 $ in the sense of Definition~\ref{def: nl} for $C_1$ sufficiently large.
\end{proposition}

\begin{proof}[Proof of Proposition \ref{p:setcentertonl}.]
	Properties \ref{a:nlbdd} and \ref{a:nlinv} are immediate from the construction. For \ref{a:nllip}, take $\W, \W'$ and apply property \ref{a:clip} of set centers:
	\[
		d_N(x_\W, x_{\W'}) \leq C(R) |\W \triangle \W'|.
	\]
	This means there exists $e \in \Isom_0$ with $e(U_{\W}) = U_{\W'}$ and $|U_{\W'}\triangle U_\W|\leq C |\W \triangle \W'|$. To simplify notation, we set $U = U_\W$ and $U' = U_{\W'}$. 
	
	First of all, after possibly composing $e$ with an isometry which fixes $U$, we may assume $d(e, \text{id})\leq \e_1$ small. Indeed, if this is false then there is a sequence $e_k \in \Isom_0$ with $|U \triangle e_k(U)| \rightarrow 0$ but $d(e_k, e_0) > \e_1$ for any $e_0$ fixing $U$. It is then straightforward to show (see \cite{MS39}, Theorem 3 and subsequent remarks) that the $e_k$ have a subsequence converging to some $e$ with $|U \triangle e(U)| = 0$ and therefore fixing $U$. This is a contradiction.
	
	Let $S_U \ss \mathfrak{G} = T_{\text{id}}\Isom_0$  be the tangent space to the subgroup of isometries which fix $U$, and $H_U$ a subspace of $\mathfrak{G}$ such that $H_U \oplus S_U = \mathfrak{G}$; set $V = \{v \in H_U : |v| = 1 \}$. One may verify that that for any $U' = \exp_{\text{id}} w (U)$ for $w$ with $|w|\leq \e_1$, $U'$ may also be represented as $\exp_{\text{id}} t v(U)$ for $v \in V$ and $|t| \leq C \e_1$ (see e.g. \cite[Theorem 3.58]{W83}). We may estimate the value of $t$ more precisely:
	\begin{equation}\label{e:delleqt}
		|U \triangle U'| \leq C \sup_{x \in U} d(\exp_{\text{id}} tv(x), x) \leq C(U) t.
	\end{equation}
	The reverse bound is also valid. Indeed, for any fixed $v \in V$, we have $
		\int_{\p U} |g(v(x), \nu_x)| > 0$
	where $\nu_x$ is the outward unit normal to $\p U$: if $|g(v(x), \nu_x)| \equiv 0$, then the isometries generated by $v$ fix $U$, contradicting that $v \in V \ss H_U \sm \{0\}$. As $V$ is compact, this gives that $
		\inf_{v \in V} \int_{\p U} |g(v(x), \nu_x)| \geq c(U) > 0.$
	Then we may estimate
	\begin{equation}\label{e:delgeqt}
		|U \triangle U'| = t \int_{\p U} |g(v(x), \nu_x)| + O(t^2) \geq c(U) t.
	\end{equation}
	
	This value $t$ also controls the distance between $U$ and $U'$ in stronger topologies:
	\[
		\sup_{x \in \Qb_R} |d(x, \p U) - d(x, \p U')| = \sup_{x \in \Qb_R} | d(x, \p U) - d(\exp_{\text{id}} -tv(x), \p U)| \leq C t
	\]
	and 
	\[
		\int |u_U - u_{U'}| \leq \sup_x |u_U(x) - u_U(\exp_{\text{id}} -tv(x))| \leq C t\sup_x |\grad u_U| \leq C t.
	\]
	In the line above, we recall from Lemma \ref{l:simpleeval} that $U$ has a \emph{unique} nonnegative, normalized, first eigenfunction $u_U$, and so $u_{U'} = u_U \circ e^{-1}$.
	So, we find that
	\[
		\left|\nl_{U}(\W) - \nl_{U'}(\W)\right| \leq \frac{1}{C_1}\left[\int_{\W}|\psi_U - \psi_{U'}| + 2\int |u_\W||u_U - u_{U'}|  \right] \leq \frac{1}{C_1} Ct \leq \frac{C}{C_1} |\W \triangle \W'|.
	\]
	Combining this with \eqref{e:nlexamplep3} of Lemma~\ref{l:A1} and choosing $C_1$ small enough gives \ref{a:nllip}.
	
	We now  verify \ref{a:nlc1}. To begin, let $U_t = e_t(U)$ denote $ U_{\phi_t(\W)}$, the unique isometry of $U$ with the same set center as $\phi_t(\W)$. Our first claim is that after possibly modifying the $e_t$, they may be represented for $t$ small as $e_t = \exp_{\text{id}} \g(t)$, where $\g(t)$ is a $C^1$ curve in $H_U\ss \mathfrak{G}$ with $\g(0) = 0$. To see this, consider the mapping $T: v \mapsto x_{\exp_{\text{id}} v (U)}$ from a small ball $B_{\e_1}(0) \ss \mathfrak{G}$ to $N$. By property \ref{a:cc1} of set centers, $T$ has continuous directional derivatives on $B_{\e_1}$, and so it is $C^1$. We also have that $d_0T(S_U) = \{0\}$, since any curve of isometries fixing $U$ has set center constant $x_U$ and  by \eqref{e:delgeqt} and property \ref{a:cproj}, this is precisely the kernel of $d_0T$. So, $T|_{H_U}$ is a $C^1$ diffeomorphism onto its image from the inverse function theorem.
	
	Applying property \ref{a:cc1}, the curve $t\mapsto x_{\phi_t(\W)}$ is $C^1$, and there is a unique $C^1$ curve $\g : (-c, c) \rightarrow H_U \ss \mathfrak{G}$ such that $x_{\phi_t(\W)} = \exp_{\text{id}} \g(t)(U)$. By definition of $U$, $\gamma(0) = 0$. Furthermore, we have that
	\[
		\g'(0) = d_0T^{-1}|_{H_U}\big(x'_{\phi_t(\W)}|_{t = 0}\big) = \lim_{t \rightarrow 0}\frac{1}{|t|} \int_{\phi_t(\W)} A_\W - \int_{\W} A_\W + O(r^n)
	\]
	for some $A_\W : \Qb_R \rightarrow \mathfrak{G}$ with $\|A_\W\|_{C^2} \leq C$, i.e. for each $x$, $A_\W(x)$ is a vector field generating isometries.
		
	Let us next estimate the $\psi_U$ term:
	\[
		\psi_{U_t}(x) = \psi_{U}(\exp_{\text{id}} -\g(t)(x)) 
		= \psi_U(x) - t g(\grad \psi_U(x), \g'(0)(x)) + o(t) |D^2 \psi_U|.
	\]
	so
	\[
		\int_\W [\psi_{U_t}(\W) - \psi_U(\W)] = - \int_\W t g_x(\grad \psi_U(x), \g'(0)) + o(t) + O(r^nt).
	\]
	In particular,
	\[
		\limsup_{t \rightarrow 0}\frac{1}{|t|} \left|\int_\W \psi_{U_t}(\W) - \psi_U(\W) + \int_{\phi_t(\W)} \int_{\W} g(A_\W(y), \grad \psi_U(x))dx dy - \int_{\W} \int_{\W} g(A_\W(y), \grad \psi_U(x))dx dy\right| \leq C r^n.
	\]
	Here and in the remainder of the proof we write $dx$ to denote $dm(x)$ to consolidate notation.
	The estimate on the eigenfunction term is similar, at least so long as we avoid the set $\{ d(x, \p U) < C|t| \}$ for $C$ large enough that $\exp_{\text{id}} - \g(s)(x) \neq \p U$ for $|s|\leq |t|$, the issue being that $u_U$ is not smooth near $\p U$. Away from this set, we have:
	\[
		u_{U_t}(x) = u_U(x) - t g(\grad u_U(x), \g'(0)) + O(t^2 \sup_U|D^2 u_U|).
	\]
	Note that while $u_U$ is not smooth, $D^2 u_U$ is uniformly bounded on $U$. Therefore,
	\begin{align*}
		\int |u_{U_t} - u_\W|^2& - \int |u_{U} - u_\W|^2 = 2 \int u_\W (u_U - u_{U_t})\\
		 &= 2 t \int g(\grad u_U(x), \g'(0))  + o(t) + O\bigg(\int_{\{ d(x, \p U) < C|t| \}} |u_U| + |u_{U_t}| + t|g(\grad u_U(x), \g'(0))| \bigg).
	\end{align*}
	The rightmost term is actually $o(t)$ as well; $|\{ d(x, \p U) < C|t| \}| \rightarrow 0$, while  also
	$
		|u_U| + |u_{U_t}|\leq C t$ and $ t|g(\grad u_U(x), \g'(0))| \leq C t
	$
	over this region. So, rewriting the expansion above, we have 
	\[
		\limsup_{t\rightarrow 0}\frac{1}{|t|}\left|\int |u_{U_t} - u_\W|^2 - \int |u_{U} - u_\W|^2 - 2 t \int g(\grad u_U(x), \g'(0))\right|  = 0.
	\]
	Moreover, the first-order term may be re-expressed in terms of $A_\W$ as well:
	\[
		\limsup_{t\rightarrow 0}\frac{1}{|t|}	\left|t \int g(\grad u_U(x), \g'(0)) - \int_{\phi_t(\W)} \int_{\W} g_x(A_\W(y), \grad u_U(x))dx dy - \int_{\W} \int_{\W} g_x(A_\W(y), \grad u_U(x))dx dy\right| \leq C r^n.
	\]
	
	To summarize, setting
	\[
		a_\W^0 = \int_{\W} g_x(A_\W(y), 2 \grad u_U(x) + \grad \psi_U(x))dx
	\]
	we have that $\|a_\W^0\|_{C^2(\Qb_R)} \leq C(U)$, and
	\[
		\limsup_{t\rightarrow 0}\frac{1}{|t|} \left|\nl_{U_t}(\W) - \nl_{U}(\W) + \int_{\phi_t(\W)} a_\W^0 - \int_{\W}a_\W^0\right| \leq Cr^n.
	\]
	Together with \eqref{e:nlexamplep4} this implies property \ref{a:nlc1}. This completes the proof.
\end{proof}
As we saw following Lemma~\ref{l:A1} above, it follows immediately from Proposition~\ref{p:setcentertonl} and Remark~\ref{rmk: admissible compositions} that the nonlinearity of Example~\ref{ex: general example} is an admissible nonlinearity  with respect to $\Isom$ in the sense of Defintion~\ref{def: nl} for all $c \in (0,1]$, after possibly normalizing by a constant $C=C(U)$.

\section{Domains admitting linear-growth solutions to elliptic equations are NTA}\label{s:appendixnta}

The purpose of this appendix is to show that if a domain $\W$ supports a nonnegative function $u$ vanishing on $\p \W$, growing like $d(x, \p \W)$ from the boundary, and solving an elliptic equation with bounded right-hand side, then $\W$ is NTA (locally). Our first lemma follows the approach of \cite{ACS87}, using a monotonicity formula argument, except that the $u$ is not assumed to be a minimizing solution to a free boundary problem. In fact, it may solve an equation with right-hand side, and not even all the way up to the boundary (so long as we are only looking for Harnack chains between points a distance $\eta$ from $\p \W$).

\begin{lemma}\label{l:appendlap}
	Let $\W \ss \R^n$ be open and $u$ be a continuous function on $\bar{\W} \cap B_1$ (which is $C^2$ on $\W$), with $u = 0$ on $\p \W$ and $u > 0$ on $\W$. Then for any $K, M, m > 0$, there are $\eta, \d > 0$ such that if
	\begin{enumerate}
		\item $0 \in \partial \W$.
		\item $|\nabla u|\leq 1$ on $\W \cap B_1$.
		\item For all $x \in \bar{\W} \cap B_1$ and $r < 1$, $\sup_{B_r(x)} u \geq m r$.
		\item $|\Lap u| \leq \d$ on $\{ d(x, \W^c) > \d \}\cap B_1$.
		\item $\W$ satisfies the outer clean ball condition with constant $K$ at $0$.
	\end{enumerate}
	then for any points $x_1, x_2 \in B_{M \eta}$ with $d(x_1, \W^c), d(x_2, \W^c) > \eta$, there exists a curve $\gamma \ss B_1 \cap \{d(x, \W^c) > c \eta \}$ of length at most $C$ connecting $x_1$ and $x_2$.
\end{lemma}

\begin{proof}
	Set $v(x) = u(x) + \frac{\d}{2n}|x|^2$; then $\Lap v \geq 0$ on $\{ d(x, \W^c) > \d \}\cap B_1$, and moreover as long as $\d < 1$ assumption (2) gives $|\nabla v|\leq 2$. Consider the set $A = \{x \in \W :  v(x) > c_0 \eta \} \cap B_{1/2}$, noting that so long as we choose $\d$ small enough in terms of $\eta, m$,
	\[
		\Big\{d(x, \W^c) > \frac{1}{2}\eta \Big\} \cap B_{1/2} \ss A \ss \{d(x, \W^c) > c_1 \eta \},
	\]
	for some $c_0, c_1$ depending only on $n$ and $m$.	Indeed, if $d(x, \W^c) > \frac{1}{2}\eta$, we apply assumption (3) on $B_{\frac{ \eta}{8}}(x)$ to learn that $\sup_{B_{\frac{ \eta}{8}}(x)}u \geq m \frac{\eta}{8}$. Then so long as $\d < \eta/4$ (and hence (4) applies on $B_{ \eta/4}(x)$) the Harnack inequality gives
	\[
		m \frac{\eta}{8} \leq \sup_{B_{\frac{\eta}{8}}(x)}u  \leq C u(x).
	\]
	Setting $c_0 = \frac{m}{16 C}$ ensures that $u(x) \geq 2 c_0 \eta$, and as long as $\d < c_0 \eta$, this gives $v(x) > u(x) - \d \geq c_0 \eta$ and so $x\in A$. On the other hand, if $x \in A$ and $\d < \frac{c_0}{2} \eta$, we have $u(x) > \frac{c_0}{2} \eta$. In light of (2), this guarantees that $d(x, \W^c) > \frac{c_0}{2} \eta$, and we set $c_1 = \frac{c_0}{2}$.
	
	Let $A_1, A_2$ be the two connected components of $A$ which contain $x_1, x_2$ respectively. We will show that $A_1 = A_2$. From a computation like the one just performed, we have that $v(x_i) > 2 c_0 \eta$. Set $d_i = d(x_i, A^c)$, and let $y_i \in \partial A_i$ such that $| y_i - x_i| = d_i$. From assumption (1) we must have $d_i < M \eta$. Our first task will be to show that $|\grad v|$ is relatively large on a set near $x_i$.
	
	Consider a line segment with endpoints $x_1$ and $y_1$. Integrating along it,
	\begin{align*}
		c_0 \eta < v(x_1) - v(y_1)
		& = \int_0^1 \nabla v(t x_1 + (1-t)y_1) \cdot (x_1 - y_1) dt \\
		&\leq d_1 \int_0^1 |\nabla v(t x_1 + (1-t)y_1)| dt\leq  M \eta \int_0^1 |\nabla v(t x_1 + (1-t)y_1)| dt.
	\end{align*}
	As $|\grad v| \leq 2$ everywhere, the second-to-last step implies that $d_1 \geq \frac{c_0 \eta}{2}$. We can say more: if $|\nabla v| < \frac{c_0}{8 M}$ for all $t > \frac{c_0}{8 M}$, we could estimate
	\[
	M \eta \int_0^1 |\nabla v(t x_1 + (1-t)y_1)| dt \leq (1 - \frac{c_0}{8 M}) \frac{c_0 \eta}{8} + \frac{c_0 \eta}{8} 2 < \frac{c_0 \eta}{2},
	\]
	and this is a contradiction. Hence there must be a $t > \frac{c_0}{8M}$ with $|\nabla v(t x_1 + (1-t)y_1)|\geq \frac{c_0}{8M}$; set $z_1 = t x_1 + (1-t)y_1$. Letting $c_2 =\frac{c_0}{8M} \cdot \frac{c_0}{8}$, we have that 
	\[
		|z_1 - y_1| = t|x_1 - y_1| = t d_1 \geq \frac{c_0}{8M} \cdot \frac{c_0 \eta}{2} = c_2 \eta.
	\]
	In particular $B_{c_2 \eta}(z_1) \ss A_1 \ss \{ d(x, \W^c) > \d \}$ and we may apply an elliptic regularity estimate to $v$ on this ball:
	\[
		(c_2 \eta)^{1 + \a} [\nabla v]_{C^{0, \a}(B_{\frac{c_2 \eta}{2}}(z_1))} \leq C [\osc_{B_{c_2 \eta}(z_1)} v + \d (c_2 \eta)^{2}] \leq C \eta,
	\]
	so in particular
	\[
		|\nabla v(x) - \nabla v(z_1)| \leq C \1\frac{|x - z_1|}{\eta}\2^\a < \frac{c_0}{16M}
	\]
	so long as $|x - z_1| < c_3 \eta$. For any such $x$, we have $|\nabla v(x)|\geq \frac{c_0}{16M}$. Integrating,
	\[
		\frac{1}{(2 M \eta)^2} \int_{B_{2M\eta} \cap A_1} |\nabla v|^2 |x|^{2-n} dx \geq \frac{1}{(2 M \eta)^{n}} \int_{B_{c_3 \eta}(z_1)} |\nabla v|^2 \geq \1\frac{c_3}{2M}\2^n |B_1| \frac{c_0}{16M} \geq c_4.
	\]
	
	Set
	\[
		J(r) = \frac{1}{r^4} \int_{B_{r} \cap A_1} |\nabla (v - c_0\eta)_+|^2 |x|^{2-n} dx \int_{B_{r} \cap A_2} |\nabla (v - c_0\eta)_+|^2 |x|^{2-n} dx.
	\]
	We have just shown that $J(2 M \eta) \geq c_4^2$ (as the computations apply equally well to $x_2$ and $A_2$). It is straightforward to check that $J(1/2) \leq C$ using just $|\nabla v|\leq 2$. From the enhanced monotonicity property of $J$ \cite[Lemma 4.4]{ACS87} and assumption (5),
	$
		r \mapsto \frac{J(r)}{r^\b}
	$
	is an increasing function for some $\b > 0$ depending only on $K$. This means that
	$
		c_4^2 \leq J(2 M \eta) \leq C (4 M \eta)^{\beta}.
	$
	This is a contradiction if $\eta$ is chosen small in terms of $M$ and $c_4$.
	
	We have shown that $A_1 = A_2$. Now cover the closure of $A_1$ by balls of radius $r = \frac{c_1}{4}\eta$; it is always possible to do so by at most $N = C \eta^{-n}$ balls, and then to connect any two points in $A_1$ by a piecewise linear curve of length $C N \eta$ contained in the union of these balls. This curve has bounded length and stays a distance at least $c_1\eta/4$ away from $\partial \W$, concluding the proof.
\end{proof}

If instead $u$ solves an elliptic equation with Lipschitz coefficients, the previous lemma still applies after a suitable change of variables and scaling argument. Here we carefully exploit the form of assumption (4) above.

\begin{lemma}\label{l:appendPDE}
	Let $\W \ss \R^n$ be open and $u$ be a continuous function on $\bar{\W}\cap B_1$ (which is $C^2$ on $\W$), with $u = 0$ on $\p \W$ and $u > 0$ on $\W$. Then for any $\l, K, M, m > 0$, there is an $\eta_0 > 0$ such that if
	\begin{enumerate}
		\item $0 \in \partial \W$
		\item $|\nabla u|\leq 1$ on $\W \cap B_1$.
		\item For all $x \in \bar{\W} \cap B_1$ and $r < 1$, $\sup_{B_r(x)} u \geq m r$.
		\item $a_{ij}\p_i\p_j u = f $ on $\W \cap B_1$, with $\|f\|_{C^{0, 1}} \leq 1$, where $\l |\xi|^2 \leq a_{ij}(x) \xi_{i}\xi_j \leq \l^{-1} |\xi|^2$ is a matrix-valued function with $[a_{ij}]_{C^{0,1}(B_1)} \leq 1$.
		\item $\W$ satisfies the outer clean ball condition with constant $K$ at $0$.
	\end{enumerate}
	then for any $\eta < \eta_0$ and points $x_1, x_2 \in B_{M \eta}$ with $d(x_1, \W^c), d(x_2, \W^c) > \eta$, there exists a curve $\gamma \ss B_1 \cap \{d(x, \W^c) > c \eta \}$ of length at most $C \eta$ connecting $x_1$ and $x_2$.
\end{lemma}

\begin{proof}
	First, take $x \in \W$ and $d = d(x, \W^c)$: then
	\[
		d |D^2 u(x)| \leq C\left[ \sup_{B_{d/2}(x)} |\nabla u| + d^{2} [f]_{C^{0, 1}} \right]
	\]
	from standard elliptic estimates and assumption (4). If $B_{d/2}(x) \ss B_1$, combining with (2) gives $|D^2 u|\leq C d^{-1}$.
	
	Fix the matrix $b_{ij} = a_{ij}(0)$, and compute
	\[
		b_{ij} \p_i \p_j u = f + (a_{ij} - b_{ij}) \p_i \p_j u.
	\]
	On $B_r \cap \{d(x, \W^c) > \d_0 r \}$, this has
	\[
		|b_{ij} \p_i \p_j u| \leq 1 + \d_0^{-1}.
	\]
	There exists a linear map $L : \R^n \rightarrow \R^n$ such that $B_\k \ss L(B_1) \ss B_{\k^{-1}}$ and if $b_{ij} \p_i \p_j u = g$, then $\Lap (u \circ L) = g \circ L$, where $\k = \k (\l)$. Let $v$ be given by
	\[
		v(x) = \frac{u(r L x)}{\k^{-1} r},
	\]
	where $r$ will be chosen, but at least $r < \k$. Let $V = L^{-1} (\W)/r$. Then
	$
		|\nabla v|\leq 1$ on $ B_1 \cap V$
	and $0 \in \partial V$. Moreover, for every $x \in \bar{V} \cap B_1$ and $s < 1$,
	\[
		\sup_{B_s(x)} v \geq m \k^2 s.
	\]
	Finally,
	\[
		|\Lap v(x)| \leq C\k r (1 + \d_0^{-1})
	\]
	on $\{d(x, V^c) > \d_0/\k \} \cap B_1$.
	
	Fix $M' = M/\k$ and $m' = \k^2 m$, and apply Lemma \ref{l:appendlap} with parameters $M', m'$ to $v$ and $V$: then there exists an $\eta', \delta'$, such that if 
	\[
	|\Lap v(x)| \leq \d'
	\]
	on $\{d(x, V^c) > \d' \} \cap B_1$, and any $y_1, y_2 \in B_{M' \eta'}$ with $d(y_i, V^c) > \eta'$, there is a path connecting $y_1$ to $y_2$ of length at most $C$ and staying $c \eta'$ away from $V^c$.
	
	Select $\d_0 = \k \d'$, and $r = \k \eta/\eta'$. Choose $\eta_0$ so that $r < \k$ and 
	\[
		|\Lap v(x)|\leq C\k r (1 + \d_0^{-1}) \leq C \k^2 \frac{\eta_0}{\eta'}(1 + \d_0^{-1}) < \d'
	\]
	on $\{d(x, V^c) > \d_0/\k \} \cap B_1$. Then set $y_i = r L x_i$: we have that $y_i \in B_{M' \eta'}$ and $d(y_i, V^c) \geq \eta'$, so the above applies to give a path $\g'$ connecting $y_1$ to $y_2$. Let $\g = r^{-1} L^{-1} \g'$; then $\g$ connects $x_1$ and $x_2$, has length at most $C r \leq C \eta$, and stays a distance $c r\eta' \geq c \eta$ from $\p \W$.
\end{proof}

By iterating this lemma finitely many times, we can now show that $\W$ satisfies the Harnack chain condition.

\begin{theorem}\label{t:NTA}
	Let $\W \ss \R^n$ be an open set, and $u$ be a continuous function on $\bar{\W}\cap B_1$ (which is $C^2$ on $\W$), with $u = 0$ on $\p \W$ and $u > 0$ on $\W$. Assume that
	\begin{enumerate}
		\item $0 \in \p \W$.
		\item $|\nabla u|\leq 1$ on $\W$.
		\item For all $x \in \bar{\W} \cap B_1$ and $r < 1$, $\sup_{B_r(x)} u \geq m r$.
		\item $a_{ij}\p_i\p_j u = f $ on $\W \cap B_1$, with $\|f\|_{C^{0, 1}} \leq 1$, where $\l |\xi|^2 \leq a_{ij}(x) \xi_{i}\xi_j \leq \l^{-1} |\xi|^2$ is a matrix-valued function with $[a_{ij}]_{C^{0,1}} \leq 1$.
		\item $\W$ satisfies the inner and outer clean ball condition with constant $K$ at every point in $\p \W \cap B_1$.
	\end{enumerate}
	Then there is a $\d > 0$ such that $\W$ satisfies the Harnack chain condition with constant $C$ at any points $x_1, x_2 \in B_\d \cap \W$. The constants $\d, C$ depend only on $m, \l, K$ and $n$.
\end{theorem}

One may check that the \emph{inner} clean ball condition follows from assumptions (2-3). The outer clean ball condition, however, does not.

\begin{proof}
	Let $y_i$ be a point in $\p \W$ with $d_i = d(x_i, \p \W) = |x_i - y_i|$. Without loss of generality, assume $d_1 \leq d_2$. We will first show that the conclusion holds if $|x_1 - x_2| \leq 4K d_2$. Note that if $|x_1 - x_2| < \frac{1}{2} d_2$, we may connect $x_1$ and $x_2$ by a line segment and the conclusion is immediate.
	
	Construct a sequence of points $z_k$ as follows, using the clean inner ball property: $z_1 = x_1$, and then given $z_k$, $d(z_{k+1}, \p \W) \geq 2 d(z_k, \p \W)$ while $|z_{k+1} - z_k| \leq 4 K d(z_{k}, \p \W)$. We continue constructing such points until $d(z_k, \p \W) \geq d_2$; let the final point be $z_J$. Each pair $z_k, z_{k+1}$ lies inside $B_{CK\eta_k}(p_k)$, where $\eta_k = d(z_k, \p \W) $ and  $p_k \in \p \W$ is a point with $\eta_k = |z_k - p_k|$. They also have $d(z_k, \p \W), d(z_{k+1}, \p \W) \geq \eta_k$. Applying Lemma \ref{l:appendPDE} with $M = CK$ centered around $p_k$, we see that as long as $\eta_k < \eta_0$ for the $\eta_0 = \eta_0(K, m, \l)$ there, the points $z_k$ and $z_{k+1}$ may be connected by a curve $\g_k$ of length at most $C \eta_k$ staying at least $c \eta_k$ away from $\p \W$. Choose $\d$ small enough that 
	\[
		\eta_k \leq \eta_J \leq C d_2 \leq C \d \leq \eta_0
	\]
	and the assumption is verified for every $k$. Finally, apply Lemma \ref{l:appendPDE} one last time to connect $z_J$ to $x_2$ by a similar curve.
	
	We claim $\g$, the concatenation of all of these curves $\g_k$, works for the Harnack chain property. Indeed, we have that $\eta_k \leq \eta_J 2^{k - J}$ by construction, so the total length is controlled by $\sum \eta_k \leq 2 \eta_J \leq C d_2$, which is comparable to $|x_1 - x_2|$. By the same argument, the total length of the first $k$ concatenated curves $\g_k$ is bounded from above by $C \eta_k$, so for any point $z = \g(t)$ on $\g_k$, we have $d(z, \p \W) \geq c \eta_k \geq c l(\g([0, t]))$.

	This leaves only the case of $|x_1 - x_2| \geq 4K d_2$. Pick a point $x_3 \in \W$ with $d(x_3, \p \W) \geq |x_1 - x_2|$ and $|x_3 - y_2|\leq K |x_1 - x_2|$, using the clean inner ball property. Then $|x_3 - x_2| \leq 2K |x_1 - x_2|\leq 2K d(x_3, \p \W)$ and similarly $|x_3 - x_1|\leq 3K d(x_3, \p \W)$, so we may connect $x_3$ to $x_2$ and $x_1$ by curves satisfying the Harnack chain property. The concatenation of these curves then works to show that $\W$ satisfies the Harnack chain condition at $x_1$ and $x_2$.
\end{proof}

A direct application of this theorem proves Lemma \ref{l:localharnackchain}.

\begin{proof}[Proof of Lemma \ref{l:localharnackchain}]
	Apply Theorem \ref{t:NTA} to $u(x) = m_0 [w_\W(r_0 x) + \sqrt{\tpar} u_\W(r_0 x)]$ in normal coordinates around a point $x \in \p \W$. By choosing $r_0$ and $m_0$ small, we have property (4) from the regularity of the metric.  Property (2) follows from Corollary \ref{cor:lip}, while (3) comes from Theorem \ref{t:lb}. The inner and outer ball conditions were verified in Lemma \ref{l:densitybd}.
\end{proof}

\bibliographystyle{amsplain}
\bibliography{quantitativeACFref}

\end{document}